\documentclass[a4paper,11pt,leqno]{smfart}

\usepackage{amssymb,amsmath,enumerate,verbatim,mathrsfs,graphics,graphicx,mathptm,float,mathtools,xcolor,pgf,tikz,epsfig}
\usepackage[T1]{fontenc}
\usepackage[latin1]{inputenc}
\usepackage[english, francais]{babel}
\usepackage[all]{xy}
\usepackage[pdfstartview=FitH,
            colorlinks=true,
            linkcolor=blue,
            urlcolor=blue,
            citecolor=red,
            bookmarks,
            bookmarksopen=true,
            bookmarksnumbered=true]{hyperref}

\usepackage{cleveref}
\theoremstyle{plain}

\newtheorem{theoalph}{Theorem}

\newtheorem{thmalph}[theoalph]{Theorem}
\newtheorem{propalph}[theoalph]{Proposition}
\newtheorem{coralph}[theoalph]{Corollary}
\newtheorem{problem}{Problem}

\theoremstyle{definition}
\newtheorem{definalph}[theoalph]{Definition}

\theoremstyle{remark}

\theoremstyle{plain}
\newtheorem{thmsec}{Theorem}[section]
\newtheorem{thm}[thmsec]{Theorem}
\newtheorem{pro}[thmsec]{Proposition}
\newtheorem{lem}[thmsec]{Lemma}
\newtheorem{cor}[thmsec]{Corollary}

\theoremstyle{definition}
\newtheorem{defin}[thmsec]{Definition}

\theoremstyle{remark}
\newtheorem{rem}[thmsec]{Remark}
\newtheorem{rems}[thmsec]{Remarks}
\newtheorem{eg}[thmsec]{Example}

\def\og{\leavevmode\raise.3ex\hbox{$\scriptscriptstyle\langle\!\langle$~}}
\def\fg{\leavevmode\raise.3ex\hbox{~$\!\scriptscriptstyle\,\rangle\!\rangle$}}

\def\noteA#1#2#3{{\begin{small}$#1$\end{small}} & {\begin{small}#2\end{small}} \hfill{\begin{small}\pageref{#3}\end{small}} \\}

\addtocounter{section}{0}             
\numberwithin{equation}{section}       

\newcommand{\C}{\mathbb{C}}

\newcommand{\pp}{\mathbb{P}^{2}_{\mathbb{C}}}
\newcommand{\pd}{\mathbb{\check{P}}^{2}_{\mathbb{C}}}
\newcommand\diffx{\mathrm{d}x}
\newcommand\diffy{\mathrm{d}y}

\newcommand\diffp{\mathrm{d}p}
\newcommand\diffq{\mathrm{d}q}

\newcommand\Sing{\mathrm{Sing}}
\newcommand\Tang{\mathrm{Tang}}
\newcommand\Leg{\mathrm{Leg}}
\newcommand\IF{\mathrm{I}_{\mathcal{F}}}
\newcommand\IinvF{\mathrm{I}_{\mathcal{F}}^{\mathrm{inv}}}
\newcommand\ItrF{\mathrm{I}_{\mathcal{F}}^{\hspace{0.2mm}\mathrm{tr}}}

\newcommand\ItrH{\mathrm{I}_{\mathcal{H}}^{\hspace{0.2mm}\mathrm{tr}}}
\newcommand\Cinv{\mathrm{C}_{\hspace{-0.3mm}\mathcal{H}}}

\newcommand\radH{\Sigma_{\mathcal{H}}^{\mathrm{rad}}}
\newcommand\radHd{\check{\Sigma}_{\mathcal{H}}^{\mathrm{rad}}}
\newcommand\F{\mathcal{F}}
\newcommand\pref{\mathscr{F}}
\newcommand\preh{\mathscr{H}}
\newcommand\W{\mathcal{W}}
\newcommand\G{\mathcal{G}}
\newcommand\omegahesse{\omega_{\scalebox{0.64}{\ensuremath H}}^{4}}
\newcommand\Omegahesse{\Omega_{\scalebox{0.64}{\ensuremath H}}^{4}}
\newcommand\omegahilbertcinq{\omega_{\scalebox{0.64}{\ensuremath H}}^{5}}
\newcommand\omegahessesept{\omega_{\scalebox{0.64}{\ensuremath H}}^{7}}
\newcommand\Gunderline{{\mspace{2mu}\underline{\mspace{-2mu}\mathcal{G}\mspace{-2mu}}\mspace{2mu}}}

\newcommand\omegaoverline{{\mspace{2mu}\overline{\mspace{-1.4mu}\omega\mspace{-1.4mu}}\mspace{2mu}}}
\newcommand\Omegaoverline{{\mspace{2mu}\overline{\mspace{-1.4mu}\Omega\mspace{-1.4mu}}\mspace{2mu}}}

\newcommand\Hesse{\mathcal{F}_{\hspace{-0.4mm}\raisebox{-0.2mm}{\tiny{$H$}}}^{4}}
\newcommand\Hilbertcinq{\mathcal{F}_{\hspace{-0.4mm}\raisebox{-0.2mm}{\tiny{$H$}}}^{5}}
\newcommand\Hessesept{\mathcal{F}_{\hspace{-0.4mm}\raisebox{-0.2mm}{\tiny{$H$}}}^{7}}

\begin{document}
\selectlanguage{english}
\title[Pre-foliations of co-degree one on $\mathbb{P}^{2}_{\mathbb{C}}$]{Pre-foliations of co-degree one on $\mathbb{P}^{2}_{\mathbb{C}}$ with a flat Legendre transform}
\date{\today}

\author{Samir \textsc{Bedrouni}}

\address{Facult\'e de Math\'ematiques, USTHB, BP $32$, El-Alia, $16111$ Bab-Ezzouar, Alger, Alg\'erie}
\email{sbedrouni@usthb.dz}


\maketitle{}

\begin{abstract}
A holomorphic pre-foliation $\mathscr{F}=\ell\boxtimes\mathcal{F}$ of co-degree $1$ and degree $d$ on $\mathbb{P}^{2}_{\mathbb{C}}$ is the data of a line $\ell$ of $\mathbb{P}^{2}_{\mathbb{C}}$ and a holomorphic foliation $\mathcal{F}$ on $\mathbb{P }^{2}_{\mathbb{C}}$ of degree $d-1.$ We study pre-foliations of co-degree $1$ on $\mathbb{P}^{2}_{\mathbb{ C}}$ with a flat Legendre transform (dual web). After having established some general results on the flatness of the dual $d$-web of a homogeneous pre-foliation of co-degree $1$ and degree $d$, we describe some explicit examples and we show that up to automorphism of $\mathbb{P}^{2}_{\mathbb{C}}$ there are two families and six examples of homogeneous pre-foliations of co-degree $1$ and degree $3$ on $\mathbb {P}^{2}_{\mathbb{C}}$ with a flat dual web. This allows us to prove an analogue for pre-foliations of co-degree $1$ and degree~$3$ of a result, obtained in collaboration with D. Mar\'{\i}n, on foliations of degree $3$ with non-degenerate singularities and a flat Legendre transform. We also show that the dual web of a reduced convex pre-foliation of co-degree $1$ on $\mathbb{P}^{2}_{\mathbb{C}}$ is flat. This is an analogue of a result on foliations of $\mathbb{P}^{2}_{\mathbb{C}}$ due to D. Mar\'{\i}n and J. V. Pereira.

\end{abstract}

\section*{Introduction}
\bigskip

\noindent This article is a continuation of a series of joint works with D. \textsc{Mar\'{\i}n}~\cite{BM18Bull,BM20Bull,BM21Four,BM22arxiv} on holomorphic foliations on the complex projective plane. For the definitions and notations used (web, discriminant $\Delta(\W)$, homogeneous foliation, inflection divisor $\IF$, radial singularity, etc.) we refer to~\cite[Sections~1~and~2]{BM18Bull}.

\noindent We begin by introducing the following definition, where the terminology \og pre-foliation\fg\, is taken from~\cite{CCD13}.

\begin{definalph}\label{definalph:prefeuilletage}
Let $0\leq k\leq d$ be integers. A {\sl holomorphic pre-foliation $\mathscr{F}$ on $\pp$ of co-degree $k$ and degree~$d$}, or simply {\sl of type $(k,d)$}, is the data of a reduced complex projective curve
$\mathcal{C}\subset\pp$ of degree $k$ and a holomorphic foliation $\mathcal{F}$ on $\mathbb{P }^{2}_{\mathbb{C}}$ of degree $d-k.$ We write $\mathscr{F}=\mathcal{C}\boxtimes\F$ and call $\mathcal{C}$ (resp. $\F$) the
{\sl associated curve} (resp. the {\sl associated foliation}) to $\mathscr{F}.$
\end{definalph}

\noindent Such a pre-foliation is given in homogeneous coordinates $[x:y:z]\in\pp$ by a $1$-form of type $\Omega=F(x,y,z)\Omega_0,$ where $\C[x,y,z]_{k}\ni F(x,y,z)=0$ is a homogeneous equation of the curve  $\mathcal{C}$ and $\Omega_0$ is a homogeneous $1$-form of degree $d-k+1$ defining the foliation $\F,$ {\it i.e.} $$\Omega_0=a(x,y,z)\mathrm{d}x+b(x,y,z)\mathrm{d}y+c(x,y,z)\mathrm{d}z,$$ where $a,$ $b$ and $c$ are homogeneous polynomials of degree $d-k+1$ without common factor and satisfying the \textsc{Euler} condition $i_{\mathrm{R}}\omega=0$, where $\mathrm{R}=x\frac{\partial{}}{\partial{x}}+y\frac{\partial{}}{\partial{y}}+z\frac{\partial{}}{\partial{z}}$ denotes the radial vector field and  $i_{\mathrm{R}}$ is the interior product by $\mathrm{R}.$

\noindent We will denote the set of pre-foliations of type $(k,d)$ on $\pp$ by $\mathbf{F}(k,d).$ It can be naturally identified with a \textsc{Zariski} open subset of the space $\mathbb{P}_{\C}^{\hspace{0.2mm}N_1}\times\mathbb{P}_{\C}^{\hspace{0.2mm}N_{2}}$, where $N_1=\frac{k(k+3)}{2}$ and $N_2=(d-k+2)^2-2.$ The set $\mathbf{F}(0,d)$ describes precisely the set of foliations of degree $d$ on $\pp.$
\smallskip

\noindent By~\cite{MP13}, to every pre-foliation $\mathscr{F}=\mathcal{C}\boxtimes\F$ of degree $d\geq1$ and co-degree $k<d$ on $\pp$ we can associate a $d$-web of degree $1$ on the dual projective plane $\pd,$ called {\sl \textsc{Legendre} transform} (or {\sl dual web}) of $\pref$ and denoted by $\Leg\pref$\label{not:Leg-pref}; if $\pref$ is given in an affine chart $(x,y)$ of $\pp$ by a $1$-form $\omega=f(x,y)\left(A(x,y)\mathrm{d}x+B(x,y)\mathrm{d}y\right)$ then, in the affine chart $(p,q)$ of $\pd$ associated to the line $\{y=px-q\}\subset\pp,$ $\Leg\pref$ is described by the implicit differential equation
\[
F(p,q,x):=f(x,px-q)\left(A(x,px-q)+pB(x,px-q)\right)=0, \qquad \text{with} \qquad x=\frac{\mathrm{d}q}{\mathrm{d}p}.
\]
When $k\geq1$, $\Leg\pref$ decomposes as $\Leg\pref=\Leg\mathcal{C}\boxtimes\Leg\F,$ where $\Leg\mathcal{C}$ is the algebraic $k$-web on $\pd$ defined~by the equation $f(x,px-q)=0$ and $\Leg\F$ is the irreducible $(d-k)$-web of degree $1$ on $\pd$ given by~$A(x,px-q)+pB(x,px-q)=0.$

\noindent Conversely, every decomposable $d$-web of degree $1$ on $\pd$ is necessarily the \textsc{Legendre} transform of a certain pre-foliation on $\pp$ of type $(k,d)$, with $1\leq k<d.$

\noindent The curvature of a web $\W$ on $\pp$ is a meromorphic $2$-form with poles along the discriminant $\Delta(\W)$. A~web~with~zero curvature is called flat. The flatness of a web $\W$ on $\pp$ is characterized by the holomorphy of its curvature $K(\W)$ along the generic points of $\Delta(\W)$, \emph{see} \S\ref{sec:courbure-platitude}.

\noindent The subset $\mathbf{FP}(k,d)$ of $\mathbf{F}(k,d)$ consisting of $\pref\in\mathbf{F}(k,d)$ such that $\Leg\pref$ is flat is \textsc{Zariski} closed in $\mathbf{F}(k,d)$.

\begin{problem}\label{problem:1}
{\sl
Let $d$ and $k$ be integers such that $d\geq3$ and $0\leq k<d.$ Describe certain irreducible components of $\mathbf{FP}(k,d).$}
\end{problem}

\noindent This problem is inspired by a part of Problem~9.3 of \cite{MP13} which consists in the description of certain irreducible components of the space of flat $d$-webs of degree $1$ on $\pd.$ The first case $(k,d)=(0,3)$ has been completely treated in \cite{BM21Four}. In fact, the author and \textsc{Mar\'{\i}n} (\emph{see} \cite[Sections~3--6]{BM18Bull}) first studied the flatness of~dual~webs of homogeneous foliations of $\pp$, and they showed that it is possible to reduce the~study~of~the~flatness of dual webs of certain inhomogeneous foliations to the homogeneous framework. This~work~(\cite[Theorems~5.1~and~6.1]{BM18Bull}) then allowed to show that $\mathbf{FP}(0,3)$ has exactly $12$ irreducible components, \emph{see}~\cite[Theorem~D]{BM21Four}. In what follows we are interested in pre-foliations of co-degree~$1$, {\it i.e.} whose associated curve is a line. We will not look for irreducible components of $\mathbf{FP}(1,d)$, but will adapt the approach of \cite{BM18Bull} to co-degree one pre-foliations.

\begin{definalph}\label{definalph:prefeuilletage-homogene}
A pre-foliation on $\pp$ is said to be {\sl homogeneous} if there is an affine chart $(x,y)$ of $\pp$ in which it is invariant under the action of the group of homotheties $(x,y)\longmapsto \lambda(x,y),\hspace{1mm} \lambda\in \mathbb{C}^{*}.$
\end{definalph}

\noindent A homogeneous pre-foliation $\preh$ of type $(1,d)$ on $\pp$ is then of the form $\preh=\ell\boxtimes\mathcal{H},$ where $\mathcal{H}$ is a homogeneous foliation of degree $d-1$ on $\pp$ and where $\ell$ is a line passing through the origin $O$ or $\ell=L_\infty.$
\vspace{-2mm}

\noindent Theorem~3.1~of~\cite{BM18Bull} states that the web $\Leg\mathcal{H}$ is flat if and only if its curvature is holomorphic on the transverse part of its discriminant $\Delta(\Leg\mathcal{H}).$\label{not:Delta-W} We prove in Section~\S\ref{sec:etude-platitude-tissu-dual-pre-feuilletage-homogene} a similar result (Theorem~\ref{thm:holomorphie-courbure-G(I^tr)-D-ell}) for the web $\Leg\preh.$
\newpage
\hfill
\vspace{2mm}

\noindent When $\ell$ passes through the origin, we establish effective criteria for the holomorphy of the curvature of~$\Leg\preh$~on certain irreducible components of the discriminant $\Delta(\Leg\preh)$ (Theorems~\ref{thm:holomorphie-courbure-homogene-D-neq-D-ell}~and~\ref{thm:holomorphie-courbure-homogene-D-ell}). In~fact,~Theorems~\ref{thm:holomorphie-courbure-G(I^tr)-D-ell},~\ref{thm:holomorphie-courbure-homogene-D-neq-D-ell}~and~\ref{thm:holomorphie-courbure-homogene-D-ell} provide a complete characterization of the flatness of $\Leg\preh.$
\smallskip

\noindent When $\ell=L_\infty$ we show (Theorem~\ref{thm:K-Leg-H-egale-K-Leg-L-infini-H}) that the webs $\Leg\mathcal{H}$ and $\Leg\preh$ have the same curvature; in particular the flatness of $\Leg\preh$ is equivalent to that of $\Leg\mathcal{H}.$ More particularly, in degree $d=3$ the web~$\Leg\preh$ is flat (Corollary~\ref{cor:Leg-L-infini-H2-plat}).
\smallskip

\noindent Recall (\emph{see}~\cite{MP13}) that a holomorphic foliation on $\pp$ is said to be \textsl{convex} if its leaves other than straight lines have no inflection points. Note (\emph{see} \cite{Per01}) that if $\F$ is a foliation of degree $d\geq1$ on $\pp,$ then $\F$ cannot have more than $3d$ (distinct) invariant lines. Moreover, if this bound is reached, then $\F$ is necessarily convex; in~this~case $\F$ is said to be \textsl{reduced convex}. We naturally extend the notions of convexity and reduced convexity of foliations to pre-foliations by putting:
\begin{definalph}\label{definalph:prefeuilletage-convexe-convexe-reduit}
Let $\pref=\mathcal{C}\boxtimes\F$ be a pre-foliation on $\pp.$ We say that $\pref$ is \textsl{convex} (resp. \textsl{reduced convex}) if the foliation $\F$ is convex (resp. reduced convex) and if moreover the curve $\mathcal{C}$ is invariant by $\F.$
\end{definalph}

\noindent From this definition and Theorem~\ref{thm:holomorphie-courbure-G(I^tr)-D-ell} we will deduce the following corollary, which is an analogue of Corollary~3.4 of \cite{BM18Bull}.
\begin{coralph}[\rm{Corollary~\ref{cor:platitude-pre-feuilletage-homogene-convexe}}]\label{coralph:platitude-pre-feuilletage-homogene-convexe}
{\sl The dual web of a homogeneous convex pre-foliation of co-degree $1$ on~$\pp$ is flat.}
\end{coralph}

\noindent In~\S\ref{sec:application-homogene-deg-type-2} we give an application of the results of~\S\ref{sec:etude-platitude-tissu-dual-pre-feuilletage-homogene} to homogeneous pre-foliations $\preh=\ell\boxtimes\mathcal{H}$~of~co-degree~$1$ such that the degree of type of $\mathcal{H}$ is equal to $2,$ {\it i.e.} $\deg\mathcal{T}_{\mathcal{H}}=2$ (\emph{see} \cite[Definition~2.3]{BM18Bull} for the definitions of the type $\mathcal{T}_{\mathcal{H}}$\label{not:Type-H} and the degree of type $\deg\mathcal{T}_{\mathcal{H}}$). More precisely, we describe, up to automorphism of $\pp$, all homogeneous pre-foliations $\preh=\ell\boxtimes\mathcal{H}$ of co-degree $1$ and degree $d\geq3$ such that $\deg\mathcal{T}_{\mathcal{H}}=2$ and the $d$-web $\Leg\preh$ is flat (Proposition~\ref{pro:class-pre-homogenes-plats-co-degre-1-degre-d-geq-4-degre-Type-2}). We obtain in particular, for $d=3$, the classification up to automorphism of homogeneous pre-foliations of type $(1,3)$ on $\pp$ whose dual $3$-web is flat: {\sl up to automorphism of $\pp,$ there are two families and six examples of homogeneous pre-foliations of co-degree $1$ and degree $3$ on $\mathbb {P}^{2}_{\mathbb{C}}$ with a flat \textsc{Legendre} transform}, \emph{see} Corollary~\ref{cor:class-pre-homogenes-plats-co-degre-1-degre-3}.
\medskip

\noindent In $2013$ \textsc{Mar\'{\i}n} and \textsc{Pereira} \cite[Theorem~4.2]{MP13} proved that the dual web of a reduced convex foliation on~$\pp$ is flat. We show in~\S\ref{sec:pre-pre-feuill-convexe-reduit} the following analogous result for co-degree one pre-foliations.
\begin{thmalph}\label{thmalph:ell-invariante-convexe-reduit-plat}
{\sl Let $\pref=\ell\boxtimes\F$ be a reduced convex pre-foliation of co-degree $1$ and degree $d\geq3$ on $\pp.$ Then the $d$-web $\Leg\pref$ is flat.}
\end{thmalph}

\noindent The following problem then arises.
\begin{problem}\label{problem:2}
{\sl
Let $\F$ be a reduced convex foliation of degree greater than or equal to $2$ on $\pp$ and let $\ell$ be a line of $\pp$ which is not invariant by $\F.$ Determine the relative position of the line $ \ell$ with respect to the invariant lines of $\F$ such that the dual web of the pre-foliation $\ell\boxtimes\F$ is flat.
}
\end{problem}

\noindent To our knowledge the only reduced convex foliations known in the literature are those presented in~\cite[Table~1.1]{MP13}: the \textsc{Fermat} foliation $\F_{0}^{d-1}$ of degree $d-1$, the \textsc{Hesse} foliation $\Hesse$ of degree~$4$, the~\textsc{Hilbert}~modular~foliation $\Hilbertcinq$ of degree~$5$ and the \textsc{Hesse} foliation $\Hessesept$ of degree $7$ defined in affine chart respectively~by~the~$1$-forms
\begin{align*}
&\hspace{4cm}\omegaoverline_{0}^{d-1}=x\diffy-y\diffx+y^{d-1}\diffx-x^{d-1}\diffy,\\
&\hspace{3.1cm}\omegahesse=y(2x^{3}-y^{3}-1)\mathrm{d}x+x(2y^{3}-x^{3}-1)\mathrm{d}y,\\
&\omegahilbertcinq=(y^2-1)(y^2-(\sqrt{5}-2)^2)(y+\sqrt{5}x)\mathrm{d}x-(x^2-1)(x^2-(\sqrt{5}-2)^2)(x+\sqrt{5}y)\mathrm{d}y,\\
&\hspace{1.83cm}\omegahessesept=(y^3-1)(y^3+7x^3+1)y\mathrm{d}x-(x^3-1)(x^3+7y^3+1)x\mathrm{d}y.
\end{align*}

\noindent The following two propositions, which will be proved in~\S\ref{sec:pre-pre-feuill-convexe-reduit}, give an answer to Problem~\ref{problem:2} in the case of the \textsc{Fermat} foliation $\F_{0}^{d-1}$ and the \textsc{Hesse} foliation $\Hesse.$
\begin{propalph}\label{proalph:ell-non-invariante-Fermat}
{\sl Let $d\geq3$ be an integer and let $\ell$ be a line of $\pp.$ Assume that $\ell$ is not invariant by the \textsc{Fermat} foliation $\F_{0}^{d-1}$ and that the $d$-web $\Leg(\ell\boxtimes\F_{0}^{d-1})$ is flat. Then $d\in\{3,4\}$ and the line $\ell$ joins two (resp. three) singularities (necessarily non-radial) of $\F_{0}^{d-1}$ if $d=3$ (resp. if $d=4$).
}
\end{propalph}

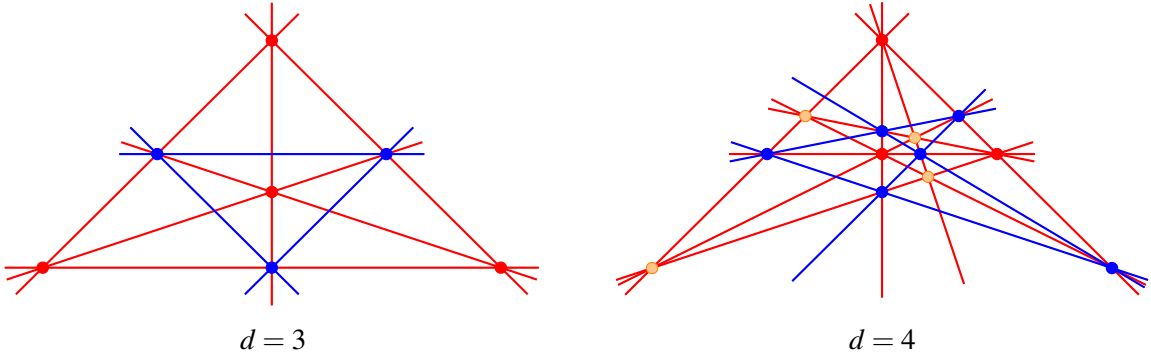
\vspace{2mm}
\begin{figure}[h!]
\centering
\begin{tikzpicture}
\begin{scope}[x=0.71cm,y=0.71cm,shift={(-5.625,4.25)},scale=1,rotate=-135]
\draw[color=red,fill=red!100,line width=0.8pt,shorten >=-0.5cm,shorten <=-0.5cm](0,0) -- (0,3)--(0,6);
\draw[color=red,fill=red!100,line width=0.8pt,shorten >=-0.5cm,shorten <=-0.5cm](0,0) -- (3,0)--(6,0);
\draw[color=red,fill=red!100,line width=0.8pt,shorten >=-0.5cm,shorten <=-0.5cm](0,0) -- (2,2)--(3,3);
\draw[color=red,fill=red!100,line width=0.8pt,shorten >=-0.5cm,shorten <=-0.5cm](0,6) --(3,3)-- (6,0);
\draw[color=red,fill=red!100,line width=0.8pt,shorten >=-0.5cm,shorten <=-0.5cm](0,3) -- (2,2)--(6,0);
\draw[color=red,fill=red!100,line width=0.8pt,shorten >=-0.5cm,shorten <=-0.5cm](3,0) -- (2,2)--(0,6);
\draw[color=blue,fill=blue!100,line width=0.8pt,shorten >=-0.5cm,shorten <=-0.5cm](0,3) --(3,0);
\draw[color=blue,fill=blue!100,line width=0.8pt,shorten >=-0.5cm,shorten <=-0.5cm](0,3) --(3,3);
\draw[color=blue,fill=blue!100,line width=0.8pt,shorten >=-0.5cm,shorten <=-0.5cm](3,0) --(3,3);
\draw[color=red,fill=red!100]   (0,0) circle (0.075cm);
\draw[color=red,fill=red!100]   (2,2) circle (0.075cm);
\draw[color=red,fill=red!100]   (0,6) circle (0.075cm);
\draw[color=red,fill=red!100]   (6,0) circle (0.075cm);
\draw[color=blue,fill=blue!100] (0,3) circle (0.075cm);
\draw[color=blue,fill=blue!100] (3,0) circle (0.075cm);
\draw[color=blue,fill=blue!100] (3,3) circle (0.075cm);
\draw(4.5,3.4)node[right]{$d=3$};
\end{scope}
\begin{scope}[x=1.26cm,y=0.63cm,shift={(3.2,2.4)},scale=0.3,rotate=-180]
\draw[color=red,fill=red!100,line width=0.8pt,shorten >=-1.4cm,shorten <=-0.5cm](0,-8)--(0,8/3);
\draw[color=red,fill=red!100,line width=0.8pt,shorten >=-0.5cm,shorten <=-0.5cm](-4,0) -- (4,0);
\draw[color=red,fill=red!100,line width=0.8pt,shorten >=-0.5cm,shorten <=-0.5cm](0,-8) --(-8,8);
\draw[color=red,fill=red!100,line width=0.8pt,shorten >=-0.5cm,shorten <=-0.5cm](-8/3,-8/3) --(8,8);
\draw[color=red,fill=red!100,line width=0.8pt,shorten >=-1.02cm,shorten <=-0.5cm](0,-8) --(-2,4);
\draw[color=red,fill=red!100,line width=0.8pt,shorten >=-0.5cm,shorten <=-0.5cm](-4,0) --(8,8);
\draw[color=red,fill=red!100,line width=0.8pt,shorten >=-0.5cm,shorten <=-0.5cm](-4,0) --(8/3,-8/3);
\draw[color=red,fill=red!100,line width=0.8pt,shorten >=-0.5cm,shorten <=-0.5cm](0,-8) --(8,8);
\draw[color=red,fill=red!100,line width=0.8pt,shorten >=-0.5cm,shorten <=-0.5cm](-8,8) --(8/3,-8/3);
\draw[color=blue,fill=blue!100,line width=0.8pt,shorten >=-0.5cm,shorten <=-0.5cm](-8/3,-8/3) --(4,0);
\draw[color=blue,fill=blue!100,line width=0.8pt,shorten >=-1.67cm,shorten <=-0.5cm](-8/3,-8/3) --(0,8/3);
\draw[color=blue,fill=blue!100,line width=0.8pt,shorten >=-0.5cm,shorten <=-0.5cm](-8,8) --(4,0);
\draw[color=blue,fill=blue!100,line width=0.8pt,shorten >=-0.5cm,shorten <=-0.5cm](-8,8) --(2,-4);
\draw[color=red,fill=red!100] (0,0) circle (0.25cm);
\draw[color=red,fill=red!100] (0,-8) circle (0.25cm);
\draw[color=red,fill=red!100] (-4,0) circle (0.25cm);
\draw[color=orange,fill=orange!45] (-8/7,-8/7) circle (0.25cm);
\draw[color=orange,fill=orange!45] (-8/5,8/5) circle (0.25cm);
\draw[color=orange,fill=orange!45] (8/3,-8/3) circle (0.25cm);
\draw[color=orange,fill=orange!45] (8,8) circle (0.25cm);
\draw[color=blue,fill=blue!100] (-8/3,-8/3) circle (0.25cm);
\draw[color=blue,fill=blue!100] (0,-8/5) circle (0.25cm);
\draw[color=blue,fill=blue!100] (0,8/3) circle (0.25cm);
\draw[color=blue,fill=blue!100] (-4/3,0) circle (0.25cm);
\draw[color=blue,fill=blue!100] (4,0) circle (0.25cm);
\draw[color=blue,fill=blue!100] (-8,8) circle (0.25cm);
\draw(1.5,13)node[right]{$d=4$};
\end{scope}
\end{tikzpicture}
\vspace{0.46cm}

\begin{minipage}[c]{0.84\linewidth}
{\normalfont\Small Relative positions of the line $\ell$ (in blue) with respect to the invariant lines (in red) of the \textsc{Fermat} foliations $\mathcal{F}_{0}^{2}$ and $\mathcal{ F}_{0}^{3}.$ The foliation $\mathcal{F}_{0}^{2}$ ($d=3$) has $4$ radial singularities (red points) and $3$ non-radial singularities (blue~points) with \textsc{Baum}-\textsc{Bott} invariant $0.$ The foliation $\mathcal{F}_{0}^{3}$ ($d=4$) admits $7$ radial singularities, $4$ of order one (orange points) and $3$ of order two (red points), and $6$ non-radial singularities with \textsc{Baum}-\textsc{Bott} invariant~$-\frac{1}{2}.$}
\end{minipage}
\end{figure}

\begin{propalph}\label{proalph:ell-non-invariante-Hesse}
{\sl
Let $\ell$ be a line of $\pp$ which is not invariant by the \textsc{Hesse} foliation $\Hesse.$ Assume that the $5$-web $\Leg(\ell\boxtimes\Hesse)$ is flat. Then the line $\ell$ passes through four (necessarily non-radial) singularities of $\Hesse$.
}
\end{propalph}

\noindent The idea of the proofs of~Propositions~\ref{proalph:ell-non-invariante-Fermat}~and~\ref{proalph:ell-non-invariante-Hesse} will be to reduce to the homogeneous case, by showing that the closures of the $\mathrm{Aut}(\pp)$-orbits of the pre-foliations $\ell\boxtimes\F_{0}^{d-1}$ and $\ell\boxtimes\Hesse$  contain homogeneous pre-foliations.
\smallskip

\noindent Theorem~6.1~of~\cite{BM18Bull} says that every foliation of degree $3$ on $\pp$ with non-degenerate singularities and a flat \textsc{Legendre} transform is linearly conjugate to the \textsc{Fermat} foliation $\F_{0}^{3}.$ We prove in~\S\ref{sec:pre-feuilletages-codegre-1-degre-3} the following similar result for pre-foliations of co-degree $1$ and degree $3.$
\begin{thmalph}\label{thmalph:Fermat2}
{\sl Let $\pref=\ell\boxtimes\F$ be a pre-foliation of co-degree $1$ and degree $3$ on $\pp.$ Assume that the foliation $\F$ has only non-degenerate singularities and that the $3$-web $\Leg\pref$ is flat. Then $\F$ is linearly conjugate to the \textsc{Fermat} foliation $\F_{0}^{2},$ and the line $\ell$ is either invariant by $\F$ or it joins two non-radial singularities of~$\F.$
}
\end{thmalph}

\noindent The proof of this theorem will essentially use the classification of homogeneous pre-foliations of type $(1,3)$ on $\pp$ whose dual web is flat (Corollary~\ref{cor:class-pre-homogenes-plats-co-degre-1-degre-3}).

\newpage
\hfill

\section{Reminders on the fundamental form and curvature of a web}\label{sec:courbure-platitude}
\bigskip

\noindent In this section, we briefly recall the definitions of the fundamental form and the curvature of a $d$-web $\W.$ Let us first assume that $\W$ is a germ of completely decomposable $d$-web on $(\C^2,0)$, $\mathcal{W}=\mathcal{F}_{1}\boxtimes\cdots\boxtimes\mathcal{F}_{d}.$ For~$i=1,\ldots,d,$~let~$\omega_{i}$~be~a~$1$-form with an isolated singularity at $0$ defining the foliation~$\mathcal{F}_{i}.$ Following~\cite{PP08}, for~each~triple~$(r,s,t)$ with $1\leq r<s<t\leq d,$ one defines $\eta_{rst}=\eta(\mathcal{F}_{r}\boxtimes\mathcal{F}_{s}\boxtimes\mathcal{F}_{t})$ as the unique meromorphic $1$-form satisfying the following equalities:
\begin{equation}\label{equa:eta-rst}
{\left\{\begin{array}[c]{lll}
\mathrm{d}(\delta_{st}\,\omega_{r}) &=& \eta_{rst}\wedge\delta_{st}\,\omega_{r}\\
\mathrm{d}(\delta_{tr}\,\omega_{s}) &=& \eta_{rst}\wedge\delta_{tr}\,\omega_{s}\\
\mathrm{d}(\delta_{rs}\,\omega_{t}) &=& \eta_{rst}\wedge\delta_{rs}\,\omega_{t}
\end{array}
\right.}
\end{equation}
where $\delta_{ij}$ denotes the function defined by $\omega_{i}\wedge\omega_{j}=\delta_{ij}\,\mathrm{d}x\wedge\mathrm{d}y.$ One calls {\sl fundamental form} of the web $\mathcal{W}=\mathcal{F}_{1}\boxtimes\cdots\boxtimes\mathcal{F}_{d}$ the $1$-form\label{not:eta-W}
\begin{equation}\label{equa:eta}
\hspace{7mm}\eta(\mathcal{W})=\eta(\mathcal{F}_{1}\boxtimes\cdots\boxtimes\mathcal{F}_{d})=\sum_{1\le r<s<t\le d}\eta_{rst}.
\end{equation}
One can easily check that $\eta(\mathcal{W})$ is a meromorphic $1$-form with poles along the discriminant $\Delta(\mathcal{W})$ of~$\mathcal{W},$ and~that it is well-defined up to addition of a closed logarithmic $1$-form $\dfrac{\mathrm{d}f}{f}$ with $f\in\mathcal{O}^*(\mathbb{C}^{2},0)$ (\emph{cf.}~\cite{Rip05,BM18Bull}).
\smallskip

\noindent Now, if $\W$ is a (not necessarily completely decomposable) $d$-web on a complex surface $M$ then its pull-back by a suitable ramified \textsc{Galois} covering is completely decomposable. The invariance of the fundamental form of this new web by the action of the \textsc{Galois} group allows us to descend it to a global meromorphic~$1$-form $\eta(\mathcal{W})$ on $M$, with poles along the discriminant of $\W$ (\emph{see} \cite{MP13}).
\smallskip

\noindent The {\sl curvature} of the web $\mathcal{W}$\label{not:K-W} is by definition the $2$-form
\begin{align*}
&K(\mathcal{W})=\mathrm{d}\,\eta(\mathcal{W}).
\end{align*}
It is a meromorphic $2$-form with poles along the discriminant $\Delta(\mathcal{W})$ of $\mathcal{W},$ canonically associated to $\mathcal{W}$. More precisely, for any dominant holomorphic map $\varphi,$ one has $K(\varphi^{*}\mathcal{W})=\varphi^{*}K(\mathcal{W}).$
\smallskip

\noindent A $d$-web $\mathcal{W}$ is said to be {\sl flat} if its curvature $K(\mathcal{W})$ vanishes identically.
\smallskip

\noindent Let us finally note that a $d$-web $\mathcal{W}$ on $\pp$ is flat if and only if its curvature is holomorphic over
the generic points of the irreducible components of $\Delta(\mathcal{W}).$ This follows from the holomorphy of $K(\W)$ on $\pp\setminus\Delta(\mathcal{W})$ and from the fact that there are no holomorphic $2$-forms on $\pp$ other than the zero $2$-form.


\section{Discriminant of the dual web of a co-degree one pre-foliation}
\bigskip

\noindent Recall that if $\F$ is a foliation on $\pp$, the \textsc{Gauss} map of $\F$ is the rational map $\mathcal{G}_{\F}\hspace{1mm}\colon\pp\dashrightarrow \pd$ defined at every regular point $m$ of $\F$ by $\mathcal{G}_{\F}(m)=\mathrm{T}^{\mathbb{P}}_{m}\F,$\label{not:Gauss-F} where $\mathrm{T}^{\mathbb{P}}_{m}\F$ denotes the tangent line to the leaf of $\F$ passing through $m.$ If $\mathcal{C}$ is a curve on $\pp$ passing through some singular points of~$\F$, one defines $\mathcal{G}_{\mathcal{F}}(\mathcal{C})$ as the closure of $\mathcal{G}_{\F}(\mathcal{C}\setminus\Sing\F).$\label{not:Sing-F}
\begin{lem}\label{lem:Delta-Leg-ell-F}
{\sl
Let $\mathscr{F}=\ell\boxtimes\F$ be a pre-foliation of co-degree $1$ on $\pp.$

\noindent\textbf{\textit{1.}} If the line $\ell$ is invariant by $\F$, then
\begin{align*}
\Delta(\Leg\mathscr{F})=\Delta(\Leg\F)\cup\check{\Sigma}_{\F}^{\ell},
\end{align*}
where $\check{\Sigma}_{\F}^{\ell}$ denotes the set of lines dual to the points of $\Sigma_{\F}^{\ell}:=\Sing\F\cap\ell.$

\noindent\textbf{\textit{2.}} If the line $\ell$ is not invariant by $\F$, then
\begin{align*}
\Delta(\Leg\mathscr{F})=\Delta(\Leg\F)\cup\mathcal{G}_{\F}(\ell)\cup\check{\Sigma}_{\F}^{\ell}.
\end{align*}
}
\end{lem}

\begin{proof}[\sl Proof]
We have
\begin{align*}
\Delta(\Leg\mathscr{F})=\Delta(\Leg\F)\cup\Tang(\Leg\ell,\Leg\F).
\end{align*}
When $\ell$ is not invariant by $\F,$ we obtain by an argument of~\cite[page~33]{Bel14} that $$\Tang(\Leg\ell,\Leg\F)=\mathcal{G}_{\F}(\ell)\cup\check{\Sigma}_{\F}^{\ell}.$$

\noindent Let us assume that $\ell$ is invariant by $\F$ and show that $\Tang(\Leg\ell,\Leg\F)=\check{\Sigma}_{\F}^{\ell}.$ Let $s\in\Sigma_{\F}^{\ell}.$ The fact that $s\in\ell$ (resp.~$s\in\Sing\F$) implies that the line $\check{s}$ dual to $s$ is invariant by $\Leg\ell$ (resp. by $\Leg\F$). Thus~$\check{s}\subset\Tang(\Leg\ell,\Leg\F)$, hence $\check{\Sigma}_{\F}^{\ell}\subset\Tang(\Leg\ell,\Leg\F).$ Conversely, let $\mathcal{C}$ be an irreducible component of~$\Tang(\Leg\ell,\Leg\F).$ Let us show that $\mathcal{C}$ is invariant by $\Leg\F.$ Assume by contradiction that $\mathcal{C}$ is transverse to $\Leg\F.$ Let $m$ be a generic point of $\mathcal{C}.$ Denote by $\check{\ell}\in\pd$ the dual point of $\ell$; then the line $(\check{\ell}m)$ is not invariant by $\Leg\F$ and is tangent to $\Leg\F$ at $m.$ Since $\ell$ is $\F$-invariant, the point $\check{\ell}$ is singular for $\Leg\F$; it~is~therefore also a tangency point between $\Leg\F$ and $(\check{\ell}m).$ The number of tangency points between $\Leg\F$ and $(\check{\ell}m)$ is then $\geq2$, which contradicts the equality $\deg(\Leg\F)=1.$ Hence the invariance of $\mathcal{C}$ by $\Leg\F$ is proved. Then $\mathcal{C}$ is also invariant by $\Leg\ell$ and is therefore a line passing through $\check{\ell}.$ There therefore exists $s\in\Sing\F$ such that $\mathcal{C}=\check{s}$; since $\check{\ell}\in\mathcal{C}$, we have $s\in \ell$ and therefore $s\in\Sigma_{\F}^{\ell}.$ Consequently, $\mathcal{C}\subset\check{\Sigma}_{\F}^{\ell}.$
\end{proof}
\smallskip

\noindent We will now apply the above lemma to the case of a homogeneous pre-foliation $\preh=\ell\boxtimes\mathcal{H}$ of co-degree~$1$ on~$\pp.$ If $\deg\preh=d$, the homogeneous foliation $\mathcal{H}$ is given, for a suitable choice of affine coordinates $(x,y),$ by a homogeneous $1$-form
\[
\hspace{1mm}\omega=A(x,y)\mathrm{d}x+B(x,y)\mathrm{d}y,\quad \text{where}\hspace{1.5mm}A,B\in\mathbb{C}[x,y]_{d-1}\hspace{1.5mm}\text{with}\hspace{1.5mm}\gcd(A,B)=1.
\]
If $\ell=L_\infty$ then $\ell$ is invariant by $\mathcal{H}$ and Lemma~\ref{lem:Delta-Leg-ell-F} ensures that
\[
\Delta(\Leg\preh)=\Delta(\Leg\mathcal{H})\cup\check{\Sigma}_{\mathcal{H}}^{\infty},
\]
where $\check{\Sigma}_{\mathcal{H}}^{\infty}$ denotes the set of lines dual to the points of $\Sigma_{\mathcal{H}}^{\infty}:=\Sing\mathcal{H}\cap L_\infty.$
\smallskip

\noindent Assume that $\ell$ passes through the origin. If $\ell$ is not invariant by $\mathcal{H}$, then, according to \cite[Proposition~2.2]{BM18Bull}, we have $\Sigma_{\mathcal{H}}^{\ell}=\{O\}.$ Since the line $\check{O}$ dual to $O$ is contained in $\Delta(\Leg\mathcal{H})$ by \cite[Lemma~3.2]{BM18Bull}, it follows from Lemma~\ref{lem:Delta-Leg-ell-F} that
\[
\Delta(\Leg\preh)=\Delta(\Leg\mathcal{H})\cup\G_{\mathcal{H}}(\ell).
\]

\noindent If $\ell$ is invariant by $\mathcal{H}$, then the point $s:=L_\infty\cap\ell$ is singular for $\mathcal{H}$ and, by \cite[Proposition~2.2]{BM18Bull}, we have $\Sigma_{\mathcal{H}}^{\ell}=\{O,s\}$. Denoting by $\check{s}$ the dual line of the point $s$, the inclusion $\check{O}\subset\Delta(\Leg\mathcal{H})$ and Lemma~\ref{lem:Delta-Leg-ell-F} imply~that
\[
\Delta(\Leg\preh)=\Delta(\Leg\mathcal{H})\cup\check{s}.
\]

\noindent According to \cite[Lemma~3.2]{BM18Bull}, the discriminant of $\Leg\mathcal{H}$ decomposes as
\[
\Delta(\Leg\mathcal{H})=\G_{\mathcal{H}}(\ItrH)\cup\radHd\cup\check{O},
\]
where $\ItrH$\label{not:Inflex-Transverse-F} denotes the transverse inflection divisor of $\mathcal{H}$ and $\radHd$ is the set of lines dual to the radial singularities of $\mathcal{H}$ (\emph{see}~\cite[\S1.3]{BM18Bull} for precise definitions of these notions). Recall however that to the homogeneous foliation $\mathcal{H}$ one can also associate the rational map $\Gunderline_{\mathcal{H}}\hspace{1mm}\colon\mathbb{P}^{1}_{\mathbb{C}}\rightarrow \mathbb{P}^{1}_{\mathbb{C}}$ defined by
\[
\Gunderline_{\mathcal{H}}([y:x])=[-A(x,y):B(x,y)],
\]
and that this map allows us to completely determine the divisor $\ItrH$ and the set $\radH$ (\emph{see}~\cite[Section~2]{BM18Bull}):
\begin{itemize}
\item [$\bullet$] $\radH$ consists of $[b:a:0]\in L_{\infty}$ such that $[a:b]\in\mathbb{P}^{1}_{\mathbb{C}}$ is a fixed critical point of $\Gunderline_{\mathcal{H}}$;

\item [$\bullet$] $\ItrH=\prod\limits_{i}T_{i}^{n_i}$, where $T_{i}=(b_i\hspace{0.2mm}y-a_i\hspace{0.2mm}x=0)$ and $[a_i:b_i]\in\mathbb{P}^{1}_{\mathbb{C}}$ is a non-fixed critical point of $\Gunderline_{\mathcal{H}}$ of multiplicity~$n_i.$
\end{itemize}
\vspace{2mm}

\noindent From the above considerations, we deduce the following lemma.
\begin{lem}\label{lem:Delta-Leg-H}
{\sl
Let $\preh=\ell\boxtimes\mathcal{H}$ be a homogeneous pre-foliation of co-degree $1$ on $\pp.$

\noindent\textbf{\textit{1.}} If $\ell=L_\infty$ then
\begin{align*}
\Delta(\Leg\preh)=\Delta(\Leg\mathcal{H})\cup\check{\Sigma}_{\mathcal{H}}^{\infty}=\G_{\mathcal{H}}(\ItrH)\cup\check{\Sigma}_{\mathcal{H}}^{\infty}\cup\check{O}.
\end{align*}

\noindent\textbf{\textit{2.}} If the line $\ell$ passes through the origin, then
\[
\Delta(\Leg\preh)=\Delta(\Leg\mathcal{H})\cup D_\ell=\G_{\mathcal{H}}(\ItrH)\cup\radHd\cup\check{O}\cup D_\ell,
\]
where the component $D_\ell$ is defined as follows. If $\ell$ is invariant by $\mathcal{H}$, then $D_\ell:=\check{s}$ is the dual line of the point $s=L_\infty\cap\ell\in\Sing\mathcal{H}.$ If $\ell$ is not invariant by $\mathcal{H}$, then $D_\ell:=\G_{\mathcal{H}}(\ell).$
}
\end{lem}


\section{Flatness of the dual web of a co-degree one homogeneous pre-foliation}\label{sec:etude-platitude-tissu-dual-pre-feuilletage-homogene}
\bigskip

\noindent Our first result shows that, for a homogeneous foliation $\mathcal{H}$ on $\pp,$ the webs $\Leg\mathcal{H}$ and $\Leg(L_\infty\boxtimes\mathcal{H})$ have~the~same curvature, so that we have equivalence between the flatness of $\Leg\mathcal{H}$ and that of $\Leg(L_\infty\boxtimes\mathcal{H}).$
\begin{thm}\label{thm:K-Leg-H-egale-K-Leg-L-infini-H}
{\sl
Let $d\geq3$ be an integer and let $\mathcal{H}$ be a homogeneous foliation of degree $d-1$ on $\pp.$ Then $$K(\Leg(L_\infty\boxtimes\mathcal{H}))=K(\Leg\mathcal{H}).$$ In particular, the $d$-web $\Leg(L_\infty\boxtimes\mathcal{H})$ is flat if and only if the $(d-1)$-web $\Leg\mathcal{H}$ is flat.
}
\end{thm}

\begin{cor}\label{cor:Leg-L-infini-H2-plat}
{\sl Let $\mathcal{H}$ be a homogeneous foliation of degree $2$ on $\pp.$ Then the $3$-web $\Leg(L_\infty\boxtimes\mathcal{H})$~is~flat.}
\end{cor}

\noindent To establish Theorem~\ref{thm:K-Leg-H-egale-K-Leg-L-infini-H}, we will need the following definition and theorem.
\begin{defin}[\cite{MPP06}]
Let $\W=\F_1\boxtimes\cdots\boxtimes\F_d$ be a regular $d$-web on~$(\C^2,0).$ A {\sl transverse symmetry} of~$\W$ is a germ of vector field $\mathrm{X}$ which is transverse to the foliations $\F_i$ ($i=1,\ldots,d$) and whose local flow $\exp(t\mathrm{X})$ preserves the $\F_i$'s.
\end{defin}

\begin{thm}\label{thm:K-FX-W-egale-K-W}
{\sl Let $d\geq3$ be an integer and let $\W_{d-1}$ be a regular $(d-1)$-web on $(\C^2,0)$ which admits a~transverse~symmetry~$\mathrm{X}.$ Denote by $\F_\mathrm{X}$ the foliation defined by $\mathrm{X}.$ Then $$K(\F_\mathrm{X}\boxtimes\W_{d-1})=K(\W_{d-1}).$$ In particular, the $d$-web $\F_\mathrm{X}\boxtimes\W_{d-1}$ is flat if and only if the $(d-1)$-web $\W_{d-1}$ is flat.
}
\end{thm}

\noindent Before proving this theorem, let us briefly recall the definition of the rank $\mathrm{rk}(\W)$ of a regular $d$-web $\W=\F_1\boxtimes\cdots\boxtimes\F_d$ on $(\C^2,0).$ For $1\leq i\leq d,$ let $\omega_i$ be a $1$-form defining the foliation $\F_i.$ One defines the $\C$-vector space $\mathcal{A}(\W)$ of {\sl abelian relations} of $\W$ by
\begin{align*}
\mathcal{A}(\W):=\Big\{(\eta_1,\ldots,\eta_d)\in(\Omega^1(\C^2,0))^d
\hspace{1mm}\Big|\hspace{1mm}
\forall i=1,\ldots,d,\hspace{1mm}
\mathrm{d}\eta_i=0,\hspace{1mm}
\eta_i\wedge\omega_i=0
\hspace{1mm}\text{ and }\hspace{1mm}
\sum_{i=1}^d\eta_i=0\Big\}.
\end{align*}
Then $\mathrm{rk}(\W):=\dim_{\C}\mathcal{A}(\W)$.\label{not:rk-W} One has the following optimal bound (\emph{cf.} \cite[Chapter~2]{PP15}):
\[
\mathrm{rk}(\W)\leq\pi_d:=\frac{(d-1)(d-2)}{2}.
\]
Recall also that every $d$-web of maximal rank ({\it i.e.} of rank $\pi_d$) is necessarily flat by \textsc{Mih\u{a}ileanu}'s criterion (\emph{cf.}~\cite[Theorem~6.3.4]{PP15}).

\noindent In fact, Theorem~\ref{thm:K-FX-W-egale-K-W} is an analogue for flat webs of a result on webs of maximal rank, due to~\textsc{Mar\'{\i}n}-\textsc{Pereira}-\textsc{Pirio}, namely:
\begin{thm}[\cite{MPP06}, \rm{Theorem~1}]\label{thm:Marin-Pereira-Pirio}
{\sl With the notations of Theorem~\ref{thm:K-FX-W-egale-K-W}, one has $$\mathrm{rk}(\F_\mathrm{X}\boxtimes\W_{d-1})=\mathrm{rk}(\W_{d-1})+(d-2).$$ In particular, $\F_\mathrm{X}\boxtimes\W_{d-1}$ is of maximal rank if and only if $\W_{d-1}$ is of maximal rank.
}
\end{thm}

\noindent The proof of Theorem~\ref{thm:K-FX-W-egale-K-W} consists essentially in applying this result for $d=3.$
\begin{proof}[\sl Proof of Theorem~\ref{thm:K-FX-W-egale-K-W}]
Writing $\W_{d-1}=\F_1\boxtimes\cdots\boxtimes\F_{d-1}$, we have
\begin{align*}
K(\F_\mathrm{X}\boxtimes\W_{d-1})=K(\W_{d-1})+\sum\limits_{1\leq i<j\leq d-1}K(\W_{3}^{i,j}),
\end{align*}
where $\W_{3}^{i,j}:=\F_\mathrm{X}\boxtimes\F_i\boxtimes\F_j.$ Moreover, since $\mathrm{X}$ is a transverse symmetry of the $2$-web $\F_i\boxtimes\F_j$ and since every $2$-web is of maximal rank, equal to~$0$, Theorem~1~of~\cite{MPP06} (\emph{cf.}~Theorem~\ref{thm:Marin-Pereira-Pirio} above) implies that the $3$-web $\W_{3}^{i,j}$ is of maximal rank, equal to $1,$ so that $K(\W_{3}^{i,j})\equiv0$, hence the announced equality holds.
\end{proof}

\begin{proof}[\sl Proof of Theorem~\ref{thm:K-Leg-H-egale-K-Leg-L-infini-H}]
By~\cite[Section~2]{BM18Bull}, we can locally decompose the $d$-web $\Leg(L_\infty\boxtimes\mathcal{H})$ as $$\Leg(L_\infty\boxtimes\mathcal{H})=\Leg(L_\infty)\boxtimes\W_{d-1},$$ where $\W_{d-1}=\F_1\boxtimes\cdots\boxtimes\F_{d-1}$ and, for any $i\in\{1,\ldots,d-1\},$ $\F_i$ is given by $\check{\omega}_i:=\lambda_{i}(p)\mathrm{d}q-q\mathrm{d}p$, with $\lambda_{i}(p)=p-p_{i}(p)$ and $\{p_{i}(p)\}=\Gunderline_{\mathcal{H}}^{-1}(p).$ Now, the vector field $\mathrm{X}:=q\frac{\partial }{\partial q}$ defines the radial foliation $\Leg(L_\infty)$ and is a transverse symmetry of the web $\W_{d-1}.$ Therefore, $K(\Leg(L_\infty\boxtimes\mathcal{H}))=K(\Leg\mathcal{H})$ by Theorem~\ref{thm:K-FX-W-egale-K-W}.
\end{proof}

\begin{rem}
We can also prove Theorem~\ref{thm:K-Leg-H-egale-K-Leg-L-infini-H} directly, without using results on webs of maximal rank. Indeed, putting $\W_{3}^{i,j}:=\Leg(L_\infty)\boxtimes\F_i\boxtimes\F_j$, for all $i,j\in\{1,\ldots,d-1\}$ with $i\neq j,$ we have
\begin{align*}
K(\Leg(L_\infty\boxtimes\mathcal{H}))=K(\Leg\mathcal{H})+\sum\limits_{1\leq i<j\leq d-1}K(\W_{3}^{i,j}).
\end{align*}
The foliation $\Leg(L_\infty)$ being defined by $\check{\omega}_0:=\mathrm{d}p,$ a direct computation using formula~(\ref{equa:eta-rst}) shows that
\begin{align*}
\eta(\W_{3}^{i,j})=\frac{\mathrm{d}\Big((\lambda_{i}\lambda_{j})(p)\Big)}{(\lambda_{i}\lambda_{j})(p)}+\frac{\mathrm{d}q}{q},
\end{align*}
so that $K(\W_{3}^{i,j})=\mathrm{d}\eta(\W_{3}^{i,j})\equiv0,$ hence $K(\Leg(L_\infty\boxtimes\mathcal{H}))=K(\Leg\mathcal{H}).$
\end{rem}

\noindent The following theorem gives an important characterization of the flatness of the dual web of a co-degree one homogeneous pre-foliation.
\begin{thm}\label{thm:holomorphie-courbure-G(I^tr)-D-ell}
{\sl Let $\preh=\ell\boxtimes\mathcal{H}$ be a homogeneous pre-foliation of type $(1,d)$ on $\pp$ with $d\geq3.$ If~the~line~$\ell$~is~invariant (resp. not invariant) by $\mathcal{H}$, then the $d$-web $\Leg\preh$ is flat if and only if its curvature $K(\Leg\preh)$ is holomorphic on $\G_{\mathcal{H}}(\ItrH)$ (resp. on $\G_{\mathcal{H}}(\ItrH)\cup D_{\ell}=\G_{\mathcal{H}}(\ItrH\cup\ell)$).
}
\end{thm}

\noindent To prove this theorem, we will need the following lemma, which is a reformulation of Lemma~3.1 of~\cite{BFM14} in~terms of homogeneous pre-foliations.
\begin{lem}[\cite{BFM14}, \rm{Lemma~3.1}]\label{lem:holomo-O}
{\sl Let $\preh$ be a homogeneous pre-foliation on $\pp.$ If the curvature of $\Leg\preh$ is holomorphic on $\pd\hspace{-0.3mm}\setminus\hspace{-0.3mm} \check{O},$ then $\Leg\preh$ is flat.}
\end{lem}

\noindent We will also need the following proposition, which has its own interest.
\begin{pro}\label{pro:holomorphie-courbure-F-W-nu-W-d-nu-1}
{\sl Let $\W_{\nu}$ be a germ of $\nu$-web on $(\mathbb{C}^{2},0)$ with $\nu\geq2.$ Assume that $\Delta(\W_{\nu})$ has an irreducible component $C$ totally invariant by $\W_{\nu}$ and of minimal multiplicity $\nu-1.$ Let $\F$ be a germ of foliation on $(\mathbb{C}^{2},0)$ leaving $C$ invariant and let $\W_{d-\nu-1}$ be a germ of regular $(d-\nu-1)$-web on $(\mathbb{C}^{2},0)$ transverse to $C.$ Then~the~curvature of the $d$-web $\W=\F\boxtimes\W_{\nu}\boxtimes\W_{d-\nu-1}$ is holomorphic along $C.$
}
\end{pro}

\begin{proof}[\sl Proof]
As in the beginning of the proof of~\cite[Proposition~2.6]{MP13}, we choose a local coordinate system $(U,(x,y))$ such that $C\cap U=\{y=0\},$ $\mathrm{T}\F|_{U}=\{\mathrm{d}y+yh(x,y)\mathrm{d}x=0\},$
\begin{Small}
\begin{align*}
\mathrm{T}\W_{\nu}|_{U}=\left\{\mathrm{d}y^{\nu}+y\Big(a_{\nu-1}(x,y)\mathrm{d}y^{\nu-1}\mathrm{d}x+\cdots+a_{0}(x,y)\mathrm{d}x^{\nu}\Big)=0\right\}
&&{\fontsize{11}{11pt}\text{and}}&&
\mathrm{T}\W_{d-\nu-1}|_{U}=\left\{\prod\limits_{l=1}^{d-\nu-1}(\mathrm{d}x+g_{l}(x,y)\mathrm{d}y)=0\right\}.
\end{align*}
\end{Small}
\hspace{-1.3mm}Then, by passing to the ramified covering $\pi\hspace{1mm}\colon(x,y)\mapsto(x,y^{\nu})$, we obtain that $\pi^{*}\F=\F_0,$\, $\pi^{*}\W_{\nu}=\boxtimes_{k=1}^{\nu}\F_k$ and~$\pi^{*}\W_{d-\nu-1}=\boxtimes_{l=1}^{d-\nu-1}\F_{\nu+l},$ where
\begin{align*}
&\F_0\hspace{0.1mm}:\hspace{0.1mm}\mathrm{d}y+\tfrac{1}{\nu}yh(x,y^{\nu})\mathrm{d}x=0,&&
\F_k\hspace{0.1mm}:\hspace{0.1mm}\mathrm{d}x+y^{\nu-2}f(x,\zeta^{k}y)\zeta^{-k}\mathrm{d}y=0,&&
\F_{\nu+l}\hspace{0.1mm}:\hspace{0.1mm}\mathrm{d}x+\nu y^{\nu-1}g_{l}(x,y^{\nu})\mathrm{d}y=0,
\end{align*}
with $\zeta=\exp(\tfrac{2\mathrm{i}\pi}{\nu}).$ Therefore we have
\begin{align*}
K(\pi^{*}\W)=K\big(\pi^{*}(\W_{\nu}\boxtimes\W_{d-\nu-1})\big)+\sum\limits_{1\leq i<j\leq d-1}K(\F_0\boxtimes\F_i\boxtimes\F_j).
\end{align*}
Now, on the one hand, \cite[Proposition~2.6]{MP13} ensures that $K(\W_{\nu}\boxtimes\W_{d-\nu-1})$ is holomorphic along $\{y=0\}$; therefore so is $K\big(\pi^{*}(\W_{\nu}\boxtimes\W_{d-\nu-1})\big)=\pi^*\big(K(\W_{\nu}\boxtimes\W_{d-\nu-1})\big).$ On the other hand, since $\{y=0\}$ is invariant by $\F_0$ and $\{y=0\}\not\subset\Tang(\F_0,\F_i\boxtimes\F_j)$, then $K(\F_0\boxtimes\F_i\boxtimes\F_j)$ is holomorphic on $\{y=0\}$ by applying \cite[Theorem~1]{MP13}, \emph{see} also \cite[Theorem~1.1]{BFM14} or \cite[Corollary~1.30]{Bed17}. It follows that $\pi^*K(\W)=K(\pi^{*}\W)$ is holomorphic on $\{y=0\}.$ As a consequence $K(\W)$ is holomorphic along $C$.
\end{proof}

\begin{rem}
Similarly, we obtain an analogue of Proposition~\ref{pro:holomorphie-courbure-F-W-nu-W-d-nu-1} by replacing the foliation $\F$ by a $2$-web $\W_2=\F_1\boxtimes\F_2$ leaving the component $C\subset\Delta(\W_{\nu})$ totally invariant.
\end{rem}

\begin{proof}[\sl Proof of Theorem~\ref{thm:holomorphie-courbure-G(I^tr)-D-ell}]
\textbf{\textit{i.}} First assume that $\ell=L_\infty.$ Then Theorem~\ref{thm:K-Leg-H-egale-K-Leg-L-infini-H} ensures that $K(\Leg\preh)=K(\Leg\mathcal{H}).$ Now, we know from~\cite[Theorem~3.1]{BM18Bull} that the flatness of the web $\Leg\mathcal{H}$ is characterized by the holomorphy of its curvature $K(\Leg\mathcal{H})$ on~$\G_{\mathcal{H}}(\ItrH).$ Therefore the same is true for the web $\Leg\preh,$ {\it i.e.} $\Leg\preh$ is flat if and only if $K(\Leg\preh)$ is holomorphic along $\G_{\mathcal{H}}(\ItrH).$
\vspace{0.1mm}

\textbf{\textit{ii.}} Now assume that $\ell$ passes through the origin. Let us fix $s\in\Sigma_{\mathcal{H}}^{\infty}$ and describe the $d$-web $\Leg\preh$ in~a~neighborhood of a generic point $m$ of the line $\check{s}$ dual to $s.$ Denote by $\nu-1\geq0$ the radiality order~of~$s$; by~\cite[Proposition~3.3]{MP13}, in a neighborhood of $m$, we can decompose $\Leg\preh$ as
\begin{align}\label{equa:Leg-preh-Leg-ell-W-nu-W-d-nu-1}
\Leg\preh=\Leg\ell\boxtimes\W_{\nu}\boxtimes\W_{d-\nu-1},
\end{align}
where $\W_{\nu}$ is an irreducible $\nu$-web leaving $\check{s}$ invariant and whose discriminant $\Delta(\W_{\nu})$ has minimal multiplicity $\nu-1$ along $\check{s}$, and where $\W_{d-\nu-1}$ is a $(d-\nu-1)$-web transverse to $\check{s}.$ More explicitly, up to linear conjugation, we can write $\ell=(y=\alpha\,x)$, $s=[1:\rho:0]$, $\check{s}=\{p=\rho\}$, $m=(\rho,q)$ and $\Gunderline_{\mathcal H}^{-1}(\rho)=\{\rho,r_1,\ldots,r_{d-\nu-1}\},$ so that  (\emph{see}~\cite[Section~2]{BM18Bull})
\begin{small}
\begin{align*}
&
\Leg\ell:(p-\alpha)\mathrm{d}q-q\mathrm{d}p=0,
&&
\W_{\nu}\Big|_{\check{s}}:\mathrm{d}p=0,
&&
\W_{d-\nu-1}\Big|_{\check{s}}:\prod_{i=1}^{d-\nu-1}\Big((\rho-r_i)\mathrm{d}q-q\mathrm{d}p\Big)=0.
\end{align*}
\end{small}
\hspace{-1mm}We deduce, in particular, the two following properties:
\begin{itemize}
\item [($\mathfrak{a}$)] if $\check{s}\not\subset\G_{\mathcal{H}}(\ItrH)$, the web $\W_{d-\nu-1}$ is regular in a neighborhood of $m$, because we then have $r_i\neq r_j$~if~$i\neq j$;

\item [($\mathfrak{b}$)] if $\check{s}\neq D_\ell=\{p=\Gunderline_{\mathcal H}(\alpha)\},$ then $\Leg\ell$ is transverse to $\check{s}$ and $\check{s}\not\subset\Tang(\Leg\ell,\W_{d-\nu-1}).$
\end{itemize}
\smallskip

\noindent If $s\in\radH$ is such that $\check{s}\not\subset\G_{\mathcal{H}}(\ItrH)\cup D_\ell$, then properties ($\mathfrak{a}$) and ($\mathfrak{b}$) ensure that the $(d-\nu)$-web $\W_{d-\nu}:=\Leg\ell\boxtimes\W_{d-\nu-1}$ is transverse to $\check{s}$ and is regular in a neighborhood of $m$. Therefore the curvature of $\Leg\preh=\W_\nu\boxtimes\W_{d-\nu}$ is holomorphic in a neighborhood of $m$ by applying~\cite[Proposition~2.6]{MP13}. It~follows~that~$K(\Leg\preh)$ is holomorphic on $\radHd\setminus(\G_{\mathcal{H}}(\ItrH)\cup D_\ell).$ Thus, according to the second assertion of Lemma~\ref{lem:Delta-Leg-H} and Lemma~\ref{lem:holomo-O}, $\Leg\preh$ is flat if and only if $K(\Leg\preh)$ is holomorphic along $\G_{\mathcal{H}}(\ItrH)\cup D_\ell.$
\smallskip

\noindent Let us show that in the particular case where $\ell$ is invariant by $\mathcal{H}$, the flatness of $\Leg\preh$ is equivalent to the holomorphy of $K(\Leg\preh)$ on $\G_{\mathcal{H}}(\ItrH).$ From the above discussion,  it suffices to prove that if $D_\ell$ is not contained in $\G_{\mathcal{H}}(\ItrH)$, then $K(\Leg\preh)$ is holomorphic on $D_\ell.$ The invariance of $\ell$ by $\mathcal{H}$ implies the existence of $s\in\Sigma_{\mathcal{H}}^{\infty}$ such that $\ell=(Os)$; then $D_\ell=\check{s}$ is invariant by the radial foliation $\Leg\ell.$ Moreover, the condition $D_\ell\not\subset\G_{\mathcal{H}}(\ItrH)$ implies that $\W_{d-\nu-1}$ is regular in a neighborhood of every generic point $m$ of $D_\ell$ (property~($\mathfrak{a}$)). By applying Theorem~1 of \cite{MP13} if $\nu=1$ and Proposition~\ref{pro:holomorphie-courbure-F-W-nu-W-d-nu-1} if $\nu\geq2,$ we deduce that $K(\Leg\preh)$ is holomorphic along $D_\ell.$
\end{proof}

\noindent From Theorem~\ref{thm:holomorphie-courbure-G(I^tr)-D-ell} we deduce the two following corollaries.
\begin{cor}\label{cor:platitude-pre-feuilletage-homogene-convexe}
{\sl Let $\preh$ be a homogeneous convex pre-foliation of co-degree $1$ and degree $d\geq3$ on $\pp.$ Then the $d$-web $\Leg\preh$ is flat.}
\end{cor}

\begin{cor}\label{cor:platitude-homogene-convexe-ell-non-invariante-non-effectif}
{\sl Let $\preh=\ell\boxtimes\mathcal{H}$ be a homogeneous pre-foliation of co-degree $1$ and degree $d\geq3$~on~$\pp.$ Assume that the homogeneous foliation $\mathcal{H}$ is convex and that the line $\ell$ is not invariant by $\mathcal{H}.$ Then the $d$-web~$\Leg\preh$ is flat if and only if its curvature $K(\Leg\preh)$ is holomorphic on $D_{\ell}=\G_{\mathcal{H}}(\ell).$
}
\end{cor}

\noindent The following theorem is an effective criterion for the holomorphy of the curvature of the web dual to a homogeneous pre-foliation $\preh=\ell\boxtimes\mathcal{H}$ (with $O\in\ell$) along an irreducible component of~$\Delta(\mathrm{Leg}\mathcal{H})\setminus(D_{\ell}\cup\check{O}).$
\begin{thm}\label{thm:holomorphie-courbure-homogene-D-neq-D-ell}
{\sl
Let $\preh=\ell\boxtimes\mathcal{H}$ be a homogeneous pre-foliation of co-degree $1$ and degree $d\geq3$ on $\pp$, defined by the $1$-form
\begin{align*}
&\hspace{1.5cm}\omega=(\alpha\,x+\beta\,y)\left(A(x,y)\mathrm{d}x+B(x,y)\mathrm{d}y\right),\quad A,B\in\mathbb{C}[x,y]_{d-1},\hspace{2mm}\gcd(A,B)=1.
\end{align*}
Let $(p,q)$ be the affine chart of $\pd$ associated to the line $\{y=px-q\}\subset\pp$ and let $D=\{p=p_0\}$ be an irreducible component of $\Delta(\mathrm{Leg}\mathcal{H})\setminus(D_{\ell}\cup\check{O}).$ Write $\Gunderline_{\mathcal H}^{-1}([p_0:1])=\{[a_1:b_1],\ldots,[a_n:b_n]\}$ and denote by $\nu_i$ the ramification index of $\Gunderline_{\mathcal H}$ at the point $[a_i:b_i]\in\mathbb{P}_{\C}^{1}.$ For $i\in\{1,\ldots,n\}$, define the polynomials $P_i\in\C[x,y]_{d-\nu_i-1}$ and~$Q_i\in\C[x,y]_{2d-\nu_i-3}$ by
\begin{SMALL}
\begin{align}\label{equa:Pi-Qi}
P_i(x,y;a_i,b_i):=\frac{\left|
\begin{array}{cc}
A(x,y)  &  A(b_i,a_i)
\\
B(x,y)  &  B(b_i,a_i)
\end{array}
\right|}{(b_iy-a_i\hspace{0.2mm}x)^{\nu_i}}
\quad{\fontsize{11}{11pt}\text{and}}\quad
Q_i(x,y;a_i,b_i):=(\nu_i-2)\left(\dfrac{\partial{B}}{\partial{x}}-\dfrac{\partial{A}}{\partial{y}}\right)P_i(x,y;a_i,b_i)+2(\nu_i+1)
\left|\begin{array}{cc}
\dfrac{\partial{P_i}}{\partial{x}} &  A(x,y)
\vspace{2mm}
\\
\dfrac{\partial{P_i}}{\partial{y}} &  B(x,y)
\end{array} \right|.
\end{align}
\end{SMALL}
\hspace{-1mm}Then the curvature of $\mathrm{Leg}\preh$ is holomorphic on $D$ if and only if
\begin{align*}
\sum_{i=1}^{n}\left(1-\frac{1}{\nu_{i}}\right)\Big(p_0\hspace{0.2mm}b_i-a_i\Big)
\left(\frac{Q_i(b_i,a_i;a_i,b_i)}{B(b_i,a_i)P_i(b_i,a_i;a_i,b_i)}+\frac{3\nu_i(\alpha+p_0\,\beta)}{\alpha\,b_i+\beta\,a_i}\right)=0.
\end{align*}
}
\end{thm}

\begin{proof}[\sl Proof]
Let $\delta\in\C$ be such that $\beta+\alpha\delta\neq0$ and $b_i-a_i\delta\neq0$ for all $i=1,\ldots,n$. Up to conjugating $\omega$~by~the~linear transformation $(x+\delta\,y,y)$, we can assume that none of the lines $\ell=(\alpha x+\beta y=0)$ and $L_i=(b_i\hspace{0.2mm}y-a_i\hspace{0.2mm}x=0)$ is vertical, {\it i.e.} that $\beta\neq0$ and $b_i\neq0$ for all $i=1,\ldots,n.$ Let us then put $\rho:=-\frac{\alpha}{\beta}$ and~$r_i:=\frac{a_i}{b_i}$; we have $\Gunderline_{\mathcal H}^{-1}(p_0)=\{r_1,\ldots,r_n\}$ with $\Gunderline_{\mathcal H}(z)=-\dfrac{A(1,z)}{B(1,z)}.$ According to \cite[Lemma~3.5]{BM22arxiv}, there therefore exists a constant $c\in\C^*$ such that $$-A(1,z)=p_0B(1,z)-c\prod_{i=1}^{n}(z-r_i)^{\nu_i}.$$ Moreover, the $d$-web $\Leg\preh$ is given in the affine chart $(p,q)$ by the differential equation
\begin{align}\label{equa:LegH}
\Big((p-\rho)x-q\Big)\Big(A(x,px-q)+pB(x,px-q)\Big)=0, \qquad \text{with} \qquad x=\frac{\mathrm{d}q}{\mathrm{d}p};
\end{align}
since $A,B\in\C[x,y]_{d-1},$ this equation can then be rewritten as
\begin{align*}
0
&=x^{d-1}\Big((p-\rho)x-q\Big)\Big(A(1,p-\tfrac{q}{x})+pB(1,p-\tfrac{q}{x})\Big)\\
&=x^{d}\Big(p-\frac{q}{x}-\rho\Big)\Big((p-p_0)B(1,p-\tfrac{q}{x})+c\prod_{i=1}^{n}(p-\tfrac{q}{x}-r_i)^{\nu_i}\Big), \qquad \text{with} \qquad x=\frac{\mathrm{d}q}{\mathrm{d}p}.
\end{align*}
Put $\check{x}:=q$,\, $\check{y}:=p-p_0$\, and \,$\check{p}:=\dfrac{\mathrm{d}\check{y}}{\mathrm{d}\check{x}}=\dfrac{1}{x}$; in these new coordinates $D=\{\check{y}=0\}$ and $\Leg\preh$ is described by the differential equation
\begin{align*}
F(\check{x},\check{y},\check{p}):=
\Big(\check{y}+p_0-\check{p}\check{x}-\rho\Big)
\Big(\check{y}B(1,\check{y}+p_0-\check{p}\check{x})+c\prod_{i=1}^{n}(\check{y}+p_0-\check{p}\check{x}-r_i)^{\nu_i}\Big)=0.
\end{align*}
We have $F(\check{x},0,\check{p})=c(-\check{x})^d\big(\check{p}-\varphi_{0}(\check{x})\big)\prod_{i=1}^{n}\big(\check{p}-\varphi_{i}(\check{x})\big)^{\nu_i},$ where $\varphi_{0}(\check{x})=\dfrac{p_0-\rho}{\check{x}}$\, and \,$\varphi_{i}(\check{x})=\dfrac{p_0-r_i}{\check{x}}$; the~hypothesis that $D\neq D_{\ell}=\{p=\Gunderline_{\mathcal H}(\rho)\}$ translates into the fact that, for all $i\in\{1,\ldots,n\},$ $r_i\neq\rho$ and therefore $\varphi_{i}\not\equiv\varphi_{0}.$ Note that if $\nu_i\geq2,$ then $\partial_{\check{y}}F\big(\check{x},0,\varphi_{i}(\check{x})\big)=(r_i-\rho)B(1,r_i)\neq0$; since $\partial_{\check{p}}F\big(\check{x},0,\varphi_{0}(\check{x})\big)\not\equiv0$ and~$\partial_{\check{p}}F\big(\check{x},0,\varphi_{i}(\check{x})\big)\not\equiv0$ if $\nu_i=1,$ we deduce that the surface\label{not:S-W}
\begin{align*}
S_{\Leg\preh}:=\left\{(\check{x},\check{y},\check{p})\in\mathbb{P}\mathrm{T}^{*}\pd\hspace{1mm}\vert\hspace{1mm}F(\check{x},\check{y},\check{p})=0\right\}
\end{align*}
is smooth along $D=\{\check{y}=0\}.$ Thus, according to~\cite[Theorem~2.1]{BM22arxiv}, the curvature of $\Leg\preh$ is holomorphic on $D=\{\check{y}=0\}$ if and only if $\sum_{i=1}^{n}(\nu_{i}-1)\varphi_{i}(\check{x})\psi_{i}(\check{x})\equiv0$\, and\, $\sum_{i=1}^{n}(\nu_{i}-1)\frac{\mathrm{d}}{\mathrm{d}\check{x}}\psi_{i}(\check{x})\equiv0,$ where, for all $i\in\{1,\ldots,n\}$ such that $\nu_i\geq2,$
\begin{align*}
\psi_{i}(\check{x})=\frac{1}{\nu_{i}}
\left[
(\nu_{i}-2)\left(d-\varphi_{i}(\check{x})\dfrac{\partial_{\check{p}}\partial_{\check{y}}F\big(\check{x},0,\varphi_{i}(\check{x})\big)}{\partial_{\check{y}}F\big(\check{x},0,\varphi_{i}(\check{x})\big)}\right)
-2(\nu_{i}+1)\left(\frac{\varphi_{0}(\check{x})}{\varphi_{i}(\check{x})-\varphi_{0}(\check{x})}+\hspace{-3.5mm}\sum\limits_{\hspace{3.5mm}j=1,j\neq i}^{n}\frac{\nu_{j}\varphi_{j}(\check{x})}{\varphi_{i}(\check{x})-\varphi_{j}(\check{x})}\right)
\right].
\end{align*}
Now, if $\nu_i\geq3$ then $\partial_{\check{p}}\partial_{\check{y}}F\big(\check{x},0,\varphi_{i}(\check{x})\big)=-\check{x}\Big(B(1,r_i)+(r_i-\rho)\partial_{y}B(1,r_i)\Big).$ It follows that
\begin{Small}
\begin{align*}
\psi_{i}(\check{x})=\psi_i:=\frac{1}{\nu_{i}}
\left[
(\nu_{i}-2)\left(d+\frac{\Big(p_0-r_i\Big)\Big(B(1,r_i)+(r_i-\rho)\partial_{y}B(1,r_i)\Big)}{(r_i-\rho)B(1,r_i)}\right)
+2(\nu_{i}+1)\left(\frac{p_0-\rho}{r_i-\rho}+\hspace{-3.5mm}\sum\limits_{\hspace{3.5mm}j=1,j\neq i}^{n}\frac{\nu_{j}(p_0-r_j)}{r_i-r_j}\right)
\right].
\end{align*}
\end{Small}
\hspace{-1mm}Therefore $K(\Leg\preh)$ is holomorphic on $D=\{\check{y}=0\}$ if and only if $\sum_{i=1}^{n}(\nu_{i}-1)\varphi_{i}(\check{x})\psi_{i}\equiv0.$ On~the~other~hand, arguing as in the proof of~\cite[Theorem~3.1]{BM22arxiv}, we obtain that
\begin{align*}
\hspace{-3.5mm}\sum\limits_{\hspace{3.5mm}j=1,j\neq i}^{n}\frac{\nu_{j}(p_0-r_j)}{r_i-r_j}
=\frac{\left|\begin{array}{cc}
\partial_{x}P_i(1,r_i;r_i,1) &  A(1,r_i)
\vspace{2mm}
\\
\partial_{y}P_i(1,r_i;r_i,1) &  B(1,r_i)
\end{array} \right|}{B(1,r_i)P_i(1,r_i;r_i,1)}
\end{align*}
and that, for all $i\in\{1,\ldots,n\}$ such that $\nu_i\geq2,$
\begin{align*}
(d-1)B(1,r_i)+(p_0-r_i)\partial_{y}B(1,r_i)=\partial_{x}B(1,r_i)-\partial_{y}A(1,r_i),
\end{align*}
so that
\begin{SMALL}
\begin{align*}
\psi_i&=\frac{1}{\nu_{i}}
\left[
(\nu_{i}-2)\left(\frac{p_0-\rho}{r_i-\rho}+\frac{\partial_{x}B(1,r_i)-\partial_{y}A(1,r_i)}{B(1,r_i)}\right)
+2(\nu_{i}+1)\left(
\frac{p_0-\rho}{r_i-\rho}+
\frac{\left|\begin{array}{cc}
\partial_{x}P_i(1,r_i;r_i,1) &  A(1,r_i)
\vspace{2mm}
\\
\partial_{y}P_i(1,r_i;r_i,1) &  B(1,r_i)
\end{array} \right|}{B(1,r_i)P_i(1,r_i;r_i,1)}
\right)
\right]
\\
&=
\dfrac{Q_{i}(1,r_i;r_i,1)}{\nu_{i}B(1,r_i)P_{i}(1,r_i;r_i,1)}+\frac{3(p_0-\rho)}{r_i-\rho}.
\end{align*}
\end{SMALL}
\hspace{-1mm}As a result, $K(\Leg\preh)$ is holomorphic along $D=\{\check{y}=0\}$ if and only if
\begin{align*}
\frac{1}{\check{x}}\sum_{i=1}^{n}\left(1-\frac{1}{\nu_{i}}\right)\Big(p_0-r_i\Big)
\left(\frac{Q_i(1,r_i;r_i,1)}{B(1,r_i)P_i(1,r_i;r_i,1)}+\frac{3\nu_i(p_0-\rho)}{r_i-\rho}\right)=0,
\end{align*}
hence the theorem follows.
\end{proof}

\begin{rems}\label{rems:holomorphie-courbure-homogene-D-neq-D-ell}
\begin{itemize}
\item []\hspace{-0.8cm}(i) We recover the fact (\emph{cf.} step \textbf{\textit{ii.}} of the proof of Theorem~\ref{thm:holomorphie-courbure-G(I^tr)-D-ell}) that~the~curvature~of $\mathrm{Leg}\preh$ is always holomorphic along $\radHd\setminus(\G_{\mathcal{H}}(\ItrH)\cup D_{\ell}).$ Indeed, if $D$ is contained in $\radHd\setminus(\G_{\mathcal{H}}(\ItrH)\cup D_{\ell}),$ then the fiber $\Gunderline_{\mathcal H}^{-1}([p_0:1])$ does not contain any non-fixed critical point of $\Gunderline_{\mathcal H}$, so that we~have~$p_0\hspace{0.2mm}b_i-a_i=0$ if $\nu_i\geq2,$ which implies (Theorem~\ref{thm:holomorphie-courbure-homogene-D-neq-D-ell}) that $K(\mathrm{Leg}\preh)$ is holomorphic on $D.$
\smallskip

\item [(ii)] We know from \cite[Theorem~3.1]{BM22arxiv} that the curvature of $\mathrm{Leg}\mathcal{H}$ is holomorphic on $D$ if and only if
\begin{align*}
\sum_{i=1}^{n}\left(1-\frac{1}{\nu_{i}}\right)\frac{(p_0\hspace{0.2mm}b_i-a_i)Q_i(b_i,a_i;a_i,b_i)}{B(b_i,a_i)P_i(b_i,a_i;a_i,b_i)}=0.
\end{align*}
From this result and Theorem~\ref{thm:holomorphie-courbure-homogene-D-neq-D-ell}, we deduce the following properties:
\smallskip
\begin{itemize}
\item  If the curvature of $\mathrm{Leg}\mathcal{H}$ is holomorphic on $D$, then the curvature of $\mathrm{Leg}\preh$ is holomorphic on $D$ if and only if
\begin{align*}
(\alpha+p_0\,\beta)\sum_{i=1}^{n}\frac{(\nu_i-1)(p_0\hspace{0.2mm}b_i-a_i)}{\alpha\,b_i+\beta\,a_i}=0.
\end{align*}

\item In particular, when $d=3$ the fiber $\Gunderline_{\mathcal H}^{-1}([p_0:1])$ is reduced to a single point, say $[a:b],$ and the~holomorphy of the curvature of $\mathrm{Leg}\preh$ on $D$ is equivalent to $(\alpha+p_0\,\beta)(p_0\hspace{0.2mm}b-a)=0$, {\it i.e.} to $\alpha+p_0\,\beta=0$ or~$[a:b]=[p_0:1]$, and therefore to $(1,p_0)\in\ell$ or $[p_0:1]$ is fixed by $\Gunderline_{\mathcal H}.$

\item If $(1,p_0)\in\ell$ then we have equivalence between the holomorphy on $D$ of $K(\mathrm{Leg}\preh)$ and that of $K(\mathrm{Leg}\mathcal{H}).$
\end{itemize}
\smallskip

\item [(iii)] Assume that $\nu_{i}=\nu\geq2$ for all $i\in\{1,\ldots,n\}.$ Then the curvature of $\mathrm{Leg}\preh$ is holomorphic on $D$ if~and~only~if
\begin{align*}
\sum_{i=1}^{n}\left(p_0\hspace{0.2mm}b_i-a_i\right)
\left(\frac{\big(\nu-2\big)\big(\partial_{x}B(b_i,a_i)-\partial_{y}A(b_i,a_i)\big)}{B(b_i,a_i)}+\frac{3\nu(\alpha+p_0\,\beta)}{\alpha\,b_i+\beta\,a_i}\right)=0.
\end{align*}
Indeed, in the above proof, put $\delta_{i,j}=\frac{(p_0-r_i)(p_0-r_j)}{(r_i-r_j)}$ and note that
\begin{Small}
\begin{align*}
\sum_{i=1}^{n}\left((\nu-1)\varphi_{i}(\check{x})\hspace{-3.5mm}\sum\limits_{\hspace{3.5mm}j=1,j\neq i}^{n}\frac{\nu(p_0-r_j)}{r_i-r_j}\right)
=\frac{\nu(\nu-1)}{\check{x}}\sum_{i=1}^{n}\hspace{-3.5mm}\sum\limits_{\hspace{3.5mm}j=1,j\neq i}^{n}\delta_{i,j}
=\frac{\nu(\nu-1)}{\check{x}}\sum_{1\leq i<j\leq n}(\delta_{i,j}+\delta_{j,i})\equiv0.
\end{align*}
\end{Small}
\hspace{-1mm}In particular, if the fiber $\Gunderline_{\mathcal H}^{-1}([p_0:1])$ contains a single non-fixed critical point of $\Gunderline_{\mathcal H}$, say $[a:b]$, then
\begin{itemize}
\item [--] either $\Gunderline_{\mathcal H}^{-1}([p_0:1])=\{[a:b]\}$, in which case $\nu=d-1$;
\item [--] or $\#\Gunderline_{\mathcal H}^{-1}([p_0:1])=2$, in which case $d$ is necessarily odd, $d=2k+1,$ and $\nu=k.$
\end{itemize}
In both cases, the curvature of $\mathrm{Leg}\preh$ is holomorphic on $D$ if and only if
\[
(\nu-2)(\alpha\,b+\beta\,a)\Big(\partial_{x}B(b,a)-\partial_{y}A(b,a)\Big)+3\nu(\alpha+p_0\,\beta)B(b,a)=0.
\]
\end{itemize}
\end{rems}

\begin{eg}
Let us consider the homogeneous pre-foliation $\preh=\ell\boxtimes\mathcal{H}$ of co-degree $1$ and odd degree $2k+1\geq5$ on $\pp$ defined by the $1$-form
\begin{align*}
\hspace{1cm}&\omega=\left(x-\tau\,y\right)\left(y^k(y-x)^k\mathrm{d}x+(y-\lambda\,x)^k(y-\mu\,x)^k\mathrm{d}y\right),&&
\text{where }\lambda,\mu\in\C\setminus\{0,1\}\hspace{2mm}\text{and}\hspace{2mm}\tau\in\C\setminus\{1\}.
\end{align*}
We know from~\cite[Example~3.4]{BM22arxiv} that $D:=\{p=0\}\subset\Delta(\mathrm{Leg}\mathcal{H})$ and that the fiber $\Gunderline_{\mathcal H}^{-1}([0:1])$ consists of the two points $[0:1]$ and $[1:1]$: the point $[0:1]$ (resp. $[1:1]$) is critical and fixed (resp. non-fixed) for $\Gunderline_{\mathcal{H}}$ with multiplicity $k-1.$ Moreover, since $\tau\neq1$, we have \begin{small}$[1:\tau]\not\in\Gunderline_{\mathcal H}^{-1}([0:1])$, so that $D\neq D_\ell=\big\{[p:1]=\Gunderline_{\mathcal H}([1:\tau])\big\}$\end{small}. From Remark~\ref{rems:holomorphie-courbure-homogene-D-neq-D-ell}~(iii), we deduce that the curvature of $\Leg\preh$ is holomorphic along~$D$ if and only if
\begin{Small}
\begin{align*}
0=(k-2)(1-\tau)\Big(\partial_{x}B(1,1)-\partial_{y}A(1,1)\Big)+3k B(1,1)
=k(1-\lambda)^k(1-\mu)^k\left(\frac{(k-2)(\tau-1)(\lambda+\mu-2\lambda\mu)}{(\lambda-1)(\mu-1)}+3\right),
\end{align*}
\end{Small}
\hspace{-1mm}{\it i.e.} if and only if the quadruple $(k,\lambda,\mu,\tau)$ satisfies the equation
\begin{small}
$(k-2)(\tau-1)(\lambda+\mu-2\lambda\mu)+3(\lambda-1)(\mu-1)=0.$
\end{small}
\hspace{-1mm}Note that, according to \cite[Example~3.4]{BM22arxiv}, the holomorphy of the curvature of $\Leg\mathcal{H}$ along $D$ is characterized by the equation $(k-2)(\lambda+\mu-2\lambda\mu)=0$. It follows, in particular,  that if the curvature of $\Leg\mathcal{H}$ is holomorphic on $D$, then the curvature of $\Leg\preh$ is not holomorphic on $D$.
\end{eg}

\begin{cor}\label{cor:holomorphie-courbure-droite-inflex-nu-1-D-neq-D-ell}
{\sl Let $\preh=\ell\boxtimes\mathcal{H}$ be a homogeneous pre-foliation of co-degree $1$ and degree $d\geq3$ on $\pp$, defined by the $1$-form
\begin{align*}
&\hspace{1.5cm}\omega=(\alpha\,x+\beta\,y)\left(A(x,y)\mathrm{d}x+B(x,y)\mathrm{d}y\right),\quad A,B\in\mathbb{C}[x,y]_{d-1},\hspace{2mm}\gcd(A,B)=1.
\end{align*}
Assume that the foliation $\mathcal{H}$ has a transverse inflection line $T=(ax+by=0)$ of order $\nu-1$. Assume moreover that~$[-a:b]\in\mathbb{P}^{1}_{\mathbb{C}}$ is the only non-fixed critical point of $\Gunderline_{\mathcal H}$ in its fiber $\Gunderline_{\mathcal H}^{-1}(\Gunderline_{\mathcal H}([-a:b]))$ and that~$[-\alpha:\beta]\not\in\Gunderline_{\mathcal H}^{-1}(\Gunderline_{\mathcal H}([-a:b])).$ Then the curvature of $\mathrm{Leg}\preh$ is holomorphic on $T'=\mathcal{G}_{\mathcal{H}}(T)$ if~and~only~if
\[
(\alpha\,b-\beta\,a)Q(b,-a;a,b)+3\nu\Big(\alpha\,B(b,-a)-\beta\,A(b,-a)\Big)P(b,-a;a,b)=0,
\]
where
\begin{Small}
\begin{align*}
Q(x,y;a,b):=(\nu-2)\left(\dfrac{\partial{B}}{\partial{x}}-\dfrac{\partial{A}}{\partial{y}}\right)P(x,y;a,b)+2(\nu+1)
\left|\begin{array}{cc}
\dfrac{\partial{P}}{\partial{x}} &  A(x,y)
\vspace{2mm}
\\
\dfrac{\partial{P}}{\partial{y}} &  B(x,y)
\end{array} \right|
\quad\text{and}\quad
P(x,y;a,b):=\frac{\left|
\begin{array}{cc}
A(x,y)  &  A(b,-a)
\\
B(x,y)  &  B(b,-a)
\end{array}
\right|}{(ax+by)^{\nu}}.
\end{align*}
\end{Small}
}
\end{cor}

\begin{proof}[\sl Proof]
Up to linear conjugation, we can assume that $T'\neq L_\infty$; then $T'$ has the equation $p=p_0$, where $p_0=-\frac{A(b,-a)}{B(b,-a)}$. According to Theorem~\ref{thm:holomorphie-courbure-homogene-D-neq-D-ell}, the curvature of $\Leg\preh$ is holomorphic on $T'$ if and only if
\begin{small}
\begin{align*}
\left(1-\tfrac{1}{\nu}\right)\left(p_0\hspace{0.2mm}b+a\right)
\left(\frac{Q(b,-a;a,b)}{B(b,-a)P(b,-a;a,b)}+\frac{3\nu(\alpha+p_0\,\beta)}{\alpha\,b-\beta\,a}\right)=0.
\end{align*}
\end{small}
\hspace{-1mm}Now, the hypothesis that the point $[-a:b]$ is not fixed by $\Gunderline_{\mathcal H}$ translates into $p_0\hspace{0.2mm}b+a\neq0.$ It follows that $K(\Leg\preh)$ is holomorphic on $T'$ if and only if
\begin{small}
\begin{align*}
\frac{Q(b,-a;a,b)}{P(b,-a;a,b)}+\frac{3\nu\big(\alpha\,B(b,-a)-\beta\,A(b,-a)\big)}{\alpha\,b-\beta\,a}=0,
\end{align*}
\end{small}
\hspace{-1mm}hence the corollary holds.
\end{proof}

\noindent In particular, we have:
\begin{cor}\label{cor:holomorphie-courbure-droite-inflex-maximal-d-2-T-neq-ell}
{\sl Let $\preh=\ell\boxtimes\mathcal{H}$ be a homogeneous pre-foliation of co-degree $1$ and degree $d\geq3$ on $\pp$, defined by the $1$-form
\begin{align*}
&\hspace{1.5cm}\omega=(\alpha\,x+\beta\,y)\left(A(x,y)\mathrm{d}x+B(x,y)\mathrm{d}y\right),\quad A,B\in\mathbb{C}[x,y]_{d-1},\hspace{2mm}\gcd(A,B)=1.
\end{align*}
Assume that $\mathcal{H}$ admits a transverse inflection line $T=(ax+by=0)$ of maximal order $d-2$ and that~$T\neq\ell.$ Then the curvature of $\mathrm{Leg}\preh$ is holomorphic along $T'=\mathcal{G}_{\mathcal{H}}(T)$ if and only~if
\[
(d-3)(\alpha\,b-\beta\,a)\Big(\partial_{x}B(b,-a)-\partial_{y}A(b,-a)\Big)+3(d-1)\Big(\alpha\,B(b,-a)-\beta\,A(b,-a)\Big)=0.
\]
}
\end{cor}

\noindent The following theorem is an effective criterion for the holomorphy of the curvature of the web dual to a homogeneous pre-foliation $\preh=\ell\boxtimes\mathcal{H}$ (with $O\in\ell$) along the component~$D_{\ell}\subset\Delta(\Leg\preh).$
\begin{thm}\label{thm:holomorphie-courbure-homogene-D-ell}
{\sl
Let $\preh=\ell\boxtimes\mathcal{H}$ be a homogeneous pre-foliation of co-degree $1$ and degree $d\geq3$ on $\pp$, defined by the $1$-form
\begin{align*}
&\hspace{1.5cm}\omega=(\alpha\,x+\beta\,y)\left(A(x,y)\mathrm{d}x+B(x,y)\mathrm{d}y\right),\quad A,B\in\mathbb{C}[x,y]_{d-1},\hspace{2mm}\gcd(A,B)=1.
\end{align*}
Write $\Gunderline_{\mathcal H}^{-1}(\Gunderline_{\mathcal H}([-\alpha:\beta]))=\{[-\alpha:\beta],[a_1:b_1],\ldots,[a_n:b_n]\}$ and denote by $\nu_i$ (resp. $\nu_0$) the ramification index of $\Gunderline_{\mathcal H}$ at the point $[a_i:b_i]$ (resp. $[-\alpha:\beta]$). Define the polynomials $P_0\in\C[x,y]_{d-\nu_0-1}$ and $Q_0\in\C[x,y]_{2d-\nu_0-3}$ by
\begin{Small}
\begin{align*}
P_0(x,y;\alpha,\beta):=\frac{\left|
\begin{array}{cc}
A(x,y)  &  A(\beta,-\alpha)
\\
B(x,y)  &  B(\beta,-\alpha)
\end{array}
\right|}{(\alpha\,x+\beta\,y)^{\nu_0}}
\quad{\fontsize{11}{11pt}\text{and}}\quad
Q_0(x,y;\alpha,\beta):=(\nu_0-1)\left(\dfrac{\partial{B}}{\partial{x}}-\dfrac{\partial{A}}{\partial{y}}\right)P_0(x,y;\alpha,\beta)+(2\nu_0+1)
\left|\begin{array}{cc}
\dfrac{\partial{P_0}}{\partial{x}} &  A(x,y)
\vspace{2mm}
\\
\dfrac{\partial{P_0}}{\partial{y}} &  B(x,y)
\end{array} \right|.
\end{align*}
\end{Small}
\hspace{-1mm}Assume that $\Gunderline_{\mathcal H}([-\alpha:\beta])\neq\infty$ and let $p_0\in\C$ be such that $[p_0:1]=\Gunderline_{\mathcal H}([-\alpha:\beta]).$ Then the curvature of $\mathrm{Leg}\preh$ is holomorphic on $D_\ell$ if and only if
\begin{small}
\begin{align*}
\left(1+\frac{1}{\nu_{0}}\right)\frac{(\alpha+p_0\,\beta)Q_0(\beta,-\alpha;\alpha,\beta)}{B(\beta,-\alpha)P_0(\beta,-\alpha;\alpha,\beta)}+
\sum_{i=1}^{n}\left(1-\frac{1}{\nu_{i}}\right)\Big(p_0\hspace{0.2mm}b_i-a_i\Big)
\left(\frac{Q_i(b_i,a_i;a_i,b_i)}{B(b_i,a_i)P_i(b_i,a_i;a_i,b_i)}+\frac{3\nu_i(\alpha+p_0\,\beta)}{\alpha\,b_i+\beta\,a_i}\right)=0,
\end{align*}
\end{small}
\hspace{-1mm}where the $P_i$'s and the $Q_i$'s ($i=1,\ldots,n$) are the polynomials given by~(\ref{equa:Pi-Qi}).
}
\end{thm}

\noindent Note that the $d$-web $\Leg\preh=\Leg\ell\boxtimes\Leg\mathcal{H}$ is not smooth along the component $D_\ell\subset\Tang(\Leg\ell,\Leg\mathcal{H})$ and therefore we cannot apply Theorem~2.1 of \cite{BM22arxiv} to $\Leg\preh$ as we did in the proof of Theorem~\ref{thm:holomorphie-courbure-homogene-D-neq-D-ell}. To prove Theorem~\ref{thm:holomorphie-courbure-homogene-D-ell}, we will first establish, for a foliation $\F$ and a web $\W$ smooth along an irreducible component $D$ of $\Tang(\F,\W),$ an effective criterion for the holomorphy of the curvature of $\F\boxtimes\W$ along $D.$
\begin{thm}\label{thm-critere-holomorphie-FW}
{\sl Let $\W$ be a holomorphic $(d-1)$-web on a complex surface $M.$ Let $\F$ be a holomorphic foliation on $M.$ Assume that $\W$ is smooth along an irreducible component $D$ of $\Tang(\F,\W).$ Then the fundamental form $\eta(\F\boxtimes\W)$ has simple poles along $D.$ More precisely, choose a local coordinate system $(x,y)$ on $M$ such that $D=\{y=0\}$ and let $F(x,y,p)=0,$ $p=\frac{\mathrm{d}y}{\mathrm{d}x},$ be an implicit differential equation defining $\W.$ Write $F(x,0,p)=a_{0}(x)\prod\limits_{\alpha=1}^{n}(p-\varphi_{\alpha}(x))^{\nu_{\alpha}},$ with $\varphi_{\alpha}\not\equiv\varphi_{\beta}$ if~$\alpha\neq\beta,$ and assume that $\F$ is given by a $1$-form $\omega$ of type $\omega=\mathrm{d}y-\left(\varphi_{1}(x)+yf(x,y)\right)\mathrm{d}x.$ Define $Q(x,p)$ by $F(x,0,p)=(p-\varphi_{1}(x))^{\nu_{1}}Q(x,p)$ and put
\begin{Small}
\begin{align*}
&
h(x)=\frac{1}{\nu_{1}}
\left[
(\nu_{1}-1)\left(d-1-\varphi_{1}(x)\dfrac{\partial_{p}\partial_{y}F(x,0,\varphi_{1}(x))+
2\delta_{\nu_1,2}f(x,0)Q(x,\varphi_{1}(x))}{\partial_{y}F(x,0,\varphi_{1}(x))}\right)
-(2\nu_{1}+1)\sum_{\alpha=2}^{n}\frac{\nu_{\alpha}\varphi_{\alpha}(x)}{\varphi_{1}(x)-\varphi_{\alpha}(x)}
\right]
\end{align*}
\end{Small}
\hspace{-1mm}(where $\delta_{\nu_1,2}=1$ if $\nu_1=2$ and $0$ otherwise). Let $\psi_{\alpha}$ be a function of the coordinate $x$ defined, for all $\alpha\in\{1,\ldots,n\}$ such that $\nu_{\alpha}\geq2,$ by
\begin{align*}
\psi_{\alpha}(x)=\frac{1}{\nu_{\alpha}}
\left[
(\nu_{\alpha}-2)
\left(
d-1-\varphi_{\alpha}(x)\dfrac{\partial_{p}\partial_{y}F\big(x,0,\varphi_{\alpha}(x)\big)}{\partial_{y}F\big(x,0,\varphi_{\alpha}(x)\big)}
\right)
-2(\nu_{\alpha}+1)\hspace{-3.5mm}\sum\limits_{\hspace{3.5mm}\beta=1,\beta\neq\alpha}^{n}\frac{\nu_{\beta}\varphi_{\beta}(x)}{\varphi_{\alpha}(x)-\varphi_{\beta}(x)}
\right].
\end{align*}
Then the $1$-form $\eta(\F\boxtimes\W)-\frac{\theta}{6y}$ is holomorphic along $D=\{y=0\}$, where
\begin{Small}
\begin{align*}
\theta=(\nu_1+1)\left[h(x)\Big(\mathrm{d}y-\varphi_1(x)\mathrm{d}x\Big)+(\nu_1-1)\mathrm{d}y\right]+
\sum_{\alpha=2}^{n}(\nu_\alpha-1)\left[\left(\psi_\alpha(x)+\frac{3\varphi_1(x)}{\varphi_1(x)-\varphi_\alpha(x)}\right)
\big(\mathrm{d}y-\varphi_\alpha(x)\mathrm{d}x\big)+(\nu_\alpha-2)\mathrm{d}y\right].
\end{align*}
\end{Small}
\hspace{-1mm}In particular, the curvature $K(\F\boxtimes\W)$ is holomorphic along $D$ if and only if
\begin{align*}
&\hspace{-4cm}
(\nu_1+1)\varphi_1(x)h(x)+
\sum_{\alpha=2}^{n}(\nu_{\alpha}-1)\varphi_{\alpha}(x)\left(\psi_{\alpha}(x)+\frac{3\varphi_{1}(x)}{\varphi_{1}(x)-\varphi_{\alpha}(x)}\right)\equiv0
\\
&\hspace{-4.95cm}\text{and}
\\
&\hspace{-4cm}
\frac{\mathrm{d}}{\mathrm{d}x}\left((\nu_1+1)h(x)+
\sum_{\alpha=2}^{n}(\nu_{\alpha}-1)\left(\psi_{\alpha}(x)+\frac{3\varphi_{1}(x)}{\varphi_{1}(x)-\varphi_{\alpha}(x)}\right)\right)\equiv0.
\end{align*}
}
\end{thm}

\begin{proof}[\sl Proof]
In a neighborhood of a generic point $m$ of $D$, the web $\W$ decomposes as $\W=\boxtimes_{\alpha=1}^{n}\W_{\alpha},$ where~$\W_{\alpha}=\boxtimes_{i=1}^{\nu_{\alpha}}\F_{i}^{\alpha}$ and $\F_{i}^{\alpha}|_{y=0}:\mathrm{d}y-\varphi_{\alpha}(x)\mathrm{d}x=0.$ Then $\eta(\F\boxtimes\W)=\eta(\W)+\eta_1+\eta_2+\eta_3+\eta_4,$ where
\begin{small}
\begin{align*}
&
\eta_1
=\sum_{1\le i<j\le\nu_{1}}\eta(\F\boxtimes\F_{i}^{1}\boxtimes\F_{j}^{1}),
&&
\eta_2=\sum_{\alpha=2}^{n}
\sum_{\raisebox{2mm}{${\underset{1\le j\le\nu_{\alpha}}{\underset{1\le i\le\nu_{1}}{}}}$}}\eta(\F\boxtimes\F_{i}^{1}\boxtimes\F_{j}^{\alpha}),
\\
&
\eta_3=\sum_{\alpha=2}^{n}\sum_{1\le i<j\le\nu_{\alpha}}\eta(\F\boxtimes\F_{i}^{\alpha}\boxtimes\F_{j}^{\alpha}),
&&
\eta_4=\sum_{2\le\alpha<\beta\le n}
\sum_{\raisebox{2mm}{${\underset{1\le j\le\nu_{\beta}}{\underset{1\le i\le\nu_{\alpha}}{}}}$}}\eta(\F\boxtimes\F_{i}^{\alpha}\boxtimes\F_{j}^{\beta}).
\end{align*}
\end{small}
\hspace{-1mm}According to~\cite[Theorem~2.1]{BM22arxiv}, the principal part of the \textsc{Laurent} series of $\eta(\W)$ at $y=0$ is given by $\frac{\theta_0}{y}$, where
\begin{align*}
\theta_0=\frac{1}{6}\sum\limits_{\alpha=1}^{n}\big(\nu_{\alpha}-1\big)\Big[\psi_{\alpha}(x)\big(\mathrm{d}y-\varphi_{\alpha}(x)\mathrm{d}x\big)
+\big(\nu_{\alpha}-2\big)\mathrm{d}y\Big].
\end{align*}

\noindent As for the $1$-forms $\eta_1,\ldots,\eta_4$, first note that, as in the proof of \cite[Theorem~2.1]{BM22arxiv}, the slope $p_{\hspace{-0.2mm}j}$ ($j=1,\ldots,\nu_{\alpha}$) of $\mathrm{T}_{(x,y)}\F_{j}^{\alpha}$ can be written as
\begin{align*}
&\hspace{1cm}p_{\hspace{-0.2mm}j}=\lambda_{\alpha,j}(x,y):=\varphi_{\alpha}(x)+\sum_{k\geq1}f_{\alpha,k}(x)\zeta_{\alpha}^{jk}y^{\frac{k}{\nu_{\alpha}}},
\quad\text{where}\hspace{1mm}
f_{\alpha,k}\in\C\{x\},
\end{align*}
with $f_{\alpha,1}\not\equiv0$ and $\mathrm{\zeta_{\alpha}}=\exp(\tfrac{2\mathrm{i}\pi}{\nu_{\alpha}}).$ Moreover, for $\alpha=1$, if $\nu_1\geq2,$ then
\begin{align}\label{equa:preuve-thm-nu1}
&\hspace{-3.6cm}(f_{1,1}(x))^{\nu_1}=-\frac{\partial_{y}F\big(x,0,\varphi_{1}(x)\big)}{Q(x,\varphi_{1}(x))}
\end{align}
and, for all $\alpha\in\{1,\ldots,n\}$ such that $\nu_\alpha\geq2$, we have
\begin{align}\label{equa:preuve-thm-nu-alpha}
&\hspace{1.4cm}\frac{f_{\alpha,2}(x)}{(f_{\alpha,1}(x))^2}
=\frac{1}{\nu_{\alpha}}\left[\dfrac{\partial_{p}\partial_{y}F\big(x,0,\varphi_{\alpha}(x)\big)}{\partial_{y}F\big(x,0,\varphi_{\alpha}(x)\big)}
-\hspace{-3.5mm}\sum\limits_{\hspace{3.5mm}\beta=1,\beta\neq\alpha}^{n}\frac{\nu_{\beta}}{\varphi_{\alpha}(x)-\varphi_{\beta}(x)}\right].
\end{align}

\noindent Put $\lambda_0(x,y)=\varphi_{1}(x)+yf(x,y)$; according to~\cite[Lemma~2.8]{BM22arxiv}, we have $\eta(\F\boxtimes\F_{i}^{1}\boxtimes\F_{j}^{1})=a_{i,j}(x,y)\mathrm{d}x+b_{i,j}(x,y)\mathrm{d}y,$ where
\begin{Small}
\begin{align*}
&\hspace{-3.3cm}a_{i,j}=
-\frac{\left(\partial_y(\lambda_{1,i}\lambda_{1,j})-\partial_x\lambda_{0}\right)\lambda_{0}}{(\lambda_{1,i}-\lambda_{0})(\lambda_{1,j}-\lambda_{0})}
-\frac{\left(\partial_y(\lambda_{1,i}\lambda_{0})-\partial_x\lambda_{1,j}\right)\lambda_{1,j}}{(\lambda_{1,i}-\lambda_{1,j})(\lambda_{0}-\lambda_{1,j})}
-\frac{\left(\partial_y(\lambda_{1,j}\lambda_{0})-\partial_x\lambda_{1,i}\right)\lambda_{1,i}}{(\lambda_{1,j}-\lambda_{1,i})(\lambda_{0}-\lambda_{1,i})}
\\
&\hspace{-4.3cm}{\fontsize{11}{11pt}\text{and}}
\\
&\hspace{-3.3cm}b_{i,j}=
\frac{\partial_y(\lambda_{1,i}\lambda_{1,j})-\partial_x\lambda_{0}}{(\lambda_{1,i}-\lambda_{0})(\lambda_{1,j}-\lambda_{0})}
+\frac{\partial_y(\lambda_{1,i}\lambda_{0})-\partial_x\lambda_{1,j}}{(\lambda_{1,i}-\lambda_{1,j})(\lambda_{0}-\lambda_{1,j})}
+\frac{\partial_y(\lambda_{1,j}\lambda_{0})-\partial_x\lambda_{1,i}}{(\lambda_{1,j}-\lambda_{1,i})(\lambda_{0}-\lambda_{1,i})}.
\end{align*}
\end{Small}
\hspace{-1.2mm}Writing $f(x,y)=\sum_{k\geq0}f_{0,k}(x)y^k$ and putting $w_{1}=y^{\frac{1}{\nu_{1}}}$, a straightforward computation leads to the following equalities:
\begin{Small}
\begin{align*}
&\hspace{-1.6cm}\partial_y(\lambda_{1,i}\lambda_{1,j})-\partial_x\lambda_{0}=\frac{1}{\nu_1y}
\left[
(\zeta_{1}^{i}+\zeta_{1}^{j})\varphi_{1}f_{1,1}w_1+
2\left(\zeta_{1}^{i+j}f_{1,1}^{2}+(\zeta_{1}^{2i}+\zeta_{1}^{2j})\varphi_{1}f_{1,2}-\delta_{\nu_1,2}\varphi_{1}^{'}\right)w_{1}^{2}+\cdots
\right],\\
&\hspace{-1.6cm}\partial_y(\lambda_{1,i}\lambda_{0})-\partial_x\lambda_{1,j}=\frac{1}{\nu_1y}
\left[
\zeta_{1}^{i}\varphi_1f_{1,1}w_1+2\left(\zeta_{1}^{2i}\varphi_1f_{1,2}+(\varphi_1f_{0,0}-\varphi_{1}^{'})\delta_{\nu_1,2}\right)w_{1}^{2}+\cdots
\right],\\
&\hspace{-1.6cm}(\lambda_{1,i}-\lambda_{0})(\lambda_{1,j}-\lambda_{0})=\zeta_{1}^{i+j}f_{1,1}^{2}w_{1}^{2}+(\zeta_{1}^{2i+j}+\zeta_{1}^{i+2j})f_{1,1}f_{1,2}w_{1}^{3}+\cdots,\\
&\hspace{-1.6cm}(\lambda_{1,i}-\lambda_{1,j})(\lambda_{0}-\lambda_{1,j})=(\zeta_{1}^{2j}-\zeta_{1}^{i+j})f_{1,1}^{2}w_{1}^{2}+
f_{1,1}\left(\big(2\zeta_{1}^{3j}-\zeta_{1}^{2i+j}-\zeta_{1}^{i+2j}\big)f_{1,2}-2\delta_{\nu_1,2}f_{0,0}\right)w_{1}^{3}+\cdots.
\end{align*}
\end{Small}
\hspace{-1mm}These equalities allow us to check that $a_{i,j}$ and $b_{i,j}$ can be written as
\begin{small}
\begin{align*}
&\hspace{-0.6cm}a_{i,j}=\frac{\varphi_{1}\left(\varphi_{1}f_{1,2}-f_{1,1}^{2}-\delta_{\nu_1,2}\varphi_{1}f_{0,0}\right)+w_1A_{i,j}}{\nu_1yf_{1,1}^{2}},
&&
b_{i,j}=\frac{2f_{1,1}^{2}-\varphi_{1}f_{1,2}+\delta_{\nu_1,2}\varphi_{1}f_{0,0}+w_1B_{i,j}}{\nu_1yf_{1,1}^{2}},
\end{align*}
\end{small}
\hspace{-1mm}where $A_{i,j},B_{i,j}\in\C\{x,w_{1}\}.$ Since $\eta_1$ is a uniform and meromorphic $1$-form, we deduce that the principal part of the \textsc{Laurent} series of $\eta_1$ at $y=0$ is given by $\frac{\theta_1}{y}$, where
\begin{SMALL}
\begin{align*}
\theta_1&=\binom{{\nu_{1}}}{{2}}
\left(\frac{\varphi_{1}(x)\Big(\varphi_{1}(x)f_{1,2}(x)-f_{1,1}(x)^{2}-\delta_{\nu_1,2}\varphi_{1}(x)f_{0,0}(x)\Big)}{\nu_1f_{1,1}(x)^{2}}\mathrm{d}x+
\frac{2f_{1,1}(x)^{2}-\varphi_{1}(x)f_{1,2}(x)+\delta_{\nu_1,2}\varphi_{1}(x)f_{0,0}(x)}{\nu_1f_{1,1}(x)^{2}}\mathrm{d}y
\right)
\\
&=\frac{1}{2}(\nu_{1}-1)
\left[
\varphi_{1}(x)
\left(\frac{\delta_{\nu_1,2}f_{0,0}(x)}{f_{1,1}(x)^{2}}-\frac{f_{1,2}(x)}{f_{1,1}(x)^{2}}\right)\big(\mathrm{d}y-\varphi_{1}(x)\mathrm{d}x\big)+2\mathrm{d}y-\varphi_{1}(x)\mathrm{d}x
\right].
\end{align*}
\end{SMALL}
\hspace{-1mm}Thanks to~(\ref{equa:preuve-thm-nu1}), (\ref{equa:preuve-thm-nu-alpha}) and the equality $f_{0,0}(x)=f(x,0)$, the $1$-form $\theta_1$ can be rewritten as
\begin{SMALL}
\begin{align*}
\theta_1=\frac{1}{2}\left(1-\frac{1}{\nu_1}\right)
\left(d-1-\varphi_{1}(x)\dfrac{\partial_{p}\partial_{y}F(x,0,\varphi_{1}(x))+2\delta_{\nu_1,2}f(x,0)Q(x,\varphi_{1}(x))}{\partial_{y}F(x,0,\varphi_{1}(x))}
+\sum_{\alpha=2}^{n}\frac{\nu_{\alpha}\varphi_{\alpha}(x)}{\varphi_{1}(x)-\varphi_{\alpha}(x)}
\right)\big(\mathrm{d}y-\varphi_{1}(x)\mathrm{d}x\big)+\frac{1}{2}(\nu_1-1)\mathrm{d}y.
\end{align*}
\end{SMALL}

\noindent Let us now pass to $\eta_2.$ Put $w_{\alpha,1}=y^{\frac{1}{\nu_{1}\nu_{\alpha}}}$; again by \cite[Lemma~2.8]{BM22arxiv}, we have $\eta(\F\boxtimes\F_{i}^{1}\boxtimes\F_{j}^{\alpha})=a_{i,j}^{\alpha}(x,y)\mathrm{d}x+b_{i,j}^{\alpha}(x,y)\mathrm{d}y,$ where
\begin{Small}
\begin{align*}
\hspace{-1.6cm}a_{i,j}^{\alpha}&=
-\frac{\left(\partial_y(\lambda_{1,i}\lambda_{\alpha,j})-\partial_x\lambda_{0}\right)\lambda_{0}}{(\lambda_{1,i}-\lambda_{0})(\lambda_{\alpha,j}-\lambda_{0})}
-\frac{\left(\partial_y(\lambda_{1,i}\lambda_{0})-\partial_x\lambda_{\alpha,j}\right)\lambda_{\alpha,j}}{(\lambda_{1,i}-\lambda_{\alpha,j})(\lambda_{0}-\lambda_{\alpha,j})}
-\frac{\left(\partial_y(\lambda_{\alpha,j}\lambda_{0})-\partial_x\lambda_{1,i}\right)\lambda_{1,i}}{(\lambda_{\alpha,j}-\lambda_{1,i})(\lambda_{0}-\lambda_{1,i})}\\
\hspace{-1.6cm}&=\frac{1}{\nu_{1}y}
\left(\frac{\varphi_{1}\varphi_{\alpha}}{\varphi_{1}-\varphi_{\alpha}}+w_{\alpha,1}A_{i,j}^{\alpha}\right)
\end{align*}
\end{Small}
\hspace{-1mm}and
\begin{Small}
\begin{align*}
\hspace{-2.2cm}b_{i,j}^{\alpha}&=
\frac{\partial_y(\lambda_{1,i}\lambda_{\alpha,j})-\partial_x\lambda_{0}}{(\lambda_{1,i}-\lambda_{0})(\lambda_{\alpha,j}-\lambda_{0})}+
\frac{\partial_y(\lambda_{1,i}\lambda_{0})-\partial_x\lambda_{\alpha,j}}{(\lambda_{1,i}-\lambda_{\alpha,j})(\lambda_{0}-\lambda_{\alpha,j})}+
\frac{\partial_y(\lambda_{\alpha,j}\lambda_{0})-\partial_x\lambda_{1,i}}{(\lambda_{\alpha,j}-\lambda_{1,i})(\lambda_{0}-\lambda_{1,i})}\\
\hspace{-2.2cm}&=-\frac{1}{\nu_{1}y}
\left(\frac{\varphi_{\alpha}}{\varphi_{1}-\varphi_{\alpha}}+w_{\alpha,1}B_{i,j}^{\alpha}\right),
\end{align*}
\end{Small}
\hspace{-1mm}where $A_{i,j}^{\alpha},B_{i,j}^{\alpha}\in\C\{x,w_{\alpha,1}\}.$ The $1$-form $\eta_2$ being uniform and meromorphic, it follows that the principal part of the \textsc{Laurent} series of $\eta_2$ at $y=0$ is given by $\frac{\theta_2}{y}$, where
\begin{align*}
\theta_2&=
\sum_{\alpha=2}^{n}\nu_{1}\nu_{\alpha}
\left(\frac{\varphi_{1}(x)\varphi_{\alpha}(x)}{\nu_{1}\left(\varphi_{1}(x)-\varphi_{\alpha}(x)\right)}\mathrm{d}x-\frac{\varphi_{\alpha}(x)}{\nu_{1}\left(\varphi_{1}(x)-\varphi_{\alpha}(x)\right)}\mathrm{d}y\right)\\
&=-(\mathrm{d}y-\varphi_{1}(x)\mathrm{d}x)\sum_{\alpha=2}^{n}\frac{\nu_{\alpha}\varphi_{\alpha}(x)}{\varphi_{1}(x)-\varphi_{\alpha}(x)}.
\end{align*}

\noindent Similarly, putting $w_{\alpha}=y^{\frac{1}{\nu_{\alpha}}}$ and using~\cite[Lemma~2.8]{BM22arxiv}, we obtain that
\begin{small}
\begin{align*}
\eta(\F\boxtimes\F_{i}^{\alpha}\boxtimes\F_{j}^{\alpha})=\frac{1}{\nu_{\alpha}y}
\left[
\left(-\frac{\varphi_{1}(x)\varphi_{\alpha}(x)}{\varphi_{1}(x)-\varphi_{\alpha}(x)}+w_{\alpha}\,\tilde{A}_{i,j}^{\alpha}(x,w_{\alpha})\right)\mathrm{d}x+
\left(\frac{\varphi_{1}(x)}{\varphi_{1}(x)-\varphi_{\alpha}(x)}+w_{\alpha}\,\tilde{B}_{i,j}^{\alpha}(x,w_{\alpha})\right)\mathrm{d}y
\right],
\end{align*}
\end{small}
\hspace{-1mm}where $\tilde{A}_{i,j}^{\alpha},\tilde{B}_{i,j}^{\alpha}\in\C\{x,w_{\alpha}\}$, so that the principal part of the \textsc{Laurent} series of $\eta_3$ at $y=0$ is given by $\frac{\theta_3}{y}$, where
\begin{align*}
\theta_3&=
\sum_{\alpha=2}^{n}\binom{{\nu_{\alpha}}}{{2}}\left(
-\frac{\varphi_{1}(x)\varphi_{\alpha}(x)}{\nu_{\alpha}\left(\varphi_{1}(x)-\varphi_{\alpha}(x)\right)}\mathrm{d}x+
\frac{\varphi_{1}(x)}{\nu_{\alpha}\left(\varphi_{1}(x)-\varphi_{\alpha}(x)\right)}\mathrm{d}y
\right)\\
&=\frac{1}{2}\sum_{\alpha=2}^{n}\frac{(\nu_{\alpha}-1)\varphi_1(x)\left(\mathrm{d}y-\varphi_{\alpha}(x)\mathrm{d}x\right)}{\varphi_1(x)-\varphi_{\alpha}(x)}.
\end{align*}

\noindent Finally, since $(\varphi_{1}-\varphi_{\alpha})(\varphi_{\alpha}-\varphi_{\beta})(\varphi_{\beta}-\varphi_{1})\not\equiv0$ for all $\beta>\alpha\geq2$, \cite[Lemma~2.8]{BM22arxiv} implies that the $1$-form $\eta(\F\boxtimes\F_{i}^{\alpha}\boxtimes\F_{j}^{\beta})$ has no poles along $y=0$;
therefore the same is true for the $1$-form $\eta_4.$

\noindent As a result, the principal part of the \textsc{Laurent} series of $\eta(\F\boxtimes\W)$ at $y=0$ is given by $\frac{\widehat{\theta}}{y}$, where
\begin{Small}
\begin{align*}
\widehat{\theta}&=\theta_0+\theta_1+\theta_2+\theta_3\\
&=\frac{1}{6}\Big((\nu_1+1)h(x)-(\nu_1-1)\psi_{1}(x)\Big)\big(\mathrm{d}y-\varphi_{1}(x)\mathrm{d}x\big)+\frac{1}{2}(\nu_1-1)\mathrm{d}y+
\frac{1}{6}\sum\limits_{\alpha=1}^{n}\big(\nu_{\alpha}-1\big)\Big[\psi_{\alpha}(x)\big(\mathrm{d}y-\varphi_{\alpha}(x)\mathrm{d}x\big)
+\big(\nu_{\alpha}-2\big)\mathrm{d}y\Big]\\
&\hspace{4.5mm}+\frac{1}{2}\sum_{\alpha=2}^{n}\frac{(\nu_{\alpha}-1)\varphi_1(x)\left(\mathrm{d}y-\varphi_{\alpha}(x)\mathrm{d}x\right)}{\varphi_1(x)-\varphi_{\alpha}(x)}\\
&=\frac{1}{6}(\nu_1+1)\left[h(x)\Big(\mathrm{d}y-\varphi_1(x)\mathrm{d}x\Big)+(\nu_1-1)\mathrm{d}y\right]+
\frac{1}{6}\sum_{\alpha=2}^{n}(\nu_\alpha-1)\left[\left(\psi_\alpha(x)+\frac{3\varphi_1(x)}{\varphi_1(x)-\varphi_\alpha(x)}\right)
\big(\mathrm{d}y-\varphi_\alpha(x)\mathrm{d}x\big)+(\nu_\alpha-2)\mathrm{d}y\right]\\
&=\frac{\theta}{6},
\end{align*}
\end{Small}
\hspace{-1mm}hence the theorem follows.
\end{proof}

\begin{proof}[\sl Proof of Theorem~\ref{thm:holomorphie-courbure-homogene-D-ell}]
As in the proof of Theorem~\ref{thm:holomorphie-courbure-homogene-D-neq-D-ell}, up to linear conjugation, we can assume that $\beta\neq0$ and $b_i\neq0$ for all $i\in\{1,\ldots,n\}.$ Then, by putting $r_0:=-\frac{\alpha}{\beta}$ and~$r_i:=\frac{a_i}{b_i}$ for $i\in\{1,\ldots,n\},$ \cite[Lemma~3.5]{BM22arxiv} implies the existence of a constant $c\in\C^*$ such that $$-A(1,z)=p_0B(1,z)-c\prod_{i=0}^{n}(z-r_i)^{\nu_i}.$$

\noindent Since $A,B\in\C[x,y]_{d-1},$ the differential equation~(\ref{equa:LegH}) describing $\Leg\preh$ in the affine chart $(p,q)$ then becomes
\begin{align*}
x^{d}\Big(p-\tfrac{q}{x}-r_0\Big)\Big((p-p_0)B(1,p-\tfrac{q}{x})+c\prod_{i=0}^{n}(p-\tfrac{q}{x}-r_i)^{\nu_i}\Big)=0, \qquad \text{with} \qquad x=\frac{\mathrm{d}q}{\mathrm{d}p}.
\end{align*}
Put $\check{x}:=q$,\, $\check{y}:=p-p_0$\, and \,$\check{p}:=\dfrac{\mathrm{d}\check{y}}{\mathrm{d}\check{x}}=\dfrac{1}{x}$; in this new coordinate system $D_\ell=\{\check{y}=0\}$ and $\Leg\preh=\Leg\ell\boxtimes\Leg\mathcal{H}$ is given by the differential equation $(\check{y}+p_0-\check{p}\check{x}-r_0)F(\check{x},\check{y},\check{p})=0$, where
\begin{align*}
F(\check{x},\check{y},\check{p})=\check{y}B(1,\check{y}+p_0-\check{p}\check{x})+c\prod_{i=0}^{n}(\check{y}+p_0-\check{p}\check{x}-r_i)^{\nu_i}.
\end{align*}
We have $F(\check{x},0,\check{p})=c(-\check{x})^{d-1}\prod_{i=0}^{n}\big(\check{p}-\varphi_{i}(\check{x})\big)^{\nu_i}$, where $\varphi_{i}(\check{x})=\frac{p_0-r_i}{\check{x}}.$ Furthermore the radial foliation $\Leg\ell$ is described by $\check{\omega}_0=\mathrm{d}\check{y}-\big(\varphi_{0}(\check{x})+\frac{\check{y}}{\check{x}}\big)\mathrm{d}\check{x}$; in particular we have $D_\ell\subset\Tang(\Leg\ell,\Leg\mathcal{H}).$ Note that if $\nu_i\geq2$, then $\partial_{\check{y}}F\big(\check{x},0,\varphi_{i}(\check{x})\big)=B(1,r_i)\neq0$; since $\partial_{\check{p}}F\big(\check{x},0,\varphi_{i}(\check{x})\big)\not\equiv0$ if $\nu_i=1$, it follows that the surface
\begin{align*}
S_{\Leg\mathcal{H}}:=\left\{(\check{x},\check{y},\check{p})\in\mathbb{P}\mathrm{T}^{*}\pd\hspace{1mm}\vert\hspace{1mm}F(\check{x},\check{y},\check{p})=0\right\}
\end{align*}
is smooth along $D_\ell=\{\check{y}=0\}.$ Therefore, according to Theorem~\ref{thm-critere-holomorphie-FW}, the curvature of $\Leg\preh$ is holomorphic on~$D_\ell=\{\check{y}=0\}$ if and only if
\begin{small}
\begin{align*}
&\hspace{-5cm}
(\nu_0+1)\varphi_0(\check{x})h(\check{x})+
\sum_{i=1}^{n}(\nu_{i}-1)\varphi_{i}(\check{x})\left(\psi_{i}(\check{x})+\frac{3\varphi_{0}(\check{x})}{\varphi_{0}(\check{x})-\varphi_{i}(\check{x})}\right)\equiv0
\\
&\hspace{-6.14cm}{\fontsize{11}{11pt}\text{and}}
\\
&\hspace{-5cm}
\frac{\mathrm{d}}{\mathrm{d}\check{x}}\left((\nu_0+1)h(\check{x})+
\sum_{i=1}^{n}(\nu_{i}-1)\left(\psi_{i}(\check{x})+\frac{3\varphi_{0}(\check{x})}{\varphi_{0}(\check{x})-\varphi_{i}(\check{x})}\right)\right)\equiv0,
\end{align*}
\end{small}
\hspace{-1mm}where
\begin{SMALL}
\begin{align*}
&\hspace{0.6cm}
h(\check{x})=\frac{1}{\nu_{0}}
\left[
(\nu_{0}-1)\left(d-1-\varphi_{0}(\check{x})\dfrac{\partial_{\check{p}}\partial_{\check{y}}F\big(\check{x},0,\varphi_{0}(\check{x})\big)
-2c\hspace{0.2mm}\delta_{\nu_0,2}(-\check{x})^{d-2}\prod_{j=1}^{n}\big(\varphi_{0}(\check{x})-\varphi_{j}(\check{x})\big)^{\nu_j}}{\partial_{\check{y}}F\big(\check{x},0,\varphi_{0}(\check{x})\big)}\right)
-(2\nu_{0}+1)\sum\limits_{j=1}^{n}\frac{\nu_{j}\varphi_{j}(\check{x})}{\varphi_{0}(\check{x})-\varphi_{j}(\check{x})}
\right]
\end{align*}
\end{SMALL}
\hspace{-1mm}and, for all $i\in\{1,\ldots,n\}$ such that $\nu_i\geq2,$
\begin{SMALL}
\begin{align*}
&\hspace{-3.8cm}
\psi_{i}(\check{x})=\frac{1}{\nu_{i}}
\left[
(\nu_{i}-2)\left(d-1-\varphi_{i}(\check{x})\dfrac{\partial_{\check{p}}\partial_{\check{y}}F\big(\check{x},0,\varphi_{i}(\check{x})\big)}{\partial_{\check{y}}F\big(\check{x},0,\varphi_{i}(\check{x})\big)}\right)
-2(\nu_{i}+1)\hspace{-3.5mm}\sum\limits_{\hspace{3.5mm}j=0,j\neq i}^{n}\frac{\nu_{j}\varphi_{j}(\check{x})}{\varphi_{i}(\check{x})-\varphi_{j}(\check{x})}
\right].
\end{align*}
\end{SMALL}
\hspace{-1mm}Now, if $\nu_i\geq2$ then
$\partial_{\check{p}}\partial_{\check{y}}F\big(\check{x},0,\varphi_{i}(\check{x})\big)=-\check{x}\Big(\partial_{y}B(1,r_i)+2c\hspace{0.2mm}\delta_{\nu_i,2}\prod\limits_{j=0,j\neq i}^{n}\big(r_i-r_j\big)^{\nu_{j}}\Big).$ From this we deduce that
\begin{Small}
\begin{align*}
&\hspace{-2.6cm}
h(\check{x})=h_0:=\frac{1}{\nu_{0}}
\left[
(\nu_{0}-1)\left(d-1+\frac{(p_0-r_0)\partial_{y}B(1,r_0)}{B(1,r_0)}\right)
+(2\nu_{0}+1)\sum\limits_{j=1}^{n}\frac{\nu_{j}(p_0-r_j)}{r_0-r_j}
\right]
\\
&\hspace{-3.85cm}{\fontsize{11}{11pt}\text{and}}
\\
&\hspace{-2.6cm}
\psi_{i}(\check{x})=\psi_i:=\frac{1}{\nu_{i}}
\left[
(\nu_{i}-2)\left(d-1+\frac{(p_0-r_i)\partial_{y}B(1,r_i)}{B(1,r_i)}\right)
+2(\nu_{i}+1)\hspace{-3.5mm}\sum\limits_{\hspace{3.5mm}j=0,j\neq i}^{n}\frac{\nu_{j}(p_0-r_j)}{r_i-r_j}
\right].
\end{align*}
\end{Small}
\hspace{-1mm}Thus $K(\Leg\preh)$ is holomorphic along $D_\ell=\{\check{y}=0\}$ if and only if
\begin{align*}
(\nu_0+1)(p_0-r_0)h_0+\sum_{i=1}^{n}(\nu_i-1)(p_0-r_i)\left(\psi_i+\tfrac{3(p_0-r_0)}{r_i-r_0}\right)=0.
\end{align*}

\noindent Moreover, we have (\emph{cf.} proof of~\cite[Theorem~3.1]{BM22arxiv})
\begin{SMALL}
\begin{align*}
&
\sum\limits_{j=1}^{n}\frac{\nu_{j}(p_0-r_j)}{r_0-r_j}
=\frac{\left|\begin{array}{cc}
\partial_{x}P_0(1,r_0;-r_0,1) &  A(1,r_0)
\vspace{2mm}
\\
\partial_{y}P_0(1,r_0;-r_0,1) &  B(1,r_0)
\end{array}\right|}{B(1,r_0)P_0(1,r_0;-r_0,1)},
&&
\hspace{-3.5mm}\sum\limits_{\hspace{3.5mm}j=0,j\neq i}^{n}\frac{\nu_{j}(p_0-r_j)}{r_i-r_j}
=\frac{\left|\begin{array}{cc}
\partial_{x}P_i(1,r_i;r_i,1) &  A(1,r_i)
\vspace{2mm}
\\
\partial_{y}P_i(1,r_i;r_i,1) &  B(1,r_i)
\end{array}\right|}{B(1,r_i)P_i(1,r_i;r_i,1)}\quad \big({\fontsize{10}{10pt}\text{for }} i=1,\ldots,n\big)
\end{align*}
\end{SMALL}
\hspace{-1mm}and, for all $i\in\{0,\ldots,n\}$ such that $\nu_i\geq2,$
\begin{align*}
(d-1)B(1,r_i)+(p_0-r_i)\partial_{y}B(1,r_i)=\partial_{x}B(1,r_i)-\partial_{y}A(1,r_i).
\end{align*}
By the definition of the polynomials $Q_i$'s, it follows that
\begin{align*}
h_0=\frac{Q_{0}(1,r_0;-r_0,1)}{\nu_{0}B(1,r_0)P_{0}(1,r_0;-r_0,1)}
\qquad\hspace{2mm}\text{and}\qquad\hspace{2mm}
\psi_i=\frac{Q_{i}(1,r_i;r_i,1)}{\nu_{i}B(1,r_i)P_{i}(1,r_i;r_i,1)}.
\end{align*}

\noindent As a consequence, $K(\Leg\preh)$ is holomorphic on $D_\ell=\{\check{y}=0\}$ if and only if
\begin{small}
\begin{align*}
\left(1+\frac{1}{\nu_0}\right)\frac{(p_0-r_0)Q_{0}(1,r_0;-r_0,1)}{B(1,r_0)P_{0}(1,r_0;-r_0,1)}
+\sum_{i=1}^{n}\left(1-\frac{1}{\nu_i}\right)\left(p_0-r_i\right)
\left(\frac{Q_{i}(1,r_i;r_i,1)}{B(1,r_i)P_{i}(1,r_i;r_i,1)}+\frac{3\nu_{i}(p_0-r_0)}{r_i-r_0}\right)=0.
\end{align*}
\end{small}
\hspace{-1mm}This ends the proof of the theorem.
\end{proof}

\begin{cor}\label{cor:holomorphie-courbure-homogene-D-ell-fibre-sans-point-critique-non-fixe}
{\sl
Let $\preh=\ell\boxtimes\mathcal{H}$ be a homogeneous pre-foliation of co-degree $1$ and degree $d\geq3$ on $\pp$, defined by the $1$-form
\begin{align*}
&\hspace{1.5cm}\omega=(\alpha\,x+\beta\,y)\left(A(x,y)\mathrm{d}x+B(x,y)\mathrm{d}y\right),\quad A,B\in\mathbb{C}[x,y]_{d-1},\hspace{2mm}\gcd(A,B)=1.
\end{align*}
Assume that the line $\ell$ is not invariant by $\mathcal{H}$ and that the fiber $\Gunderline_{\mathcal{H}}^{-1}\big(\Gunderline_{\mathcal{H}}([-\alpha:\beta])\big)$ does not contain any non-fixed critical point of $\Gunderline_{\mathcal H}.$ Then the curvature of $\mathrm{Leg}\preh$ is holomorphic on $D_{\ell}=\mathcal{G}_{\mathcal{H}}(\ell)$ if and only~if $Q(\beta,-\alpha;\alpha,\beta)=0,$ where
\begin{align}\label{equa:polynome-Q}
&
Q(x,y;\alpha,\beta):=\left| \begin{array}{cc}
\dfrac{\partial{P}}{\partial{x}} &  A(\beta,-\alpha)
\vspace{2mm}
\\
\dfrac{\partial{P}}{\partial{y}} &  B(\beta,-\alpha)
\end{array}\right|
&&\text{and}&&
P(x,y;\alpha,\beta):=\frac{\left|
\begin{array}{cc}
A(x,y)  &  A(\beta,-\alpha)
\\
B(x,y)  &  B(\beta,-\alpha)
\end{array}
\right|}{\alpha\,x+\beta\,y}.
\end{align}
}
\end{cor}

\begin{rem}\label{rem:Q(b,-a,a,b)}
In particular, in degree $d=3$, the curvature of $\mathrm{Leg}\preh$ is holomorphic along $D_{\ell}$ if and only if the line with equation~$A(\beta,-\alpha)x+B(\beta,-\alpha)y=0$ is invariant by $\mathcal{H}$, or equivalently, if and only if $\Gunderline_{\mathcal{H}}\big(\Gunderline_{\mathcal{H}}([-\alpha:\beta])\big)=\Gunderline_{\mathcal{H}}([-\alpha:\beta]).$

\noindent Indeed, putting $a=A(\beta,-\alpha)$, $b=B(\beta,-\alpha)$ and $P(x,y;\alpha,\beta)=f(\alpha,\beta)x+g(\alpha,\beta)y$ we obtain
\begin{align*}
Q(\beta,-\alpha\hspace{0.2mm};\alpha,\beta)=f(\alpha,\beta)b-g(\alpha,\beta)a=P(b,-a\hspace{0.2mm};\alpha,\beta)
=-\frac{bA(b,-a)-aB(b,-a)}{\beta\,a-\alpha\,b}
=-\frac{\Cinv\left(b,-a\right)}{\Cinv(\beta,-\alpha)},
\end{align*}
where $\Cinv=x\hspace{0.2mm}A+yB$\label{not:C-H} denotes the tangent cone of $\mathcal{H}$ at the origin $O,$ \emph{see}~\cite[Section~2]{BM18Bull}.
\end{rem}

\vspace{2.1cm}
\noindent Combining Corollaries~\ref{cor:platitude-homogene-convexe-ell-non-invariante-non-effectif}~and~\ref{cor:holomorphie-courbure-homogene-D-ell-fibre-sans-point-critique-non-fixe}, we obtain:
\begin{cor}\label{cor:platitude-homogene-convexe-ell-non-invariante-effectif}
{\sl
Let $\preh=\ell\boxtimes\mathcal{H}$ be a homogeneous pre-foliation of co-degree $1$ and degree $d\geq3$ on $\pp$, defined by the $1$-form
\begin{align*}
&\hspace{1.5cm}\omega=(\alpha\,x+\beta\,y)\left(A(x,y)\mathrm{d}x+B(x,y)\mathrm{d}y\right),\quad A,B\in\mathbb{C}[x,y]_{d-1},\hspace{2mm}\gcd(A,B)=1.
\end{align*}
Assume that the homogeneous foliation $\mathcal{H}$ is convex and that the line $\ell$ is not invariant by $\mathcal{H}.$ Then the $d$-web $\mathrm{Leg}\preh$ is flat if and only if $Q(\beta,-\alpha;\alpha,\beta)=0,$ where $Q$ is the polynomial given by (\ref{equa:polynome-Q}).
}
\end{cor}

\noindent Here are two examples that will be useful in Section~\S\ref{sec:pre-pre-feuill-convexe-reduit}.
\begin{eg}\label{eg:H0}
Let us consider the homogeneous foliation $\mathcal{H}_{0}^{d-1}$ defined in the affine chart $z=1$ by the~$1$-form\label{not:omega0}
$$\omega_{\hspace{0.2mm}0}^{\hspace{0.2mm}d-1}=(d-2)y^{d-1}\mathrm{d}x+x\left(x^{d-2}-(d-1)y^{d-2}\right)\mathrm{d}y.$$
We know from~\cite[Example~6.5]{BM18Bull} that $\mathcal{H}_{0}^{d-1}$ is convex, of type $1\cdot\mathrm{R}_{d-2}+(d-2)\cdot\mathrm{R}_{1}$ and with inflection divisor\label{not:Inflex-F}\label{not:Inflex-Invariant-F}
\[
\mathrm{I}_{\mathcal{H}_{0}^{d-1}}
=\mathrm{I}_{\mathcal{H}_{0}^{d-1}}^{\hspace{0.2mm}\mathrm{inv}}
=-(d-1)(d-2)x\hspace{0.2mm}z\hspace{0.2mm}y^{d-1}\big(y^{d-2}-x^{d-2}\big)^2.
\]
If $\ell$ is one of the invariant lines of $\mathcal{H}_{0}^{d-1}$, {\it i.e.} if $\ell\in\{xyz(y-\zeta^{k} x)=0,\,k=0,\ldots,d-3\},$ where $\zeta=\mathrm{exp}\left(\frac{2\mathrm{i}\pi}{d-2}\right),$ then the $d$-web $\Leg(\ell\boxtimes\mathcal{H}_{0}^{d-1})$ is flat by Corollary~\ref{cor:platitude-pre-feuilletage-homogene-convexe}.
\smallskip

\noindent If $\ell=(y-\rho x=0)$ is not invariant by $\mathcal{H}_{0}^{d-1}$, then the $d$-web $\Leg(\ell\boxtimes\mathcal{H}_{0}^{d-1})$ is flat if and only if $\rho^{d-2}=\frac{1}{2(d-2)}$, {\it i.e.} if and only if $\ell\in\left\{y-\rho_{0}\zeta^{k}\,x=0,\,k=0,\ldots,d-3\right\},$ where $\rho_{0}=\sqrt[d-2]{\frac{1}{2\,d-4}}.$ Indeed, with the notations of Corollary~\ref{cor:holomorphie-courbure-homogene-D-ell-fibre-sans-point-critique-non-fixe}, we have
\begin{Small}
\begin{align*}
&
Q(x,y;-\rho,1)=\left(1-(d-1)\rho^{d-2}\right)\frac{\partial{P}}{\partial{x}}-(d-2)\rho^{d-1}\frac{\partial{P}}{\partial{y}}
\\
&\hspace{-1.17cm}{\fontsize{11}{11pt}\text{and}}
\\
&
P(x,y;-\rho,1)=-(d-2)\left((d-1)(\rho y)^{d-2}-\frac{y^{d-1}-(\rho x)^{d-1}}{y-\rho x}\right)
=-(d-2)\left((d-1)(\rho y)^{d-2}-\sum_{i=0}^{d-2}\rho^ix^iy^{d-2-i}\right),
\end{align*}
\end{Small}
\hspace{-1mm}so that, according to Corollary~\ref{cor:platitude-homogene-convexe-ell-non-invariante-effectif}, the flatness of $\Leg(\ell\boxtimes\mathcal{H}_{0}^{d-1})$ is characterized by
\begin{Small}
\begin{align*}
0=Q(1,\rho;-\rho,1)=\frac{1}{2}(d-1)(d-2)^2\rho^{d-2}\left(\rho^{d-2}-1\right)\left((2\,d-4)\rho^{d-2}-1\right)
\Longleftrightarrow\rho^{d-2}=\frac{1}{2(d-2)}.
\end{align*}
\end{Small}
\hspace{-1.35mm}In all cases, for any line $\ell\subset\pp$ such that $O\in\ell$ or~$\ell=L_\infty,$ the $d$-web $\Leg(\ell\boxtimes\mathcal{H}_{0}^{d-1})$ is flat if and only if, up~to~linear conjugation, $\ell=L_\infty$ or $\ell\in\left\{xy(y-x)(y-\rho_{0}\,x)=0\right\}.$ Indeed, putting $\varphi(x,y)=(x,\zeta^{k}y)$, we have
\begin{align*}
\varphi^*\left(\big(y-\zeta^{k}x\big)\omega_{0}^{d-1}\right)=\zeta^{2k}\big(y-x\big)\omega_{0}^{d-1}
&&\hspace{2mm}\text{and}\hspace{2mm}&&
\varphi^*\left(\big(y-\rho_{0}\zeta^{k}\,x\big)\omega_{0}^{d-1}\right)=\zeta^{2k}\big(y-\rho_{0}\,x\big)\omega_{0}^{d-1}.
\end{align*}
\end{eg}

\begin{eg}\label{eg:H4}
For $d\geq4$, let $\mathcal{H}_{4}^{d-1}$ be the homogeneous foliation defined in the affine chart $z=1$ by the~$1$-form\label{not:omega4}
\begin{align*}
\hspace{7mm}
\omega_{\hspace{0.2mm}4}^{\hspace{0.2mm}d-1}=y(\sigma_d\,x^{d-2}-y^{d-2})\diffx+x(\sigma_d\,y^{d-2}-x^{d-2})\diffy,
\quad\quad\text{where}\hspace{1mm}\sigma_d=1+\tfrac{2}{d-3}.
\end{align*}
This foliation is convex of type $(d-2)\cdot\mathrm{R}_2$; indeed, a straightforward computation shows that
\[
\mathrm{I}_{\mathcal{H}_{4}^{d-1}}
=\mathrm{I}_{\mathcal{H}_{4}^{d-1}}^{\hspace{0.2mm}\mathrm{inv}}
=\sigma_d(\sigma_d-1)xyz\big(x^{d-2}+y^{d-2}\big)^3.
\]
\vspace{1mm}

\noindent Let $\ell$ be a line of $\pp$ such that $O\in\ell$ or $\ell=L_\infty.$ If $\ell$ is invariant by $\mathcal{H}_{4}^{d-1},$ then Corollary~\ref{cor:platitude-pre-feuilletage-homogene-convexe} ensures that~$\Leg(\ell\boxtimes\mathcal{H}_{4}^{d-1})$ is flat, and we have $\ell\in\{xyz(y-\xi^{2k+1}\,x)=0,\,k=0,\ldots,d-3\},$ where $\xi=\mathrm{exp}\left(\frac{\mathrm{i}\pi}{d-2}\right).$

\noindent If $\ell$ is not invariant by $\mathcal{H}_{4}^{d-1}$, then $\ell=\{y-\rho\,x=0\}$ with $\rho(\rho^{d-2}+1)\neq0$; by applying Corollary~\ref{cor:platitude-homogene-convexe-ell-non-invariante-effectif}, we~obtain that the $d$-web~$\Leg(\ell\boxtimes\mathcal{H}_{4}^{d-1})$ is flat if and only if
\[
0=Q(1,\rho;-\rho,1)=-\sigma_{d}(d-2)(\rho^{d-2}+1)^2(\rho^{d-2}-1),
\]
hence if and only if $\rho^{d-2}=1,$ which is equivalent to $\ell\in\{y-\xi^{2k}x=0,\,k=0,\ldots,d-3\}.$
\smallskip

\noindent Note that, in all cases, the $d$-web $\Leg(\ell\boxtimes\mathcal{H}_{4}^{d-1})$ is flat if and only if, up~to~linear conjugation, $\ell=L_\infty$ or~$\ell\in\left\{x(y-x)(y-\xi\,x)=0\right\}.$ Indeed, putting $\varphi(x,y)=(y,x)$ and $\psi(x,y)=(x,\xi^{2k}y),$ we have
\begin{small}
\begin{align*}
\varphi^*(y\omega_{4}^{d-1})=x\omega_{4}^{d-1},&&\hspace{3mm}
\psi^*\left(\big(y-\xi^{2k}x\big)\omega_{4}^{d-1}\right)=\xi^{4k}\big(y-x\big)\omega_{4}^{d-1},&&\hspace{3mm}
\psi^*\left(\big(y-\xi^{2k+1}\,x\big)\omega_{4}^{d-1}\right)=\xi^{4k}\big(y-\xi\,x\big)\omega_{4}^{d-1}.
\end{align*}
\end{small}
\end{eg}

\begin{cor}\label{cor:holomorphie-ell-droite-inflexion-ordre-nu-1}
{\sl Let $d\geq3$ be an integer and let $\mathcal{H}$ be a homogeneous foliation of degree $d-1$ on $\pp$ defined by the $1$-form
\begin{align*}
&\hspace{1.5cm}\omega=A(x,y)\mathrm{d}x+B(x,y)\mathrm{d}y,\quad A,B\in\mathbb{C}[x,y]_{d-1},\hspace{2mm}\gcd(A,B)=1.
\end{align*}
Assume that $\mathcal{H}$ admits a transverse inflection line $\ell=(\alpha\,x+\beta\,y=0)$ of order $\nu-1$. Assume moreover that $[-\alpha:\beta]\in\mathbb{P}^{1}_{\mathbb{C}}$ is the only non-fixed critical point of $\Gunderline_{\mathcal H}$ in its fiber $\Gunderline_{\mathcal H}^{-1}(\Gunderline_{\mathcal H}([-\alpha:\beta])).$ Put $\preh:=\ell\boxtimes\mathcal{H}.$ Then the curvature of $\mathrm{Leg}\preh$ is holomorphic along $D_\ell$ if and only if $Q(\beta,-\alpha;\alpha,\beta)=0,$ where
\begin{Small}
\begin{align*}
Q(x,y;\alpha,\beta):=(\nu-1)\left(\dfrac{\partial{B}}{\partial{x}}-\dfrac{\partial{A}}{\partial{y}}\right)P(x,y;\alpha,\beta)+(2\nu+1)
\left|\begin{array}{cc}
\dfrac{\partial{P}}{\partial{x}} &  A(x,y)
\vspace{2mm}
\\
\dfrac{\partial{P}}{\partial{y}} &  B(x,y)
\end{array} \right|
\quad{\fontsize{11}{11pt}\text{and}}\quad
P(x,y;\alpha,\beta):=\frac{\left|
\begin{array}{cc}
A(x,y)  &  A(\beta,-\alpha)
\\
B(x,y)  &  B(\beta,-\alpha)
\end{array}
\right|}{(\alpha\,x+\beta\,y)^{\nu}}.
\end{align*}
\end{Small}
}
\end{cor}

\begin{cor}\label{cor:holomorphie-ell-droite-inflexion-maximal}
{\sl Let $d\geq3$ be an integer and let $\mathcal{H}$ be a homogeneous foliation of degree $d-1$ on $\pp$ defined by the $1$-form
\begin{align*}
&\hspace{1.5cm}\omega=A(x,y)\mathrm{d}x+B(x,y)\mathrm{d}y,\quad A,B\in\mathbb{C}[x,y]_{d-1},\hspace{2mm}\gcd(A,B)=1.
\end{align*}
Assume that $\mathcal{H}$ has a transverse inflection line $\ell=(\alpha\,x+\beta\,y=0)$ of maximal order $d-2$. Put $\preh:=\ell\boxtimes\mathcal{H}.$ Then the curvature of $\mathrm{Leg}\preh$ is holomorphic along $D_\ell$ if and only if the $2$-form $\mathrm{d}\omega$ vanishes on the line $\ell.$
}
\end{cor}

\begin{rem}
When $d\geq4$ the condition \og $\mathrm{d}\omega$ vanishes on the line $\ell$\fg\, also expresses the holomorphy of the curvature of $\mathrm{Leg}\mathcal{H}$ along $D_\ell$, thanks to~\cite[Theorem~3.8]{BM18Bull}. Thus Corollary~\ref{cor:holomorphie-ell-droite-inflexion-maximal} establishes the equivalence between the holomorphy on $D_\ell$ of $K(\mathrm{Leg}\preh)$ and that of $K(\mathrm{Leg}\mathcal{H}).$
\end{rem}

\vspace{3cm}

\section{Flatness and homogeneous pre-foliations $\ell\boxtimes\mathcal{H}$ of co-degree $1$ such that $\deg\mathcal{T}_{\mathcal{H}}=2$}
\label{sec:application-homogene-deg-type-2}
\bigskip

\noindent In this section we propose to classify, up to automorphism of $\pp$, all homogeneous pre-foliations $\preh=\ell\boxtimes\mathcal{H}$ of co-degree $1$ and degree $d\geq3$ on $\pp$ such that $\deg\mathcal{T}_{\mathcal{H}}=2$ and the $d$-web $\Leg\preh$ is flat. The equality $\deg\mathcal{T}_{\mathcal{H}}=2$ holds if and only if the type $\mathcal{T}_{\mathcal{H}}$ of $\mathcal{H}$ is of one of the following three forms: $2\cdot\mathrm{R}_{d-2}$, $2\cdot\mathrm{T}_{d-2}$, $1\cdot\mathrm{R}_{d-2}+1\cdot\mathrm{T}_{d-2}.$ According to~\cite[Proposition~4.1]{BM18Bull}, every homogeneous foliation of type $2\cdot\mathrm{R}_{d-2}$ is linearly conjugate to the convex foliation $\mathcal{H}_{1}^{d-1}$ defined by the $1$-form\label{not:omega1}
$$\omega_{1}^{d-1}=y^{d-1}\mathrm{d}x-x^{d-1}\mathrm{d}y.$$ The homogeneous foliations of type $2\cdot\mathrm{T}_{d-2}$, resp. $1\cdot\mathrm{R}_{d-2}+1\cdot\mathrm{T}_{d-2},$ are given, up to linear conjugation, by\label{not:omega2}\label{not:omega3}
\begin{align*}
&\omega_{2}^{d-1}(\lambda,\mu)=(x^{d-1}+\lambda\,y^{d-1})\mathrm{d}x+(\mu\,x^{d-1}-y^{d-1})\mathrm{d}y,
\quad\text{where}\hspace{1mm} \lambda,\mu\in\C, \text{with}\hspace{1mm}\lambda\mu\neq-1,
\\
\text{resp.}\hspace{1mm}&\omega_{3}^{d-1}(\lambda)=(x^{d-1}+\lambda y^{d-1})\mathrm{d}x+x^{d-1}\mathrm{d}y,
\quad\text{where}\hspace{1mm}\lambda\in\C^*,
\end{align*}
\emph{cf.}~proof~of~\cite[Proposition~4.1]{BM18Bull}. We will denote by $\mathcal{H}_{2}^{d-1}(\lambda,\mu)$, resp. $\mathcal{H}_{3}^{d-1}(\lambda)$, the foliation defined by $\omega_{2}^{d-1}(\lambda,\mu)$, resp.~$\omega_{3}^{d-1}(\lambda).$

\noindent In the following three lemmas, $\ell$ denotes a line of $\pp$ such that $O\in\ell$ or~$\ell=L_\infty.$
\begin{lem}\label{lem:platitude-Leg-ell-H1}
{\sl The $d$-web $\Leg(\ell\boxtimes\mathcal{H}_{1}^{d-1})$ is flat if and only if, up to linear conjugation, $\ell=L_\infty$ or $\ell\in\{x(y-x)(y-\xi\,x)=0\},$ where $\xi=\mathrm{exp}\left(\frac{\mathrm{i}\pi}{d-2}\right)$.}
\end{lem}

\begin{proof}[\sl Proof]
Note first of all that the foliation $\mathcal{H}_{1}^{d-1}$ has inflection divisor $$\mathrm{I}_{\mathcal{H}_{1}^{d-1}}=\mathrm{I}_{\mathcal{H}_{1}^{d-1}}^{\hspace{0.2mm}\mathrm{inv}}=(d-1)z\,x^{d-1}y^{d-1}\big(y^{d-2}-x^{d-2}\big).$$

\textbf{\textit{i.}} If $\ell$ is invariant by $\mathcal{H}_{1}^{d-1}$, then $\ell\in\{xyz(y-\xi^{2k} x)=0,\,k=0,\ldots,d-3\}$ and the $d$-web $\Leg(\ell\boxtimes\mathcal{H}_{1}^{d-1})$ is flat (Corollary~\ref{cor:platitude-pre-feuilletage-homogene-convexe}).
\smallskip

\textbf{\textit{ii.}} Assume that $\ell$ is not invariant by $\mathcal{H}_{1}^{d-1}$; then $\ell=(y-\rho x=0)$ with $\rho(\rho^{d-2}-1)\neq0.$ According to~Corollary~\ref{cor:platitude-homogene-convexe-ell-non-invariante-effectif}, the $d$-web $\Leg(\ell\boxtimes\mathcal{H}_{1}^{d-1})$ is flat if and only if $Q(1,\rho;-\rho,1)=0,$ where
\begin{align*}
Q(x,y;-\rho,1)=-\frac{\partial{P}}{\partial{x}}-\rho^{d-1}\frac{\partial{P}}{\partial{y}}
\qquad\text{and}\qquad
P(x,y;-\rho,1)=-\frac{y^{d-1}-(\rho x)^{d-1}}{y-\rho x}=-\sum_{i=0}^{d-2}\rho^i x^i y^{d-2-i}.
\end{align*}
Thus $Q(1,\rho;-\rho,1)=\frac{1}{2}(d-1)(d-2)\rho^{d-2}(\rho^{d-2}+1)$, and the flatness of $\Leg(\ell\boxtimes\mathcal{H}_{1}^{d-1})$ is equivalent to~$\rho^{d-2}=-1$ and therefore to $\ell\in\{y-\xi^{2k+1} x=0,\,k=0,\ldots,d-3\}.$
\smallskip

\noindent In the two cases considered, $\Leg(\ell\boxtimes\mathcal{H}_{1}^{d-1})$ is flat if and only if, up to conjugation, $\ell=L_\infty:=(z=0)$ or~$\ell\in\{x(y-x)(y-\xi\,x)=0\}.$ Indeed, putting $\varphi_1(x,y)=(y,x)$ and $\varphi_2(x,y)=(x,\xi^{2k}y)$, we have
\begin{small}
\begin{align*}
\varphi_1^*(y\omega_{1}^{d-1})=-x\omega_{1}^{d-1},&&\hspace{1mm}
\varphi_2^*\left((y-\xi^{2k} x)\omega_{1}^{d-1}\right)=\xi^{4k}(y-x)\omega_{1}^{d-1},&&\hspace{1mm}
\varphi_2^*\left((y-\xi^{2k+1} x)\omega_{1}^{d-1}\right)=\xi^{4k}(y-\xi\,x)\omega_{1}^{d-1}.
\end{align*}
\end{small}
\end{proof}
\vspace{-2.6mm}

\begin{lem}\label{lem:platitude-Leg-ell-H2}
{\sl The $d$-web $\Leg\big(\ell\boxtimes\mathcal{H}_{2}^{d-1}(\lambda,\mu)\big)$ is flat if and only if, up to linear conjugation, one of the following cases occurs:
\begin{itemize}
\item [(i)]   $\ell=L_\infty$ and $d=3$;

\item [(ii)]  $\ell=L_\infty,$ $d\geq4$ and $\lambda=\mu=0$;

\item [(iii)] $\ell=(x=0)$ and $\lambda=\mu=0$;

\item [(iv)]  $\ell=(y-x=0),$ $d\geq4$ and $(\lambda,\mu)=\left(\frac{3}{d},-\frac{3}{d}\right)$;

\item [(v)]   $\ell=(y-\xi'\,x=0),$ $d\geq4$ and $(\lambda,\mu)=\left(\frac{3\xi'}{d},-\frac{3}{d\xi'}\right)$, where $\xi'=\mathrm{exp}\left(\frac{\mathrm{i}\pi}{d}\right)$.
\end{itemize}
}
\end{lem}

\begin{proof}[\sl Proof]
We have \begin{small}$\omega_{2}^{d-1}(\lambda,\mu)=A(x,y)\mathrm{d}x+B(x,y)\mathrm{d}y$\end{small}, where \begin{small}$A(x,y)=x^{d-1}+\lambda\,y^{d-1}$\end{small} and \begin{small}$B(x,y)=\mu\,x^{d-1}-y^{d-1}$\end{small}; an immediate computation shows that
\begin{align*}
\mathrm{I}_{\mathcal{H}_{2}^{d-1}(\lambda,\mu)}^{\hspace{0.2mm}\mathrm{inv}}=z(x^d+\mu\,x^{d-1}y+\lambda\,xy^{d-1}-y^d)
&&\text{and}&&
\mathrm{I}_{\mathcal{H}_{2}^{d-1}(\lambda,\mu)}^{\hspace{0.2mm}\mathrm{tr}}=x^{d-2}y^{d-2}.
\end{align*}

\textbf{\textit{1.}} If $\ell=L_\infty$ and $d=3,$ then the web $\Leg(\ell\boxtimes\mathcal{H}_{2}^{2}(\lambda,\mu))$ is flat by Corollary~\ref{cor:Leg-L-infini-H2-plat}.
\vspace{2mm}

\textbf{\textit{2.}} Assume that $\ell=L_\infty$ and $d\geq4.$ Then, according to~\cite[Theorem~3.1~and~3.8]{BM18Bull}, the web $\Leg(\mathcal{H}_{2}^{d-1}(\lambda,\mu))$ is flat if and only if $\mathrm{d}\big(\omega_{2}^{d-1}(\lambda,\mu)\big)$ vanishes on the two lines $xy=0.$ Now,
\begin{align*}
\mathrm{d}\big(\omega_{2}^{d-1}(\lambda,\mu)\big)\Big|_{x=0}=-(d-1)\lambda\,y^{d-2}\mathrm{d}x\wedge\mathrm{d}y
\qquad\text{and}\qquad
\mathrm{d}\big(\omega_{2}^{d-1}(\lambda,\mu)\big)\Big|_{y=0}=(d-1)\mu\,x^{d-2}\mathrm{d}x\wedge\mathrm{d}y.
\end{align*}
Therefore $\Leg(\mathcal{H}_{2}^{d-1}(\lambda,\mu))$ is flat if and only if $\lambda=\mu=0$; hence the same holds for $\Leg(\ell\boxtimes\mathcal{H}_{2}^{d-1}(\lambda,\mu))$ (Theorem~\ref{thm:K-Leg-H-egale-K-Leg-L-infini-H}).
\vspace{2mm}

\textbf{\textit{3.}} Let us consider the case where $\ell\in\{xy=0\}.$ Up to permuting the coordinates $x$ and $y$, we can assume that
$\ell=(x=0).$ According to Theorem~\ref{thm:holomorphie-courbure-G(I^tr)-D-ell}, the $d$-web $\Leg(\ell\boxtimes\mathcal{H}_{2}^{d-1}(\lambda,\mu))$ is flat if and only if its curvature is holomorphic on $\mathcal{G}_{\mathcal{H}_{2}^{d-1}(\lambda,\mu)}(\{xy=0\}).$ Now, on the one hand, $K(\Leg(\ell\boxtimes\mathcal{H}_{2}^{d-1}(\lambda,\mu)))$ is holomorphic on $D_\ell=\mathcal{G}_{\mathcal{H}_{2}^{d-1}(\lambda,\mu)}(\ell)$ if and only if $\mathrm{d}\big(\omega_{2}^{d-1}(\lambda,\mu)\big)$ vanishes on $\ell=(x=0)$ (Corollary~\ref{cor:holomorphie-ell-droite-inflexion-maximal}), {\it i.e.} if and only if $\lambda=0.$ On the other hand, according to Corollary~\ref{cor:holomorphie-courbure-droite-inflex-maximal-d-2-T-neq-ell}, $K(\Leg(\ell\boxtimes\mathcal{H}_{2}^{d-1}(\lambda,\mu)))$ is holomorphic on $\mathcal{G}_{\mathcal{H}_{2}^{d-1}(\lambda,\mu)}(\{y=0\})$ if and only if
\[
0=(d-3)\big(\partial_{x}B(1,0)-\partial_{y}A(1,0)\big)+3(d-1)B(1,0)=d(d-1)\mu\Longleftrightarrow\mu=0.
\]
It follows that $\Leg(\ell\boxtimes\mathcal{H}_{2}^{d-1}(\lambda,\mu))$ is flat if and only if $\lambda=\mu=0.$
\vspace{2mm}

\textbf{\textit{4.}} Let us examine the case where $\ell=(y-\rho\,x=0)$ with $\rho\neq0.$ By Corollary~\ref{cor:holomorphie-courbure-droite-inflex-maximal-d-2-T-neq-ell}, $K(\Leg(\ell\boxtimes\mathcal{H}_{2}^{d-1}(\lambda,\mu)))$ is holomorphic on~$\mathcal{G}_{\mathcal{H}_{2}^{d-1}(\lambda,\mu)}(\{xy=0\})$ if and only if
\begin{Small}
\begin{align*}
\left\{
\begin{array}[l]{l}
0=-(d-3)\Big(\partial_{x}B(0,-1)-\partial_{y}A(0,-1)\Big)-3(d-1)\Big(A(0,-1)+\rho B(0,-1)\Big)=(-1)^{d}(d-1)(d\lambda-3\rho)
\\
\\
0=-\rho(d-3)\Big(\partial_{x}B(1,0)-\partial_{y}A(1,0)\Big)-3(d-1)\Big(A(1,0)+\rho B(1,0)\Big)=-(d-1)(d\rho\mu+3),
\end{array}
\right.
\end{align*}
\end{Small}
\hspace{-1mm}{\it i.e.} if and only if $\lambda=\lambda_0:=\frac{3\rho}{d}$,\, $\mu=\mu_0:=-\frac{3}{d\rho}$\, and \,$d\neq3$, because $\lambda\mu\neq-1.$ We now distinguish two cases according to whether or not $\ell$ is invariant by $\mathcal{H}_{2}^{d-1}(\lambda_0,\mu_0).$
\vspace{2mm}

\textbf{\textit{4.1.}} Assume that $\ell$ is invariant by $\mathcal{H}_{2}^{d-1}(\lambda_0,\mu_0).$ Then the dual web of $\ell\boxtimes\mathcal{H}_{2}^{d-1}(\lambda_0,\mu_0)$ is flat by~Theorem~\ref{thm:holomorphie-courbure-G(I^tr)-D-ell}. Since $\mathrm{I}_{\mathcal{H}_{2}^{d-1}(\lambda_0,\mu_0)}^{\hspace{0.2mm}\mathrm{inv}}\Big|_{y=\rho x}=\left(\frac{3}{d}-1\right)\left(\rho^d-1\right)z\,x^d,$ the invariance of $\ell$ by $\mathcal{H}_{2}^{d-1}(\lambda_0,\mu_0)$ is~equivalent~to~$\rho^d=1$; as a consequence $(\rho,\lambda_0,\mu_0)\in\left\{\Big(\xi'^{2k},\frac{3\xi'^{2k}}{d},-\frac{3}{d\xi'^{2k}}\Big),k=0,\ldots,d-1\right\}.$ Up to conjugation,~$(\rho,\lambda_0,\mu_0)=(1,\frac{3}{d},-\frac{3}{d})$; indeed, putting $\varphi(x,y)=(x,\xi'^{2k}y)$ we have
\begin{align*}
\varphi^*\left(\big(y-\xi'^{2k}x\big)\omega_{2}^{d-1}\left(\tfrac{3\xi'^{2k}}{d},-\tfrac{3}{d\xi'^{2k}}\right)\right)
=\xi'^{2k}\left(y-x\right)\omega_{2}^{d-1}\left(\tfrac{3}{d},-\tfrac{3}{d}\right).
\end{align*}
\vspace{2mm}

\textbf{\textit{4.2.}} Assume that $\ell$ is not invariant by $\mathcal{H}_{2}^{d-1}(\lambda_0,\mu_0).$ Then, by~Theorem~\ref{thm:holomorphie-courbure-G(I^tr)-D-ell} and Corollary~\ref{cor:holomorphie-courbure-homogene-D-ell-fibre-sans-point-critique-non-fixe}, the flatness of $\Leg(\ell\boxtimes\mathcal{H}_{2}^{d-1}(\lambda_0,\mu_0))$ translates into $Q(1,\rho;-\rho,1)=0,$ where
\begin{small}
\begin{align*}
&\hspace{-4cm}Q(x,y;-\rho,1)=\left(\mu_0-\rho^{d-1}\right)\frac{\partial{P}}{\partial{x}}-\left(\lambda_0\rho^{d-1}+1\right)\frac{\partial{P}}{\partial{y}}
\\
&\hspace{-5.17cm}{\fontsize{11}{11pt}\text{and}}
\\
&\hspace{-4cm}P(x,y;-\rho,1)=\frac{\left(\lambda_0\mu_0+1\right)\left(y^{d-1}-(\rho x)^{d-1}\right)}{y-\rho x}
=\left(\lambda_0\mu_0+1\right)\sum_{i=0}^{d-2}\rho^i x^i y^{d-2-i}.
\end{align*}
\end{small}
\hspace{-1mm}Hence \begin{small}$Q(1,\rho;-\rho,1)=\frac{1}{2}\left(\frac{3}{d}-1\right)\left(\frac{3}{d}+1\right)^2(d-1)(d-2)\rho^{d-3}(\rho^d+1)$\end{small}, and consequently $\Leg(\ell\boxtimes\mathcal{H}_{2}^{d-1}(\lambda_0,\mu_0))$ is flat if and only if $\rho^d=-1$, hence if and only if $(\rho,\lambda_0,\mu_0)\in\left\{\Big(\xi'^{2k+1},\frac{3\xi'^{2k+1}}{d},-\frac{3}{d\xi'^{2k+1}}\Big),k=0,\ldots,d-1\right\}$, or equivalently, if and only if, up to conjugation,  $(\rho,\lambda_0,\mu_0)=\Big(\xi',\frac{3\xi'}{d},-\frac{3}{d\xi'}\Big),$ because
\begin{align*}
\varphi^*\left(\big(y-\xi'^{2k+1}x\big)\omega_{2}^{d-1}\left(\tfrac{3\xi'^{2k+1}}{d},-\tfrac{3}{d\xi'^{2k+1}}\right)\right)
=\xi'^{2k}\left(y-\xi'x\right)\omega_{2}^{d-1}\left(\tfrac{3\xi'}{d},-\tfrac{3}{d\xi'}\right).
\end{align*}
\end{proof}

\begin{lem}\label{lem:platitude-Leg-ell-H3}
{\sl The $d$-web $\Leg\big(\ell\boxtimes\mathcal{H}_{3}^{d-1}(\lambda)\big)$ is flat if and only if one of the following cases holds:
\begin{itemize}
\item [(i)]   $\ell=L_\infty$ and $d=3$;

\item [(ii)]  $\ell=(dy+3x=0),$ $d\geq4$ and $\lambda=\dfrac{(-1)^d(d-3)d^{d-2}}{3^{d-1}}$;

\item [(iii)] $\ell=(dy+3x=0)$ and $\lambda=\dfrac{(-1)^d(d+3)d^{d-2}}{3^{d-1}}.$
\end{itemize}
}
\end{lem}

\begin{proof}[\sl Proof]
We have $\omega_{3}^{d-1}(\lambda)=A(x,y)\mathrm{d}x+B(x,y)\mathrm{d}y$, where $A(x,y)=x^{d-1}+\lambda y^{d-1}$ and $B(x,y)=x^{d-1}$; an immediate computation leads to
\begin{align*}
\mathrm{I}_{\mathcal{H}_{3}^{d-1}(\lambda)}^{\hspace{0.2mm}\mathrm{inv}}=z\,x^{d-1}\big(x^{d-1}+x^{d-2}y+\lambda\,y^{d-1}\big)
&&\text{and}&&
\mathrm{I}_{\mathcal{H}_{3}^{d-1}(\lambda)}^{\hspace{0.2mm}\mathrm{tr}}=y^{d-2}.
\end{align*}

\textbf{\textit{1.}} Assume that $\ell=L_\infty.$ If $d=3$, then the web $\Leg(\ell\boxtimes\mathcal{H}_{3}^{2}(\lambda))$ is flat, thanks to Corollary~\ref{cor:Leg-L-infini-H2-plat}. For~$d\geq4,$ the webs $\Leg(\mathcal{H}_{3}^{d-1}(\lambda))$ and $\Leg(\ell\boxtimes\mathcal{H}_{3}^{d-1}(\lambda))$ have the same curvature (Theorem~\ref{thm:K-Leg-H-egale-K-Leg-L-infini-H}) and cannot be flat. Indeed, we have $$\mathrm{d}\big(\omega_{3}^{d-1}(\lambda)\big)\Big|_{y=0}=(d-1)x^{d-2}\mathrm{d}x\wedge\mathrm{d}y\not\equiv0;$$ this implies, according to \cite[Theorem~3.8]{BM18Bull}, that $K(\Leg(\mathcal{H}_{3}^{d-1}(\lambda)))$ cannot be holomorphic along $\mathcal{G}_{\mathcal{H}_{3}^{d-1}(\lambda)}(\{y=0\}).$
\vspace{2mm}

\textbf{\textit{2.}} If $\ell=(y=0),$ then the fact that $\mathrm{d}\big(\omega_{3}^{d-1}(\lambda)\big)$ does not vanish on $\ell$ implies, by Corollary~\ref{cor:holomorphie-ell-droite-inflexion-maximal}, that~$K(\Leg(\ell\boxtimes\mathcal{H}_{3}^{d-1}(\lambda)))$ cannot be holomorphic on $\mathcal{G}_{\mathcal{H}_{3}^{d-1}(\lambda)}(\ell),$ so that $\Leg(\ell\boxtimes\mathcal{H}_{3}^{d-1}(\lambda))$ cannot be flat.
\vspace{2mm}

\textbf{\textit{3.}} Assume that $\ell=(x-\rho\,y=0),$ where $\rho\in\C.$ By Corollary~\ref{cor:holomorphie-courbure-droite-inflex-maximal-d-2-T-neq-ell}, $K(\Leg(\ell\boxtimes\mathcal{H}_{3}^{d-1}(\lambda)))$ is holomorphic on $\mathcal{G}_{\mathcal{H}_{3}^{d-1}(\lambda)}(\{y=0\})$ if and only if
\begin{small}
\begin{align*}
0=(d-3)\Big(\partial_{x}B(1,0)-\partial_{y}A(1,0)\Big)+3(d-1)\Big(B(1,0)+\rho A(1,0)\Big)=(d-1)(3\rho+d),
\end{align*}
\end{small}
\hspace{-1mm}hence if and only if $\rho=-\frac{d}{3}$, which is equivalent to $\ell=\ell_0$ where $\ell_0=(dy+3x=0).$ Then we have to distinguish two cases:
\vspace{2mm}

\textbf{\textit{3.1.}} If $\ell_0$ is invariant by $\mathcal{H}_{3}^{d-1}(\lambda)$, then Theorem~\ref{thm:holomorphie-courbure-G(I^tr)-D-ell} ensures that the $d$-web $\Leg(\ell_0\boxtimes\mathcal{H}_{3}^{d-1}(\lambda))$ is flat; since
\begin{small}
\begin{align*}
\mathrm{I}_{\mathcal{H}_{3}^{d-1}(\lambda)}^{\hspace{0.2mm}\mathrm{inv}}\Big|_{x=-\frac{d}{3} y}=-d^{d-1}z\left(\frac{y}{3}\right)^{2d-2}\left((-1)^{d}3^{d-1}\lambda-(d-3)d^{d-2}\right),
\end{align*}
\end{small}
\hspace{-1mm}the invariance of $\ell_0$ by $\mathcal{H}_{3}^{d-1}(\lambda)$ is characterized by $\lambda=\dfrac{(-1)^d(d-3)d^{d-2}}{3^{d-1}}$ and $d\neq3$, because $\lambda\neq0.$
\vspace{2mm}

\textbf{\textit{3.2.}} Assume that $\ell_0$ is not invariant by $\mathcal{H}_{3}^{d-1}(\lambda).$ Then, according to Theorem~\ref{thm:holomorphie-courbure-G(I^tr)-D-ell} and Corollary~\ref{cor:holomorphie-courbure-homogene-D-ell-fibre-sans-point-critique-non-fixe}, the~$d$-web $\Leg(\ell_0\boxtimes\mathcal{H}_{3}^{d-1}(\lambda))$ is flat if and only if $Q(d,-3;3,d)=0,$ where
\begin{small}
\begin{align*}
&\hspace{-5.1cm}Q(x,y;3,d)=d^{d-1}\frac{\partial{P}}{\partial{x}}-\left(d^{d-1}+(-3)^{d-1}\lambda\right)\frac{\partial{P}}{\partial{y}}
\\
&\hspace{-6.35cm}{\fontsize{11}{11pt}\text{and}}
\\
&\hspace{-5.1cm}P(x,y;3,d)=\frac{\lambda\left((dy)^{d-1}-(-3x)^{d-1}\right)}{dy+3x}=\lambda\sum_{i=0}^{d-2}(-3x)^i(dy)^{d-2-i}.
\end{align*}
\end{small}
\hspace{-1mm}Thus \begin{Small}$Q(d,-3;3,d)=-\frac{1}{6}\lambda(d-1)(d-2)(3d)^{d-2}\Big(3^{d-1}\lambda-(-1)^{d}(d+3)d^{d-2}\Big)$\end{Small} and the flatness of $\Leg(\ell_0\boxtimes\mathcal{H}_{3}^{d-1}(\lambda))$ translates into $\lambda=\dfrac{(-1)^d(d+3)d^{d-2}}{3^{d-1}}.$
\end{proof}

\noindent Lemmas~\ref{lem:platitude-Leg-ell-H1}, \ref{lem:platitude-Leg-ell-H2} and \ref{lem:platitude-Leg-ell-H3} imply the following proposition.
\begin{pro}\label{pro:class-pre-homogenes-plats-co-degre-1-degre-d-geq-4-degre-Type-2}
{\sl Let $\preh=\ell\boxtimes\mathcal{H}$ be a homogeneous pre-foliation of co-degree $1$ and degree $d\geq3$~on~$\pp.$ Assume that $\deg\mathcal{T}_{\mathcal{H}}=2,$ or equivalently that the map $\Gunderline_{\mathcal H}$ has exactly two critical points. Then, for~$d\geq4$ the~web $\Leg\preh$ is flat if and only if $\preh$ is linearly conjugate to one of the ten following pre-foliations
\vspace{2mm}

\begin{itemize}
\item [\texttt{1.}]\hspace{1mm} $\preh_{1}^{d}=L_\infty\boxtimes\mathcal{H}_{1}^{d-1}$;
\smallskip

\item [\texttt{2.}]\hspace{1mm} $\preh_{2}^{d}=\{x=0\}\boxtimes\mathcal{H}_{1}^{d-1}$;
\smallskip

\item [\texttt{3.}]\hspace{1mm} $\preh_{3}^{d}=\{y-x=0\}\boxtimes\mathcal{H}_{1}^{d-1}$;
\smallskip

\item [\texttt{4.}]\hspace{1mm} $\preh_{4}^{d}=\{y-\xi\,x=0\}\boxtimes\mathcal{H}_{1}^{d-1}$,\, where $\xi=\mathrm{exp}\left(\frac{\mathrm{i}\pi}{d-2}\right)$;
\smallskip

\item [\texttt{5.}]\hspace{1mm} $\preh_{5}^{d}=\{x=0\}\boxtimes\mathcal{H}_{2}^{d-1}(0,0)$;
\smallskip

\item [\texttt{6.}]\hspace{1mm} $\preh_{6}^{d}=\{dy+3x=0\}\boxtimes\mathcal{H}_{3}^{d-1}(\lambda_0)$,\, where $\lambda_0=\frac{(-1)^d(d+3)d^{d-2}}{3^{d-1}}$;
\smallskip

\item [\texttt{7.}]\hspace{1mm} $\preh_{7}^{d}=\{dy+3x=0\}\boxtimes\mathcal{H}_{3}^{d-1}(\lambda_1)$,\, where $\lambda_1=\frac{(-1)^d(d-3)d^{d-2}}{3^{d-1}}$;
\smallskip

\item [\texttt{8.}]\hspace{1mm} $\preh_{8}^{d}=L_\infty\boxtimes\mathcal{H}_{2}^{d-1}(0,0)$;

\item [\texttt{9.}]\hspace{1mm} $\preh_{9}^{d}=\{y-x=0\}\boxtimes\mathcal{H}_{2}^{d-1}\big(\tfrac{3}{d},-\tfrac{3}{d}\big)$;
\smallskip

\item [\texttt{10.}]            $\preh_{10}^{d}=\{y-\xi'\,x=0\}\boxtimes\mathcal{H}_{2}^{d-1}\big(\tfrac{3\xi'}{d},-\tfrac{3}{d\xi'}\big)$,\, where
                                   $\xi'=\mathrm{exp}\left(\frac{\mathrm{i}\pi}{d}\right)$.
\end{itemize}
\vspace{2mm}

\noindent For $d=3$ the web $\Leg\preh$ is flat if and only if, up to linear conjugation, either $\preh$ is one of the six pre-foliations $\preh_{1}^{3},\preh_{2}^{3},\ldots,\preh_{6}^{3},$ or $\preh$ is of one of the following two types
\vspace{2mm}

\begin{itemize}
\item [\texttt{11.}] $\preh_{7}^{3}(\lambda)=L_\infty\boxtimes\mathcal{H}_{3}^{2}(\lambda)$,\, where $\lambda\in\C^*$;
\smallskip

\item [\texttt{12.}] $\preh_{8}^{3}(\lambda,\mu)=L_\infty\boxtimes\mathcal{H}_{2}^{2}(\lambda,\mu)$,\, where $\lambda,\mu\in\C$ with $\lambda\mu\neq-1.$
\end{itemize}
}
\end{pro}

\vspace{2mm}

\noindent Combining Proposition~\ref{pro:class-pre-homogenes-plats-co-degre-1-degre-d-geq-4-degre-Type-2} with the fact that every homogeneous foliation of degree $2$ on $\pp$ has degree of~type $2$, we obtain the classification, up to automorphism of $\pp$, of homogeneous pre-foliations of type~$(1,3)$ on $\pp$ whose dual web is flat.
\begin{cor}\label{cor:class-pre-homogenes-plats-co-degre-1-degre-3}
{\sl Up to automorphism of $\pp$, there are six examples and two families of homogeneous pre-foliations of co-degree $1$ and degree $3$ on $\mathbb {P}^{2}_{\mathbb{C}}$ with a flat \textsc{Legendre} transform, namely:
\begin{itemize}
\item [\texttt{1.}]\hspace{1mm} $\preh_{1}^{3}=L_\infty\boxtimes\mathcal{H}_{1}^{2}$;
\smallskip

\item [\texttt{2.}]\hspace{1mm} $\preh_{2}^{3}=\{x=0\}\boxtimes\mathcal{H}_{1}^{2}$;
\smallskip

\item [\texttt{3.}]\hspace{1mm} $\preh_{3}^{3}=\{y-x=0\}\boxtimes\mathcal{H}_{1}^{2}$;
\smallskip

\item [\texttt{4.}]\hspace{1mm} $\preh_{4}^{3}=\{y+x=0\}\boxtimes\mathcal{H}_{1}^{2}$;
\smallskip

\item [\texttt{5.}]\hspace{1mm} $\preh_{5}^{3}=\{x=0\}\boxtimes\mathcal{H}_{2}^{2}(0,0)$;
\smallskip

\item [\texttt{6.}]\hspace{1mm} $\preh_{6}^{3}=\{y+x=0\}\boxtimes\mathcal{H}_{3}^{2}(-2)$;
\smallskip

\item [\texttt{7.}]\hspace{1mm} $\preh_{7}^{3}(\lambda)=L_\infty\boxtimes\mathcal{H}_{3}^{2}(\lambda)$,\, where $\lambda\in\C^*$;
\smallskip

\item [\texttt{8.}]\hspace{1mm} $\preh_{8}^{3}(\lambda,\mu)=L_\infty\boxtimes\mathcal{H}_{2}^{2}(\lambda,\mu)$,\, where $\lambda,\mu\in\C$ with $\lambda\mu\neq-1.$
\end{itemize}
}
\end{cor}

\noindent In Section~\S\ref{sec:pre-feuilletages-codegre-1-degre-3} we will need, for $\mathcal{H}\in\{\mathcal{H}_{1}^{2},\mathcal{H}_{2}^{2}(0,0),\mathcal{H}_{3}^{2}(-2)\}$, the values of the \textsc{Camacho-Sad}
indices $\mathrm{CS}(\mathcal{H},L_{\infty},s)$,\label{not:indice-CS} $s\in\Sing\mathcal{H}\cap L_{\infty}$. For this, we have computed, for each of these three foliations, the following polynomial (called \textsl{\textsc{Camacho-Sad} polynomial of the homogeneous foliation} $\mathcal{H}$)
\begin{align*}
\mathrm{CS}_{\mathcal{H}}(\lambda)=\prod\limits_{s\in\Sing\mathcal{H}\cap L_{\infty}}(\lambda-\mathrm{CS}(\mathcal{H},L_{\infty},s)).
\end{align*}

\noindent The following table summarizes the types and the \textsc{Camacho-Sad} polynomials of the foliations $\mathcal{H}_{1}^{2}$, $\mathcal{H}_{2}^{2}(0,0)$ and~$\mathcal{H}_{3}^{2}(-2)$.

\begingroup
\renewcommand*{\arraystretch}{1.5}
\begin{table}[h]
\begin{center}
\begin{tabular}{|c|c|c|}\hline
$\mathcal{H}$               & $\mathcal{T}_{\mathcal{H}}$             &  $\mathrm{CS}_{\mathcal{H}}(\lambda)$
\\\hline
$\mathcal{H}_{1}^{2}$       & $2\cdot\mathrm{R}_1$                    &  $(\lambda-1)^{2}(\lambda+1)$
\\\hline
$\mathcal{H}_{2}^{2}(0,0)$  & $2\cdot\mathrm{T}_1$                    &  $(\lambda-\frac{1}{3})^{3}$
\\\hline
$\mathcal{H}_{3}^{2}(-2)$   & $1\cdot\mathrm{R}_1+1\cdot\mathrm{T}_1$ & $(\lambda-1)(\lambda-\frac{1}{3})(\lambda+\frac{1}{3})$
\\\hline
\end{tabular}
\end{center}
\bigskip
\caption{Types and \textsc{Camacho-Sad} polynomials of the foliations $\mathcal{H}_{1}^{2}$, $\mathcal{H}_{2}^{2}(0,0)$ and $\mathcal{H}_{3}^{2}(-2)$.}
\label{tab:CS(lambda)}
\end{table}
\endgroup

\section{Pre-foliations of co-degree $1$ whose associated foliation is reduced convex}\label{sec:pre-pre-feuill-convexe-reduit}
\bigskip

\noindent We now give the proofs of Theorem~\ref{thmalph:ell-invariante-convexe-reduit-plat} and Propositions~\ref{proalph:ell-non-invariante-Fermat}~and~\ref{proalph:ell-non-invariante-Hesse} stated in the Introduction.
\begin{proof}[\sl Proof of Theorem~\ref{thmalph:ell-invariante-convexe-reduit-plat}]
Since by hypothesis $\F$ is reduced convex, all its singularities are non-degenerate (\cite[Lemma~6.8]{BM18Bull}). According to~\cite[Lemma~2.2]{BFM14}, the discriminant of $\Leg\F$ then consists of the lines dual to~the radial singularities of $\F.$ The first assertion of Lemma~\ref{lem:Delta-Leg-ell-F} therefore implies that
\begin{align*}
\Delta(\Leg\pref)=\check{\Sigma}_{\F}^{\mathrm{rad}}\cup\check{\Sigma}_{\F}^{\ell}.
\end{align*}
To show that the curvature of $\Leg\pref$ is identically zero, it suffices therefore to show that it is holomorphic along the dual line of every point of $\Sigma_{\F}^{\mathrm{rad}}\cup\Sigma_{\F}^{\ell}.$ Let $s$ be an arbitrary point of $\Sigma_{\F}^{\mathrm{rad}}\cup\Sigma_{\F}^{\ell}.$ Denote by $\nu=\tau(\F,s)$\label{not:tau-F-s} the tangency order of $\F$ with a generic line passing through $s$; then $\nu-1$ denotes the radiality order~of~$s$, and $s\in\Sigma_{\F}^{\mathrm{rad}}$ if and only if $\nu\geq2$, \emph{see}~\cite[\S1.3]{BM18Bull}.
By \cite[Proposition~3.3]{MP13}, locally near the line $\check{s}$ dual to $s$, we can decompose $\Leg\F$ as~$\Leg\F=\W_{\nu}\boxtimes\W_{d-\nu-1},$ where $\W_{\nu}$ is an irreducible $\nu$-web leaving $\check{s}$ invariant and whose discriminant $\Delta(\W_{\nu})$ has minimal multiplicity $\nu-1$ along $\check{s}$, and where $\W_{d-\nu-1}$ is a $(d-\nu-1)$-web transverse to $\check{s}.$ Furthermore, the convexity of $\F$ implies, by an argument of the proof of \cite[Theorem~4.2]{MP13}, that the web $\W_{d-\nu-1}$ is regular near $\check{s},$ {\it i.e.} that through a generic point of $\check{s}$ pass $(d-\nu-1)$ distinct tangent lines to $\W_{d-\nu-1}.$

\noindent Thus, near the line $\check{s},$ we have the decomposition
\begin{align*}
\Leg\pref=\Leg\ell\boxtimes\W_{\nu}\boxtimes\W_{d-\nu-1}.
\end{align*}
We now distinguish two cases:
\vspace{2mm}

\textbf{\textit{1.}} If $s\in\ell$ then $\check{s}$ is invariant by $\Leg\ell$; by applying Theorem~1 of \cite{MP13} if $\nu=1$ and Proposition~\ref{pro:holomorphie-courbure-F-W-nu-W-d-nu-1} if $\nu\geq2,$ it follows that $K(\Leg\pref)$ is holomorphic along $\check{s}.$
\vspace{2mm}

\textbf{\textit{2.}} Assume that $s\not\in\ell$; then $s\in\Sigma_{\F}^{\mathrm{rad}}\setminus\Sigma_{\F}^{\ell}.$ In this case the radial foliation $\Leg\ell$ is transverse to $\check{s}.$ From~the~above~discussion, the $(d-\nu)$-web $\W_{d-\nu}:=\Leg\ell\boxtimes\W_{d-\nu-1}$ is therefore also transverse to $\check{s}$ and~we~have $\Leg\pref=\W_{\nu}\boxtimes\W_{d-\nu}.$ Moreover, since $\ell$ is $\F$-invariant, $\Tang(\Leg\ell,\Leg\F)=\check{\Sigma}_{\F}^{\ell}$ (\emph{cf.} proof of Lemma~\ref{lem:Delta-Leg-ell-F}); in particular, $\Tang(\Leg\ell,\W_{d-\nu-1})\subset\check{\Sigma}_{\F}^{\ell}$ and therefore $\check{s}\not\subset\Tang(\Leg\ell,\W_{d-\nu-1}).$ It follows that the web $\W_{d-\nu}$ is regular near $\check{s},$ because $\W_{d-\nu-1}$ is so. As a consequence the curvature of $\Leg\pref$ is holomorphic along $\check{s}$ by applying~\cite[Proposition~2.6]{MP13}.
\end{proof}

\begin{proof}[\sl Proof of Proposition~\ref{proalph:ell-non-invariante-Fermat}]
The \textsc{Fermat} foliation $\F_{0}^{d-1}$ is given in homogeneous coordinates by the $1$-form
\[
\Omegaoverline_{0}^{d-1}=x^{d-1}(y\mathrm{d}z-z\mathrm{d}y)+y^{d-1}(z\mathrm{d}x-x\mathrm{d}z)+z^{d-1}(x\mathrm{d}y-y\mathrm{d}x).
\]
It has the following $3(d-1)$ invariant lines:
\begin{align*}
&
x=0,
&&
y=0,
&&
z=0,
&&
y=\zeta^k x,
&&
y=\zeta^k z,
&&
x=\zeta^k z,
&&\text{where}\hspace{1mm}k\in\{0,\ldots,d-3\}\hspace{2mm}\text{and}\hspace{2mm}\zeta=\exp(\tfrac{2\mathrm{i}\pi}{d-2}).
\end{align*}
Since the coordinates $x,y$ and $z$ play a symmetric role and since $\ell$ is not invariant by $\F_{0}^{d-1}$, we can assume that $\ell=\{\alpha\,x+\beta\,y-z=0\}$ with $\beta\neq0.$ Then $\overline{\mathcal{O}(\ell\boxtimes\F_{0}^{d-1})}$\label{not:orbite-pref} contains the following homogeneous pre-foliations:
\begin{align*}
\preh_1=\{y-\alpha\,x=0\}\boxtimes\mathcal{H}_{1}^{d-1},&&
\preh_2=\{y-\beta\,x=0\}\boxtimes\mathcal{H}_{1}^{d-1},&&
\preh_3=\big\{x-(\alpha+\beta)y=0\big\}\boxtimes\mathcal{H}_{0}^{d-1}.
\end{align*}
Indeed, $\ell\boxtimes\F_{0}^{d-1}$ is described in the affine chart $z=1$ by $\omega=(\alpha\,x+\beta\,y-1)\omegaoverline_{0}^{d-1}$; putting $\varphi_1=\Big(\frac{x}{y},\frac{\varepsilon}{y}\Big),$ $\varphi_2=\Big(\frac{\varepsilon}{y},\frac{x}{y}\Big)$ and $\varphi_3=\Big(\frac{y+\varepsilon}{x},\frac{y}{x}\Big),$ we obtain that
\begin{small}
\begin{align*}
\lim_{\varepsilon\to 0}\varepsilon^{-1}y^{d+2}\varphi_1^*\omega=(y-\alpha\,x)\omega_{1}^{d-1},
&&\hspace{3mm}
\lim_{\varepsilon\to 0}\varepsilon^{-1}y^{d+2}\varphi_2^*\omega=(\beta\,x-y)\omega_{1}^{d-1},
&&\hspace{3mm}
\lim_{\varepsilon\to 0}\varepsilon^{-1}x^{d+2}\varphi_3^*\omega=\Big((\alpha+\beta)y-x\Big)\omega_{0}^{d-1}.
\end{align*}
\end{small}
\hspace{-2.25mm}The hypothesis that $\Leg(\ell\boxtimes\F_{0}^{d-1})$ is flat implies that the webs $\Leg\preh_i$ $(i=1,2,3)$ are also flat. Let~us~show~that~the~flatness of $\Leg\preh_1$ and $\Leg\preh_2$ implies that, up to linear conjugation,
\begin{align*}
(\alpha,\beta)\in E:=\Big\{(0,\xi),(1,1),(1,\xi),(\xi,\xi)\Big\},
\quad\text{where}\hspace{1mm}\xi=\mathrm{exp}\left(\tfrac{\mathrm{i}\pi}{d-2}\right).
\end{align*}
First of all, the $d$-web $\Leg\preh_1,$ resp. $\Leg\preh_2$, is flat if and only if (\emph{cf.}~proof~of~Lemma~\ref{lem:platitude-Leg-ell-H1})
\begin{align*}
\alpha(\alpha^{d-2}-1)(\alpha^{d-2}+1)=0,
\qquad\qquad\hspace{2mm}
\text{resp.}\hspace{1mm}(\beta^{d-2}-1)(\beta^{d-2}+1)=0,
\end{align*}
{\it i.e.} if and only if $\alpha\in\{0,\zeta^{k},\xi\zeta^{k},\,k=0,\ldots,d-3\},$ resp. $\beta\in\{\zeta^{k},\xi\zeta^{k},\,k=0,\ldots,d-3\}.$
If $\alpha=0$ then $\beta\neq\zeta^k$, because otherwise $\ell$ would be invariant by $\F_{0}^{d-1}.$ It follows that
\begin{align*}
(\alpha,\beta)\in
\Big\{
(0,\xi\zeta^{k}),(\zeta^k,\zeta^{k'}),(\zeta^k,\xi\zeta^{k'}),(\xi\zeta^{k},\zeta^{k'}),(\xi\zeta^{k},\xi\zeta^{k'}),\hspace{1.5mm}k,k'=0,\ldots,d-3
\Big\}.
\end{align*}
If, for $k,k'\in\{0,\ldots,d-3\},$
\begin{align*}
(\alpha,\beta)=(0,\xi\zeta^{k}),
&&\text{resp.}\hspace{1mm}
(\alpha,\beta)\in\Big\{(\zeta^k,\zeta^{k'}),(\zeta^k,\xi\zeta^{k'}),(\xi\zeta^{k},\xi\zeta^{k'})\Big\},
&&
&&\text{resp.}\hspace{1mm}
(\alpha,\beta)=(\xi\zeta^{k},\zeta^{k'}),
\end{align*}
then by conjugating $\omega$ by $\Big(x,\frac{y}{\zeta^{k}}\Big)$, resp.~$\Big(\frac{x}{\zeta^{k}},\frac{y}{\zeta^{k'}}\Big),$ resp.~$\Big(\frac{y}{\zeta^{k}},\frac{x}{\zeta^{k'}}\Big),$ we reduce ourselves to $(\alpha,\beta)=(0,\xi)$, resp.~$(\alpha,\beta)\in\{(1,1),(1,\xi),(\xi,\xi)\},$  resp.~$(\alpha,\beta)=(1,\xi).$ Thus, up to conjugation, $(\alpha,\beta)$ belongs to $E.$

\noindent Moreover, according to Example~\ref{eg:H0}, the flatness of $\Leg\preh_3$ is equivalent to
\[
0=(\alpha+\beta)\Big((\alpha+\beta)^{d-2}-1\Big)\Big((\alpha+\beta)^{d-2}-2(d-2)\Big)=:f_d(\alpha,\beta).
\]
Since
\begin{Small}
\begin{align*}
&
f_d(0,\xi)=2\xi(2d-3)\neq0,&&\quad\quad
f_d(1,1)=4(2^{d-2}-1)(2^{d-3}-d+2)=0\Longleftrightarrow d\in\{3,4\},\\
&
f_d(\xi,\xi)=2\xi(2^{d-2}+1)(2^{d-2}+2d-4)\neq0,&&\quad\quad
f_d(1,\xi)=(\xi+1)\Big((\xi+1)^{d-2}-1\Big)\Big((\xi+1)^{d-2}-2(d-2)\Big)=0\Longleftrightarrow d=3,
\end{align*}
\end{Small}
\hspace{-1mm}we deduce that $d\in\{3,4\}$ and, up to conjugation,
\begin{align*}
&(\alpha,\beta)\in\{(1,1),(1,-1)\}\hspace{1mm}\text{if}\hspace{1mm} d=3
&&\text{and}&&
(\alpha,\beta)=(1,1)\hspace{1mm}\text{if}\hspace{1mm} d=4,
\end{align*}
{\it i.e.}, putting $\ell_1=\{x+y-z=0\}$\, and \,$\ell_2=\{x-y-z=0\}$, we have $\ell\in\{\ell_1,\ell_2\}$ if $d=3$\,\, and \,\,$\ell=\ell_1$ if $d=4.$
\vspace{0.1mm}

\noindent To conclude, it suffices to note that $\ell_1=(s_1s_2s_4)$ and $\ell_2=(s_1s_3)$, where $s_1=[1:0:1]$, $s_2=[0:1:1]$, $s_3=[1:1:0]$ and $s_4=[-1:1:0]$: the points $s_1,s_2$ and $s_3$ are singular for $\F_{0}^{d-1}$, and $s_4\in\Sing\F_{0}^{d-1}$ if~and~only~if $d$ is even; in particular the point $s_4$ is singular for $\F_{0}^{3}$ but not for $\F_{0}^{2}.$
\end{proof}

\begin{proof}[\sl Proof of Proposition~\ref{proalph:ell-non-invariante-Hesse}]
The \textsc{Hesse} foliation $\Hesse$ is described in homogeneous coordinates by the $1$-form
\[
\Omegahesse=yz(2\,x^3-y^3-z^3)\mathrm{d}x+xz(2y^3-x^3-z^3)\mathrm{d}y+xy(2z^3-x^3-y^3)\mathrm{d}z.
\]
Its $12$ invariant lines are given by
\begin{Small}
\begin{align*}
xyz(x+y+z)(\zeta\,x+y+z)(x+\zeta\,y+z)(x+y+\zeta\,z)(\zeta^2x+y+z)(x+\zeta^2y+z)(x+y+\zeta^2z)(\zeta^2x+\zeta\,y+z)(\zeta\,x+\zeta^2y+z)=0,
\end{align*}
\end{Small}
\hspace{-1.4mm}where $\zeta=\exp(\tfrac{2\mathrm{i}\pi}{3}).$ As above, we can assume that $\ell=\{\alpha\,x+\beta\,y-z=0\}$ with $\beta\neq0.$ Then the closure of~the~$\mathrm{Aut}(\pp)$-orbit of $\ell\boxtimes\Hesse$ contains the following three homogeneous pre-foliations:
\begin{align*}
\preh_1=\{y-\alpha\,x=0\}\boxtimes\mathcal{H}_{4}^{4},&&
\preh_2=\{y-\beta\,x=0\}\boxtimes\mathcal{H}_{4}^{4},&&
\preh_3=\big\{a\,x+by=0\big\}\boxtimes\mathcal{H}_{4}^{4},
\end{align*}
where $a=\alpha+\beta-1$ and $b=\alpha+\zeta^2\beta-\zeta.$ Indeed, in the affine chart $z=1$, the pre-foliation~$\ell\boxtimes\Hesse$ is given by $\omega=(\alpha\,x+\beta\,y-1)\omegahesse$; putting $\psi_1=\Big(\frac{x}{y},\frac{\varepsilon}{y}\Big),$ $\psi_2=\Big(\frac{\varepsilon}{y},\frac{x}{y}\Big)$ and $\psi_3=\Big(\frac{x+y}{x+\zeta\,y+\varepsilon},\frac{x+\zeta^2y}{x+\zeta\,y+\varepsilon}\Big),$ a straightforward computation shows that
\begin{small}
\begin{align*}
\lim_{\varepsilon\to 0}\varepsilon^{-1}y^{7}\psi_1^*\omega=(\alpha\,x-y)\omega_{4}^{4},
&&\hspace{3mm}
\lim_{\varepsilon\to 0}\varepsilon^{-1}y^{7}\psi_2^*\omega=(\beta\,x-y)\omega_{4}^{4},
&&\hspace{3mm}
\lim_{\varepsilon\to 0}\varepsilon^{-1}(x+\zeta\,y+\varepsilon)^7\psi_3^*\omega=-9\zeta(a\,x+by)\omega_{4}^{4}.
\end{align*}
\end{small}
\hspace{-1.8mm}Since the $5$-web $\Leg(\ell\boxtimes\Hesse)$ is flat by hypothesis, so are the $5$-webs $\Leg\preh_i,\,i=1,2,3.$ Now, according to~Example~\ref{eg:H4}, for any line $\ell_0$ passing through the origin, $\Leg(\ell_0\boxtimes\mathcal{H}_{4}^{4})$ is flat if and only if $\ell_0=\{x=0\}$ or~$\ell_0=\{y-\rho\,x=0\}$ with $\rho(\rho^3-1)(\rho^3+1)=0,$ {\it i.e.} $\rho\in E:=\{0,\zeta^k,-\zeta^k,\,k=0,1,2\}.$ Therefore, the flatness of $\Leg\preh_1$ (resp.~$\Leg\preh_2$) is equivalent to $\alpha\in E$ (resp. $\beta\in E\setminus\{0\}$ because $\beta\neq0$). Note that $(\alpha,\beta)\neq(-\zeta^k,-\zeta^{k'})$, for otherwise $\ell$ would be invariant by $\Hesse$. As a result
\begin{align*}
(\alpha,\beta)\in\Big\{(0,\zeta^{k}),(0,-\zeta^{k}),(\zeta^k,\zeta^{k'}),(\zeta^k,-\zeta^{k'}),(-\zeta^k,\zeta^{k'}),\hspace{1.5mm}k,k'=0,1,2\Big\}.
\end{align*}
If, for $k,k'\in\{0,1,2\},$
\begin{align*}
(\alpha,\beta)\in\Big\{(0,\zeta^k),(0,-\zeta^k)\Big\},
&&\text{resp.}\hspace{1mm}
(\alpha,\beta)\in\Big\{(\zeta^k,\zeta^{k'}),(\zeta^k,-\zeta^{k'})\Big\},
&&
&&\text{resp.}\hspace{1mm}
(\alpha,\beta)=(-\zeta^k,\zeta^{k'}),
\end{align*}
then by conjugating $\omega$ by $\Big(x,\frac{y}{\zeta^{k}}\Big)$, resp.~$\Big(\frac{x}{\zeta^{k}},\frac{y}{\zeta^{k'}}\Big),$ resp.~$\Big(\frac{y}{\zeta^{k}},\frac{x}{\zeta^{k'}}\Big),$ we reduce ourselves to $(\alpha,\beta)\in\{(0,1),(0,-1)\}$, resp.~$(\alpha,\beta)\in\{(1,1),(1,-1)\},$  resp.~$(\alpha,\beta)=(1,-1).$ It follows that, up to linear conjugation,
\[
(\alpha,\beta)\in F:=\big\{(0,1),(0,-1),(1,1),(1,-1)\big\}.
\]
Now, for $(\alpha,\beta)\in F,$ $\Leg\preh_3$ is flat if and only if $(\alpha,\beta)=(0,1),$ since putting $\tau(\alpha,\beta)=-\frac{a}{b},$ we have
\begin{align*}
\tau(0,1)=0\in E,&&
\tau(0,-1)=2\not\in E,&&
\tau(1,1)=\tfrac{\zeta^2}{2}\not\in E,&&
\tau(1,-1)=\tfrac{1}{2}\not\in E.
\end{align*}
Therefore, up to conjugation, $(\alpha,\beta)=(0,1)$, {\it i.e.}~$\ell=\{y-z=0\}$; then $\ell$ passes through four singular points of~$\Hesse,$ namely the points $s_1=[1:0:0],$ $s_2=[1:1:1]$, $s_3=[\zeta:1:1]$ and $s_4=[\zeta^2:1:1].$
\end{proof}


\section{Pre-foliations of type $(1,3)$ whose associated foliation has only non-degenerate singularities}\label{sec:pre-feuilletages-codegre-1-degre-3}
\bigskip

\noindent In this section, we prove Theorem~\ref{thmalph:Fermat2} stated in the Introduction. To do this, we need two preliminary results, the first of which holds in any degree.

\noindent Let us first recall that in~\S\ref{sec:pre-pre-feuill-convexe-reduit} we have proved Propositions \ref{proalph:ell-non-invariante-Fermat} and \ref{proalph:ell-non-invariante-Hesse} by reducing to the homogeneous case; in~fact this argument is implicitly based on the following proposition.
\begin{pro}\label{pro:pref-degenere-preh}
{\sl Let $\pref=\ell\boxtimes\F$ be a pre-foliation of co-degree $1$ and degree $d\geq2$ on $\pp.$ Assume that the foliation $\F$ has an invariant line $D$ and that all its singularities on $D$ are non-degenerate. There is a~homogeneous pre-foliation $\preh=\ell_0\boxtimes\mathcal{H}$ of co-degree $1$ and degree $d$ on $\pp$ such that:
\vspace{1mm}
\begin{itemize}
\item [(i)] $\preh\in\overline{\mathcal{O}(\pref)}$ and $\mathcal{H}\in\overline{\mathcal{O}(\F)}$;
\vspace{0.5mm}
\item [(ii)] if $\ell=D$ (resp. $\ell\neq D$), then $\ell_0=L_\infty$ (resp. $\ell_0\neq L_\infty)$;
\vspace{0.5mm}
\item [(iii)] $D$ is invariant by $\mathcal{H}$;
\vspace{0.5mm}
\item [(iv)] $\Sing\mathcal{H}\cap D=\Sing\F\cap D$;
\vspace{0.5mm}
\item [(v)] $\forall\hspace{0.5mm}s\in\Sing\mathcal{H}\cap D,\hspace{1mm}
                   \mu(\mathcal{H},s)=1\label{not:nombre-milnor}
                   \hspace{2mm}\text{and}\hspace{2mm}
                   \mathrm{CS}(\mathcal{H},D,s)=\mathrm{CS}(\F,D,s)$.
\end{itemize}
\vspace{1mm}
If, moreover, $\Leg\pref$ (resp. $\Leg\F$) is flat, then $\Leg\preh$ (resp. $\Leg\mathcal{H}$) is also flat.
}
\end{pro}

\noindent This proposition is an analogue for co-degree one pre-foliations of Proposition~6.4 of~\cite{BM18Bull} on foliations of $\pp.$

\begin{proof}[\sl Proof]
Choose a homogeneous coordinate system $[x:y:z]\in\pp$ such that $D=L_\infty=(z=0)$. Since $D$ is~$\F$-invariant, $\F$ is defined in the affine chart $z=1$ by a $1$-form $\omega$ of type $$\omega=\sum_{i=0}^{d-1}(A_i(x,y)\mathrm{d}x+B_i(x,y)\mathrm{d}y),$$ where $A_i,\,B_i$ are homogeneous polynomials of degree $i.$ According to~\cite[Proposition 6.4]{BM18Bull}, since all the singularities of $\F$ on $D$ are non-degenerate, the $1$-form $\omega_{d-1}=A_{d-1}(x,y)\mathrm{d}x+B_{d-1}(x,y)\mathrm{d}y$ defines a homogeneous foliation~$\mathcal{H}$ of degree $d-1$ on $\pp$ belonging to $\overline{\mathcal{O}(\F)}$ and satisfying the stated properties (iii), (iv) and~(v).

\noindent Now, write $\ell=\{\alpha\,x+\beta y+\gamma\,z=0\}$; in homogeneous coordinates, $\pref,$ resp. $\mathcal{H},$ is given by
\begin{align*}
&\Omega_{d+1}=(\alpha\,x+\beta y+\gamma\,z)\sum_{i=0}^{d-1}z^{d-i-1}\Big(A_i(x,y)(z\mathrm{d}x-x\mathrm{d}z)+B_i(x,y)(z\mathrm{d}y-y\mathrm{d}z)\Big),&&\\
\text{resp}.\hspace{1.5mm}
&\Omega_{d}=A_{d-1}(x,y)(z\mathrm{d}x-x\mathrm{d}z)+B_{d-1}(x,y)(z\mathrm{d}y-y\mathrm{d}z).
\end{align*}
Putting $\varphi=\varphi_{\varepsilon}=\left[\frac{x}{\varepsilon}:\frac{y}{\varepsilon}:z\right],$ we see that if $(\alpha,\beta)=(0,0),$ resp. $(\alpha,\beta)\neq(0,0),$ then
\begin{align*}
&\hspace{0.5cm}\lim_{\varepsilon\to 0}\varepsilon^{d}\gamma^{-1}\varphi^*\Omega_{d+1}=z\hspace{0.3mm}\Omega_{d},
&&\text{resp.}\hspace{1mm}
\lim_{\varepsilon\to 0}\varepsilon^{d+1}\varphi^*\Omega_{d+1}=(\alpha\,x+\beta y)\Omega_{d}.
\end{align*}
It follows that the closure of the $\mathrm{Aut}(\pp)$-orbit of $\pref$ contains the homogeneous pre-foliation $\preh=\ell_0\boxtimes\mathcal{H},$ where $\ell_0=L_\infty$ if $\ell=D$ and $\ell_0=\{\alpha\,x+\beta y=0\}\neq L_\infty$ if $\ell\neq D.$
\end{proof}

\noindent The following technical lemma is an analogue for pre-foliations of type $(1,3)$ of Lemma~6.7 of \cite{BM18Bull} on foliations of degree $3.$ It plays a key role in the proof of Theorem~\ref{thmalph:Fermat2}.
\begin{lem}\label{lem:2-droites-invariantes}
{\sl
Let $\pref=\ell\boxtimes\F$ be a pre-foliation of co-degree $1$ and degree $3$ on $\pp.$ Assume that the $3$-web $\Leg\pref$ is flat and that the foliation $\F$ has a non-degenerate singularity $m$ satisfying $\mathrm{BB}(\F,m)\neq4.$\label{not:invariant-BB} Then, through the point $m$ pass exactly two $\F$-invariant lines.
}
\end{lem}

\begin{proof}[\sl Proof]
The hypotheses $\mu(\F,m)=1$ and $\mathrm{BB}(\F,m)\neq4$ ensure the existence of an affine chart~$(x,y)$ of $\pp$ in which
$m=(0,0)$ and $\F$ is defined by a $1$-form $\omega_0$ of type $\omega_0=\omega_{0,1}+\omega_{0,2}+\omega_{0,3},$ where
\begin{small}
\begin{align*}
\omega_{0,1}=\lambda\hspace{0.1mm}y\mathrm{d}x+\mu\hspace{0.1mm}x\mathrm{d}y,&&
\omega_{0,2}=\left(\sum_{i=0}^{2}a_{i}\hspace{0.1mm}x^{2-i}y^{i}\right)\mathrm{d}x+\left(\sum_{i=0}^{2}b_{i}\hspace{0.1mm}x^{2-i}y^{i}\right)\mathrm{d}y,&&
\omega_{0,3}=\left(\sum_{i=0}^{2}c_{i}\hspace{0.1mm}x^{2-i}y^{i}\right)(x\mathrm{d}y-y\mathrm{d}x),
\end{align*}
\end{small}
\hspace{-1mm}with $\lambda\mu(\lambda+\mu)\neq0.$

\noindent The only lines passing through $m$ and which can be invariant by $\F$ are $(x=0)$ and $(y=0).$ Indeed, denote by~$\mathrm{R}=x\frac{\partial{}}{\partial{x}}+y\frac{\partial{}}{\partial{y}}$ the radial vector field centered at $m$; if $L=(ux+vy=0)$ is $\F$-invariant, then $ux+vy$ must divide the tangent cone $\mathrm{C}_{\omega_{0,1}}:=\omega_{0,1}(\mathrm{R})=(\lambda+\mu)xy,$ so that $u=0$ or $v=0.$

\noindent We will show that indeed $(x=0)$ and $(y=0)$ are invariant by $\F,$ which will establish the lemma. We have to prove that $a_0=b_2=0,$ since the invariance by $\F$ of $(x=0)$, resp. $(y=0),$ is equivalent to the vanishing of $b_2$, resp. $a_0.$

\noindent If $\ell=\{\alpha\,x+\beta\,y+\gamma=0\}$ then, in the affine chart $(p,q)$ of $\pd$ corresponding to the line $\{y=px-q\}\subset{\mathbb{P}^{2}_{\mathbb{C}}},$ the $3$-web $\Leg\pref$ is described by the symmetric $3$-form
\begin{Small}
\begin{align*}
&\hspace{0.5cm}\check{\omega}=\big((\gamma-\beta\,q)\mathrm{d}p+(\alpha+\beta\,p)\mathrm{d}q\big)\check{\omega}_{0},
\\
&\hspace{-0.76cm}{\fontsize{11}{11pt}\text{where}}
\\
&\hspace{0.5cm}\check{\omega}_{0}=\mu\,p\mathrm{d}p\mathrm{d}q
+(a_0+b_0p+c_0q)\mathrm{d}q^2
+\Big(\lambda\,\mathrm{d}p+(a_1+b_1p+c_1q)\mathrm{d}q\Big)(p\mathrm{d}q-q\mathrm{d}p)
+\big(a_2+b_2p+c_2q\big)\big(p\mathrm{d}q-q\mathrm{d}p\big)^2.
\end{align*}
\end{Small}
\hspace{-1.74mm}Assume by contradiction that $a_0\neq0.$ Consider the family of automorphisms $\varphi=\varphi_{\varepsilon}=(a_0\hspace{0.1mm}\varepsilon\hspace{0.1mm}p,\hspace{0.1mm}a_0\hspace{0.1mm}\varepsilon^{2}\hspace{0.1mm}q).$ We~see~that~if $\gamma\neq0$,\, resp. $\gamma=0$ and $\alpha\neq0,$\, resp. $\gamma=\alpha=0,$ then
\begin{align*}
&\hspace{0.5cm}\lim_{\varepsilon\to 0}\varepsilon^{-5}\gamma^{-1}a_{0}^{-4}\varphi^*\check{\omega}=\theta_1\eta,
&&\text{resp.}\hspace{1mm}
\lim_{\varepsilon\to 0}\varepsilon^{-6}\alpha^{-1}a_{0}^{-4}\varphi^*\check{\omega}=\theta_2\eta,
&&\text{resp.}\hspace{1mm}
\lim_{\varepsilon\to 0}\varepsilon^{-7}\beta^{-1}a_{0}^{-5}\varphi^*\check{\omega}=\theta_3\eta,
\end{align*}
where
\begin{align*}
&\hspace{-0.14cm}\theta_1=\mathrm{d}p,
&&
\theta_2=\mathrm{d}q,
&&
\theta_3=p\mathrm{d}q-q\mathrm{d}p,
&&
\eta=-\lambda\,q\diffp^2+(\lambda+\mu)p\diffp\diffq+\diffq^2.
\end{align*}
For $i=1,2,3,$ put $\W_{3}^{(i)}=\F_i\boxtimes\W_2,$ where $\W_2$ (resp. $\F_i$) denotes the $2$-web (resp. the foliation) defined by $\eta$ (resp. by $\theta_i$). It follows that if $\gamma\neq0$,\, resp. $\gamma=0$ and $\alpha\neq0,$\, resp. $\gamma=\alpha=0,$ then the closure of the $\mathrm{Aut}(\pd)$-orbit of $\Leg\pref$ contains the $3$-web $\W_{3}^{(1)},$ resp. $\W_{3}^{(2)},$ resp. $\W_{3}^{(3)}.$ Now, since $\Leg\pref$ is flat by hypothesis, every $3$-web belonging to $\overline{\mathcal{O}(\Leg\pref)}$ is also flat. Therefore, to obtain a contradiction, it~suffices to show that for every~$i=1,2,3,$ $\W_{3}^{(i)}$ is not flat. Since $\Delta(\eta)=f(p,q):=4\lambda\,q+(\lambda+\mu)^2p^2$, it suffices again to show that for every $i=1,2,3,$ the curvature of $\W_{3}^{(i)}$ is not holomorphic along the component $\mathcal{C}=\{f(p,q)=0\}\subset\Delta(\W_2),$ which is a parabola, because $\lambda(\lambda+\mu)\neq0.$
\medskip

\noindent First of all, let us note that $\mathcal{C}$ is not invariant by $\W_2$, since putting $\eta_0=(\lambda+\mu)p\mathrm{d}p+2\mathrm{d}q,$ we have
\begin{align*}
&\hspace{-0.4cm}\eta\big|_{\mathcal{C}}=\left(\frac{\eta_0}{2}\right)^2
\qquad\qquad\text{and}\qquad\qquad
\eta_0\wedge\mathrm{d}f=-4\mu(\lambda+\mu)p\diffp\wedge\diffq\not\equiv0.
\end{align*}

\noindent Let us consider the case where $i\in\{1,2\}.$ Since
\begin{align*}
&\hspace{1.27cm}\eta_0\wedge\theta_1\Big|_{\mathcal{C}}=-2\diffp\wedge\diffq\not\equiv0
\qquad\qquad\text{and}\qquad\qquad
\eta_0\wedge\theta_2\Big|_{\mathcal{C}}=(\lambda+\mu)p\diffp\wedge\diffq\not\equiv0,
\end{align*}
we have $\mathcal{C}\not\subset\Tang(\W_2,\F_i).$ Therefore, according to~\cite[Theorem~1]{MP13} (\emph{cf.} \cite[Theorem~1.1]{BFM14}), the curvature $K(\W_{3}^{(i)})$ is holomorphic on $\mathcal{C}$ if and only if $\mathcal{C}$ is invariant by $\F_i,$ which is impossible, because each $\F_i$ is a~pencil of lines and hence cannot admit a parabola as an invariant curve.
\medskip

\noindent Let us now examine the case where $i=3.$ In this case $\mathcal{C}\subset\Tang(\W_2,\F_3)$ if and only if $\lambda=\mu,$ because
\begin{align*}
\eta_0\wedge\theta_3\Big|_{\mathcal{C}}=\frac{1}{2\lambda}(\lambda-\mu)(\lambda+\mu)p^2\diffp\wedge\diffq\equiv0\Longleftrightarrow\lambda=\mu.
\end{align*}
If $\lambda\neq\mu$, then, as above, we can apply Theorem~1~of~\cite{MP13} and assert that $K(\W_{3}^{(3)})$ cannot be holomorphic~on~$\mathcal{C}.$
\medskip

\noindent We therefore assume that $\lambda=\mu$ and prove that $K(\W_{3}^{(3)})\not\equiv0.$ The pull-back of $\W_{3}^{(3)}$ by the rational map $\psi(p,q)=\left(p,\mu(q^2-p^2)\right)$ writes as $\psi^*\W_{3}^{(3)}=\F_{3}^{(1)}\boxtimes\F_{3}^{(2)}\boxtimes\F_{3}^{(3)}$, where
\begin{align*}
&
\F_{3}^{(1)}\hspace{0.1mm}:\hspace{0.1mm}(p^2+q^2)\diffp-2pq\diffq=0,&&
\F_{3}^{(2)}\hspace{0.1mm}:\hspace{0.1mm}(p+q)\diffp-2q\diffq=0,&&
\F_{3}^{(3)}\hspace{0.1mm}:\hspace{0.1mm}(p-q)\diffp-2q\diffq=0.
\end{align*}

\noindent Using formula~(\ref{equa:eta-rst}), a direct computation leads to
\begin{align*}
&\hspace{-1cm}
\eta(\psi^{*}\W_{3}^{(3)})=-\frac{p\mathrm{d}p}{q^2}+\frac{4\mathrm{d}q}{q}+\frac{\mathrm{d}(p^2-q^2)}{p^2-q^2},
\end{align*}
so that
\[
K(\psi^{*}\W_{3}^{(3)})=\mathrm{d}\eta(\psi^{*}\W_{3}^{(3)})=-\frac{2p}{q^3}\mathrm{d}p\wedge\mathrm{d}q\not\equiv0,
\]
hence $\psi^{*}K(\W_{3}^{(3)})=K(\psi^{*}\W_{3}^{(3)})\not\equiv0$ and therefore $K(\W_{3}^{(3)})\not\equiv0.$
\vspace{2mm}

\noindent We have thus shown that $a_0=0$, which means that the line $(y=0)$ is invariant by $\F$. Exchanging~the~roles~of the coordinates $x$ and $y$, the same argument shows that $b_2=0$, {\it i.e.} that the line $(x=0)$ is also invariant~by~$\F.$
\end{proof}

\noindent Before starting the proof of Theorem~\ref{thmalph:Fermat2}, let us recall (\emph{cf.} \cite{Bru15}) that if $\F$ is a foliation of degree $d$ on $\pp$ then
\begin{align}\label{equa:Darboux-BB}
&\sum_{s\in\mathrm{Sing}\F}\mu(\F,s)=d^2+d+1
&&\text{and}&&
\sum_{s\in\mathrm{Sing}\F}\mathrm{BB}(\F,s)=(d+2)^2.
\end{align}

\vspace{2mm}

\begin{proof}[\sl Proof of Theorem~\ref{thmalph:Fermat2}]
Write $\Sing\F=\Sigma^{1}\cup\Sigma^{2}$, where
\begin{align*}
\Sigma^{1}=\{s\in\Sing\F\hspace{1mm}\colon \mathrm{BB}(\F,s)=4\}
&&\text{and}&&
\Sigma^{2}=\Sing\F\setminus\Sigma^{1}.
\end{align*}
For $i=1,2,$ put $\kappa_i=\#\hspace{0.5mm}\Sigma^{i}$. By hypothesis, we have $\deg\F=2$ and, for any $s\in\Sing\F$, $\mu(\F,s)=1$. Formulas (\ref{equa:Darboux-BB}) then give
\begin{align}\label{equa:Dar-BB-2}
&\#\hspace{0.5mm}\Sing\F=\kappa_1+\kappa_2=7 &&\text{and} &&4\kappa_1+\sum_{s\in\Sigma^{2}}\mathrm{BB}(\F,s)=16.
\end{align}
It follows that $\Sigma^{2}$ is non-empty. Let $m$ be a point of $\Sigma^{2}$; by Lemma~\ref{lem:2-droites-invariantes} through the point $m$ pass exactly two $\F$-invariant lines $ D_{m}^{(1)}$ and $ D_{m}^{(2)}$. Then, for $i=1,2$, Proposition~\ref{pro:pref-degenere-preh} ensures the existence of a homogeneous pre-foliation $\preh^{(i)}_{m}=\ell_{m}^{(i)}\boxtimes\mathcal{H}^{(i)}_{m}$ of type $(1,3)$ on $\pp$ belonging to $\overline{\mathcal{O}(\pref)}$ and such that the line $D_{m}^{(i)}$ is $\mathcal{H}^{(i)}_{m}$-invariant. Since $\Leg\pref$ is flat by hypothesis, so are~$\Leg\preh^{(1)}_{m}$ and $\Leg\preh^{(2)}_{m}.$ Therefore, $\preh^{(i)}_{m}$ ($i=1,2$) is linearly conjugate to one of the eight models of Corollary~\ref{cor:class-pre-homogenes-plats-co-degre-1-degre-3}. For~$i=1,2,$ Proposition~\ref{pro:pref-degenere-preh} also ensures that
\vspace{1mm}
\begin{itemize}
\item [($\mathfrak{a}$)] if $\ell\neq D_{m}^{(i)}$, then $\ell_{m}^{(i)}\neq L_\infty$;
\vspace{0.5mm}

\item [($\mathfrak{b}$)] $\Sing\F\cap D_{m}^{(i)}=\Sing\mathcal{H}^{(i)}_{m}\cap D_{m}^{(i)}$;
\vspace{0.5mm}

\item [($\mathfrak{c}$)] $\forall\hspace{0.5mm}s\in\Sing\mathcal{H}^{(i)}_{m}\cap D_{m}^{(i)},\quad
                   \mu(\mathcal{H}^{(i)}_{m},s)=1\hspace{2mm}\text{and}\quad
                   \mathrm{CS}(\mathcal{H}^{(i)}_{m}, D_{m}^{(i)},s)=\mathrm{CS}(\F, D_{m}^{(i)},s)$.
\end{itemize}
\vspace{1mm}

\noindent Since $\mathrm{CS}(\F, D_{m}^{(1)},m)\mathrm{CS}(\F, D_{m}^{(2)},m)=1,$ we have
\begin{align}\label{equa:produit-CS-1}
\mathrm{CS}(\mathcal{H}^{(1)}_{m}, D_{m}^{(1)},m)\mathrm{CS}(\mathcal{H}^{(2)}_{m}, D_{m}^{(2)},m)=1.
\end{align}

\noindent Let us first assume that $\ell\neq D_{m}^{(i)}$ for $i=1,2$; this is obviously the case if $\ell$ is not invariant by $\F.$ Then, by~($\mathfrak{a}$), we have $\ell_{m}^{(i)}\neq L_\infty$ for $i=1,2.$ Therefore, each of the $\preh^{(i)}_{m}$ is conjugate to one of the five pre-foliations $\preh_{j}^{3},j=2,\ldots,6,$ so that each of the $\mathcal{H}^{(i)}_{m}$ is conjugate to one of the three foliations $\mathcal{H}_{1}^{2}$, $\mathcal{H}_{2}^{2}(0,0),$ $\mathcal{H}_{3}^{2}(-2)$ (Corollary~\ref{cor:class-pre-homogenes-plats-co-degre-1-degre-3}). Consulting Table~\ref{tab:CS(lambda)} and using equality~(\ref{equa:produit-CS-1}) as well as relations ($\mathfrak{b}$) and ($\mathfrak{c}$), we see that
\begin{Small}
\begin{align}\label{equa:CS-moins-1-Sigma-i-cap-D-m-i}
\mathrm{CS}(\mathcal{H}^{(1)}_{m}, D_{m}^{(1)},m)=\mathrm{CS}(\mathcal{H}^{(2)}_{m}, D_{m}^{(2)},m)=-1,&&\hspace{2mm}
\#\hspace{0.5mm}(\Sigma^{1}\cap D_{m}^{(1)})=\#\hspace{0.5mm}(\Sigma^{1}\cap D_{m}^{(2)})=2,&&\hspace{2mm}
\Sigma^{2}\cap D_{m}^{(1)}=\Sigma^{2}\cap D_{m}^{(2)}=\{m\}.
\end{align}
\end{Small}
\hspace{-1.2mm}Let us now assume that the line $\ell$ is equal to one of the lines $D_{m}^{(i)}$, say $\ell=D_{m}^{(2)}.$ Let us show that equalities~(\ref{equa:CS-moins-1-Sigma-i-cap-D-m-i}) still hold. Since $\ell\neq D_{m}^{(1)}$, $\mathcal{H}^{(1)}_{m}$ is conjugate to one of the foliations $\mathcal{H}_{1}^{2}$, $\mathcal{H}_{2}^{2}(0,0),$ $\mathcal{H}_{3}^{2}(-2).$ Moreover $\Sigma^2\cap D_{m}^{(1)}=\{m\}$; indeed, if $\Sigma^2\cap D_{m}^{(1)}$ contained another point $m'\neq m,$ we would have $\ell\neq D_{m'}^{(i)}$ for $i=1,2,$ so~that~$\{m'\}=\Sigma^{2}\cap D_{m'}^{(i)}=\Sigma^{2}\cap D_{m}^{(1)}\supset\{m,m'\}$, which is impossible. From Table~\ref{tab:CS(lambda)}, we deduce that
\begin{align*}
&\mathrm{CS}(\mathcal{H}^{(1)}_{m}, D_{m}^{(1)},m)=\mathrm{CS}(\mathcal{H}^{(2)}_{m},\ell,m)=-1
\qquad\qquad\text{and}\qquad\qquad
\#\hspace{0.5mm}(\Sigma^{1}\cap D_{m}^{(1)})=2,
\end{align*}
hence
\begin{align*}
&\hspace{-6.26cm}\mathrm{CS}(\F, D_{m}^{(1)},m)=\mathrm{CS}(\F,\ell,m)=-1.
\end{align*}
Since these equalities are valid for any choice of $m\in\Sigma^2\cap\ell$ and since every line of $\pp$ cannot contain more than $\deg\F+1=3$ singular points of $\F,$ the \textsc{Camacho}-\textsc{Sad} formula (\emph{see} \cite{CS82}) $\sum_{s\in\Sing\F\cap \ell}\mathrm{CS}(\F,\ell,s)=1$ implies that
\begin{align*}
\#\hspace{0.5mm}(\Sigma^{1}\cap\ell)=2
\qquad\qquad\text{and}\qquad\qquad
\Sigma^{2}\cap\ell=\{m\}.
\end{align*}

\noindent Equalities~(\ref{equa:CS-moins-1-Sigma-i-cap-D-m-i}) are thus established in all cases. It follows in particular that $\mathrm{BB}(\F,m)=0.$ The point~$m\in\Sigma^{2}$ being arbitrary, $\Sigma^{2}$ consists of $s\in\Sing\F$ such that $\mathrm{BB}(\F,s)=0$. System (\ref{equa:Dar-BB-2}) then rewrites as  $\kappa_1+\kappa_2=7$\, and \,$4\kappa_1=16,$ whose unique solution is $(\kappa_1,\kappa_2)=(4,3)$, that is \hspace{1mm} $\Sing\F=\Sigma^{1}\cup\Sigma^{2},\hspace{2mm}\#\hspace{0.5mm}\Sigma^{1}=4$\hspace{2mm} and \hspace{2mm}$\#\hspace{0.5mm}\Sigma^{2}=3.$ Since $\Sigma^2\cap (D_{m}^{(1)}\cup D_{m}^{(2)})=\{m\}$, $\F$ has $3\cdot2=6$ invariant lines, which means that $\F$ is~reduced~convex. It~then~follows~from the classification of convex foliations of degree two (\emph{cf.} \cite[Proposition~7.4]{FP15} or \cite[Theorem~A]{BM20Bull}) that $\F$ is linearly conjugate to the \textsc{Fermat} foliation $\F_{0}^{2}.$ We conclude by noting that if the~line~$\ell$~is~not~invariant by $\F$, the flatness of $\Leg\pref$ and Proposition~\ref{proalph:ell-non-invariante-Fermat} imply that $\ell$ must join two non-radial singularities of~$\F.$
\end{proof}

\section*{Index of notations}

\bigskip
\bigskip

\newlength{\largeur}
\setlength{\largeur}{12.5cm}
\addtolength{\largeur}{0cm}
\hspace{-0.5cm}
\begin{tabular}{p{1.5cm}p{\largeur}}
\noteA{\Leg\pref}{\textsc{Legendre} transform of the pre-foliation $\pref$}{not:Leg-pref}

\noteA{\Delta(\W)}{discriminant of the web $\W$}{not:Delta-W}

\noteA{\mathcal{T}_{\mathcal{H}}}{type of the homogeneous foliation $\mathcal{H}$}{not:Type-H}

\noteA{\eta(\W)}{fundamental form of the web $\W$}{not:eta-W}

\noteA{K(\W)}{curvature of the web $\W$}{not:K-W}

\noteA{\mathcal{G}_{\F}}{\textsc{Gauss} map associated to the foliation $\F$}{not:Gauss-F}

\noteA{\mathrm{Sing}\mathcal{F}}{singular locus of the foliation $\mathcal{F}$}{not:Sing-F}

\noteA{\ItrF}{transverse part of the inflection divisor $\IF$}{not:Inflex-Transverse-F}

\noteA{\mathrm{rk}(\W)}{rank of the web $\W$}{not:rk-W}

\noteA{S_\W}{characteristic surface of the web $\W$}{not:S-W}

\noteA{\Cinv}{tangent cone at the origin of the homogeneous foliation $\mathcal{H}$}{not:C-H}

\noteA{\IF}{inflection divisor of the foliation $\F$}{not:Inflex-F}

\noteA{\IinvF}{invariant part of the inflection divisor $\IF$}{not:Inflex-Invariant-F}

\noteA{\omega_{\hspace{0.2mm}0}^{\hspace{0.2mm}d-1}}{$(d-2)y^{d-1}\mathrm{d}x+x\left(x^{d-2}-(d-1)y^{d-2}\right)\mathrm{d}y$}{not:omega0}

\noteA{\omega_{\hspace{0.2mm}4}^{\hspace{0.2mm}d-1}}{$y(\sigma_d\,x^{d-2}-y^{d-2})\diffx+x(\sigma_d\,y^{d-2}-x^{d-2})\diffy$}{not:omega4}

\noteA{\omega_{1}^{d-1}}{$y^{d-1}\mathrm{d}x-x^{d-1}\mathrm{d}y$}{not:omega1}

\noteA{\omega_{2}^{d-1}(\lambda,\mu)}{$(x^{d-1}+\lambda\,y^{d-1})\mathrm{d}x+(\mu\,x^{d-1}-y^{d-1})\mathrm{d}y$}{not:omega2}

\noteA{\omega_{3}^{d-1}(\lambda)}{$(x^{d-1}+\lambda y^{d-1})\mathrm{d}x+x^{d-1}\mathrm{d}y$}{not:omega3}

\noteA{\mathrm{CS}(\mathcal{F},\mathcal{C},s)}{\textsc{Camacho-Sad} index of the foliation $\mathcal{F}$ at the point $s$ along the curve $\mathcal{C}$}{not:indice-CS}

\noteA{\tau(\mathcal{F},s)}{tangency order of the foliation $\F$ with a generic line passing through the point $s$}{not:tau-F-s}

\noteA{\mathcal{O}(\pref)}{orbit of the pre-foliation $\pref$ under the action of $\mathrm{Aut}(\pp)$}{not:orbite-pref}

\noteA{\mu(\mathcal{F},s)}{\textsc{Milnor} number of the foliation $\mathcal{F}$ at the singular point $s$}{not:nombre-milnor}

\noteA{\mathrm{BB}(\mathcal{F},s)}{\textsc{Baum}-\textsc{Bott} invariant of the foliation $\mathcal{F}$ at the singular point $s$}{not:invariant-BB}
\end{tabular}


\end{document}